\documentclass[reqno,9pt]{amsart}

\usepackage[top=2.0cm,bottom=2.0cm,left=3cm,right=3cm]{geometry}
\usepackage{amsthm,amsmath,amssymb,dsfont}
\usepackage{mathrsfs,amsfonts,functan,extarrows,mathtools}
\usepackage[colorlinks]{hyperref}

\usepackage{marginnote}
\usepackage{xcolor}
\usepackage{stmaryrd}
\usepackage{esint}
\usepackage{graphicx}
\usepackage{bm}

\newtheorem{theorem}{Theorem}[section]

\newtheorem{remark}{Remark}[section]
\newtheorem{lemma}{Lemma}[section]

\numberwithin{equation}{section}
\allowdisplaybreaks

\def\d{\mathrm{d}}
\def\no{\nonumber}
\def\R{\mathbb{R}}

\def\l{\langle}
\def\r{\rangle}

\newcounter{wronumber}

\begin{document}
\title[From VPB to incompressible NSFP]
			{Convergence from two-species Vlasov-Poisson-Boltzmann system to two-fluid incompressible Navier-Stokes-Fourier-Poisson system with Ohm's law}

\author[Zhendong Fang ]{Zhendong Fang}
\address[Zhendong Fang]
        {\newline School of Mathematics and Statistics, Wuhan University, Wuhan, 430072, P. R. China}
\email{fangzhendong@whu.edu.cn}
\author[Ning Jiang]{Ning Jiang}
\address[Ning Jiang]{\newline School of Mathematics and Statistics, Wuhan University, Wuhan, 430072, P. R. China}
\email{njiang@whu.edu.cn}

\thanks{2020}

\maketitle

\begin{abstract}
	In this paper, we justify the convergence from the two-species Vlasov-Poisson-Boltzmann (in briefly,VPB) system to the  two-fluid incompressible Navier-Stokes-Fourier-Poisson (in briefly, NSFP) system with Ohm's law in the context of classical solutions. We prove the uniform estimates with respect to the Knudsen number $\varepsilon$ for the solutions to the two-species VPB system near equilibrium by treating the strong interspecies interactions. Consequently,  we prove the convergence to the  two-fluid incompressible NSFP as $\varepsilon$ go to 0. \\
	
	
\end{abstract}





\section{Introduction}

\subsection{Two-species Vlasov-Poisson-Boltzmann system}
The two-species VPB system (see pages 3 to 4 of Chapter 1 in \cite{Arsenio-SaintRaymond-2016})
\begin{equation}\label{two-species VPB}
\begin{cases}
&\partial_{t}f^{+}+v\cdot\nabla_{x}f^{+}+\tfrac{q}{m}\nabla_{x}\phi
\cdot\nabla_{v}f^{+}=B(f^{+}, f^{+})+B(f^{+}, f^{-}),\\
&\partial_{t}f^{-}+v\cdot\nabla_{x}f^{+}-\tfrac{q}{m}\nabla_{x}\phi
\cdot\nabla_{v}f^{-}=B(f^{-}, f^{-})+B(f^{-}, f^{+}),\\
&\Delta_{x}\phi=\tfrac{q}{\epsilon_{0}}\int_{\R^{3}}(f^{+}-f^{-})dv.
\end{cases}
\end{equation}
describes the evolution of a gas of two species of oppositely charged and same mass particles (cations of $q>0$ and mass $m>0$, and anions of charge $-q<0$ and $m>0$) under the influence of the interactions with themselves through collisions and their self-consistent electrostatic field. The first equation in \eqref{two-species VPB} is the equation of Vlasov-Boltzmann for cations, the second one is the equation of Vlasov-Boltzmann for anions and the last is the Gauss' law. Here, the particle number densities $f^{\pm}(t,x,v)\geq 0$ represent the distributions of the positively charged ions (i.e. cations) and the negatively charged ions (i.e., anions) respectively, at time $t\geq 0$, position $x=(x_{1}, x_{2}, x_{3})\in \mathbb{R}^{3}$ and velocity $v=(v_{1}, v_{2}, v_{3})\in\mathbb{R}^{3}$. The physical constants $\epsilon_{0}>0$ is the vacuum permittivity (or electric constant) and $\nabla_{x}\phi$ describes the electric field. The collision operators $B(f^{+},f^{-})$ and $B(f^{-},f^{+})$ have been added to the right-hand sides of the respective Vlasov-Boltzmann equations in $(\ref{two-species VPB})$ in order to account for the variations in the densities $f^{+}\geq0$
and $f^{-}\geq0$ due to interspecies collisions. The self-consistent electric potential $\phi=\phi(t,x)\in\mathbb{R}$ is coupled with the function $f^{+}-f^{-}$ through the Poisson equation. The bilinear function $B$ with hard-sphere interaction is defined by
\begin{equation}
B(f,g)(t,x,v)=\int_{\mathbb{R}^{3}}\int_{\mathbb{S}^{2}}(f'g'_{*}-fg_{*})|(v-v_{*})\cdot \omega|d\omega d v_{*},
\end{equation}
where
\begin{equation}\no
\begin{aligned}
f&=f(t,x,v),f'=f(t,x,v'),f_{*}=f(t,x,v_{*}),f'_{*}=f(t,x,v'_{*}),\\
v'&=v-[(v-v_{*})\cdot\omega]\omega,v'_{*}=v_{*}+[(v-v_{*})\cdot\omega]\omega,\omega\in\mathbb{S}^{2}.
\end{aligned}
\end{equation}

We will consider the data $f^{\pm}(t,x,v)$ which are fluctuations $g^{\pm}$ of order Mach number $Ma$
\begin{equation}\no
f^{\pm}(t,x,v)=M(1+Ma\,\, g^{\pm}(t,x,v)),
\end{equation}
around a global normalized Maxwellian equilibrium $M(v)=\tfrac{1}{(2\pi)^{\frac{3}{2}}}e^{-\frac{|v|^{2}}{2}}$ and take the dimensionless number $Kn$, $St$ and $Ma$ all as $\varepsilon$ to obtain in the fast relaxation limit. After nondimensionalization, the scaled two-species VPB system (see Section 2.1, Section 2.2 and Section 2.4.7 of Chapter 2 in \cite{Arsenio-SaintRaymond-2016}) is in the following form:
\begin{equation}\label{scaled VPB}
\begin{cases}
\varepsilon\partial_{t}f^{\pm}_{\varepsilon}+v\cdot\nabla_{x}f_{\varepsilon}^{\pm}\pm\alpha\nabla_{x}\phi_{\varepsilon}\cdot\nabla_{v}f^{\pm}_{\varepsilon}=\tfrac{1}{\varepsilon}B(f_{\varepsilon}^{\pm},f_{\varepsilon}^{\pm})+\tfrac{\delta^{2}}{\varepsilon}B(f_{\varepsilon}^{\pm},f_{\varepsilon}^{\mp}),\\
f_{\varepsilon}^{\pm}=M(1+\varepsilon g_{\varepsilon}^{\pm}),\\
\Delta_{x}\phi_{\varepsilon}=\tfrac{\alpha}{\varepsilon}\int_{\R^{3}}(g_{\varepsilon}^{+}-g_{\varepsilon}^{-})Mdv,
\end{cases}
\end{equation}
where $\alpha>0$ measures the electric repulsion according to Gauss' law, and $\delta>0$ represents the strength of interactions. The size of the bounded parameter $\delta$ will be compared to the Knudsen number $Kn=\varepsilon$ and divided into three cases:

$\bullet$ $\delta\sim$ 1, strong interspecies interactions;

$\bullet$ $\delta=o(1)$ and $\tfrac{\delta}{\varepsilon}$ unbounded, weak interspecies interactions;

$\bullet$ $\delta=O(\varepsilon)$, very weak interspecies interactions.

In this paper, we consider the strong interspecies interactions, which is the most singular case. We also suppose $\alpha=\varepsilon$ as in \cite{Arsenio-SaintRaymond-2016}. Therefore, we have the scaled two-species VPB system as follow
\begin{equation}\label{VPB-g}
\begin{cases}
&\varepsilon\partial_{t}f_{\varepsilon}+v\cdot\nabla_{x}f_{\varepsilon}-\varepsilon g_{\varepsilon}v\cdot\nabla_{x}\phi_{\varepsilon}+\varepsilon\nabla_{v}f_{\varepsilon}\cdot\nabla_{x}\phi_{\varepsilon}+\frac{1}{\varepsilon}\mathcal{L}_{1}f_{\varepsilon}=Q(f_{\varepsilon}, f_{\varepsilon}),\\
&\varepsilon\partial_{t}g_{\varepsilon}+v\cdot\nabla_{x}g_{\varepsilon}-2v\cdot\nabla_{x}\phi_{\varepsilon}-\varepsilon f_{\varepsilon}v\cdot\nabla_{x}\phi_{\varepsilon}+\varepsilon\nabla_{v}g_{\varepsilon}\cdot\nabla_{x}\phi_{\varepsilon}+\frac{1}{\varepsilon}\mathcal{L}_{2}g_{\varepsilon}
=Q(g_{\varepsilon}, f_{\varepsilon}),\\
&\Delta_{x}\phi_{\varepsilon}=\int_{\mathbb{R}^{3}}g_{\varepsilon}Mdv,
\end{cases}
\end{equation}
with the initial data
\begin{equation}\label{f's intial}
\begin{aligned}
f_{\varepsilon,0}^{\pm}(x,v)&=M(1+\varepsilon g_{\varepsilon,0}^{\pm}(x,v))\geq0,\\
\Delta_{x}\phi_{\varepsilon,0}(x)&=\int_{\R^{3}}(g_{\varepsilon,0}^{+}(x,v)-g_{\varepsilon,0}^{-}(x,v))Mdv,
\end{aligned}
\end{equation}
where we introduce $f_{\varepsilon}=g^{+}_{\varepsilon}+g^{-}_{\varepsilon}$, $g_{\varepsilon}=g^{+}_{\varepsilon}-g^{-}_{\varepsilon}$ and the bilinear symmetric operator $Q$, the linearized Boltzmann operators $\mathcal{L}_{1}$ and $\mathcal{L}_{2}$ are given by
\begin{equation}\label{linear operator}
\begin{aligned}
Q(f,g)&=\frac{1}{M}[B(Mf,Mg)+B(Mg,Mf)],
\mathcal{L}_{1}(f,g)=-\frac{2}{M}[B(Mf,M)+B(M,Mg)],\\
\mathcal{L}_{1}f&=\mathcal{L}_{1}(f,f)=-\frac{2}{M}[B(Mf,M)+B(M,Mf)],
\mathcal{L}_{2}f=-\frac{2}{M}B(Mf,M).
\end{aligned}
\end{equation}

Furthermore, the collisional frequency $\nu(v)$ for the hard sphere interaction is given by
\begin{equation}\label{1.10}
\nu(v)=\int_{\mathbb{R}^{3}}|v-v_{*}|M(v_{*})dv_{*}.
\end{equation}

It is easy to see that there exist two positive number $C_{1},C_{2}>0$ such that
\begin{equation}\label{nu}
 C_{1}(1+|v|)\leq\nu(v)\leq C_{2}(1+|v|).
 \end{equation}

In \cite{BGL1} and Proposition 5.7, Proposition 5.9 in \cite{Arsenio-SaintRaymond-2016},  the null space of the operator $\mathcal{L}_{1}$ and $\mathcal{L}_{2}$ are the five-dimensional space spanned and the one-dimensional space spanned as
\begin{equation}\label{Coll-I}
\begin{aligned}
\mathcal{N}_{1}=\text{Ker} \mathcal{L}_{1}&=\text{Span}\{1,v_{1},v_{2},v_{3},|v|^{2}\},\\
\mathcal{N}_{2}=\text{Ker} \mathcal{L}_{2}&=\text{Span}\{1\}.
\end{aligned}
\end{equation}

We note that the operators $\mathcal{L}_{1}$ and $\mathcal{L}_{2}$ are positive self-adjoint compact defined in $L^{2}(Mdv)$ (see Proposition 5.7 and Proposition 5.9 in \cite{Arsenio-SaintRaymond-2016} for instance). Then, we can define two projection operators $P_{i}$ from $L^{2} (Mdv)$ to $\mathcal{N}_{i}$. In what follows, we will use  $P_{1}f_{\varepsilon}$  and $P_{2}g$ as the form of
\begin{equation}\label{Decom}
\begin{aligned}
P_{1}f_{\varepsilon}&=a_{\varepsilon}(t,x)+b_{\varepsilon}\cdot v+c_{\varepsilon}(t,x)|v|^{2},\\
P_{2}g_{\varepsilon}&=d_{\varepsilon}(t,x),
\end{aligned}
\end{equation}
where $a_{\varepsilon}=\int f_{\varepsilon}(\frac{5}{2}-\frac{|v|^{2}}{2})Md v ,b_{\varepsilon}=(b_{\varepsilon,1},b_{\varepsilon,2},b_{\varepsilon,3})=\int f_{\varepsilon}vMd v, c_{\varepsilon}=\int f_{\varepsilon}(\frac{|v|^{2}}{6}-\frac{1}{2})Md v$ and $d_{\varepsilon}=\int g_{\varepsilon}Md v$.

\subsection{Notation and main results.}
In this paper, we the symbol $C(\cdot)$ denotes that constants which depend on some parameters. In addition, $A\lesssim B$ means that there exists a  constant $C > 0$ such that $A\leq CB$ and $A\thicksim B$ means that there exist two positive numbers $C_{1}, C_{2}>0$ such that $C_{1}A\leq B\leq C_{2}A$. And we introduce the following spaces. The $L^{p}$ spaces by the name of the concerned variable, Namely,
\begin{equation*}
  L_{x}^{p}=L^{p}(dx),L_{v}^{p}(\omega)=L^{p}(\omega Mdv),L^{\infty}_{x}=L^{\infty}(dx),
\end{equation*}
with the norms
\begin{equation*}
\|\phi\|_{L^{p}_{x}}=(\int_{\R^{3}}|\phi(x)|^{p}dx)^{\frac{1}{p}}<\infty,\|f\|_{L^{p}_{v}(\omega)}=(\int_{\R^{3}}|f|^{p}\omega Mdv)^{\frac{1}{p}}<\infty,\|\phi\|_{L^{\infty}_{x}}=\text{ess} \sup\limits_{x\in\R^{3}}|\phi (x)|<\infty,
\end{equation*}
where $\omega=1$ or $\nu$. If $\omega=1$, we denote by $L^{p}_{v}=L^{p}_{v}(1)$. Next we introduce $L^{p}_{x}L^{q}_{v}(\omega)$ space for $p,q\in[1,+\infty]$ endowed with the norms as follow
\begin{equation*}
  \begin{aligned}
    \| g \|_{L^p_x L^q_v (w)} =
    \left\{
      \begin{array}{l}
        \Big( \int_{\R^3} \| g (x, \cdot) \|^p_{L^q_v (w)} \d x \Big)^\frac{1}{p}  \qquad\qquad\qquad p,q \in [ 1, \infty ) \,, \\
        \textrm{ess}\sup_{x \in \R^3} \| g (x,\cdot) \|_{L^q_v (w)}  \qquad\qquad\quad\ \  p = \infty \,, q \in [ 1, \infty ) \,, \\
        \Big( \int_{\R^3} \big| \textrm{ess}\sup_{v \in \R^3} |g (x,v) \big|^p \d x \Big)^\frac{1}{p}  \quad p \in [ 1, \infty ) \,, q = \infty \,, \\
        \textrm{ess}\sup_{(x,v) \in \R^3 \times \R^3} |g (x,v) w(v)| \,, \qquad\quad \ p = q = \infty \,.
      \end{array}
    \right.
  \end{aligned}
\end{equation*}

If $w=1$, we also denote $L^p_x L^q_v = L^p_x L^q_v (1)$. Moreover, if $p = q$, the notation $L^p_{x,v} (w)$ means $L^p_x L^p_v (w)$. For $p=2$, we use $ \l\cdot , \cdot \r_{L^2_{x,v}} $, $\l \cdot , \cdot \r_{L^2_v}$ and $\l \cdot , \cdot \r_{L^2_x}$ to denote the inner product in the Hilbert spaces $L^2_{x,v}$, $L^2_v$ and $L^2_x$.

The multi-index $\alpha=[\alpha_1,\alpha_2,\alpha_3]$ and $\beta=[\beta_1,\beta_2,\beta_3]$ in $\mathbb{N}^3$ will be used to record spatial and velocity derivatives, respectively. The $(\alpha, \beta)^{th}$ partial derivative denoted by
 $$ \partial_x^\alpha \partial_v^\beta = \partial_{x_1}^{\alpha_1} \partial_{x_2}^{\alpha_2} \partial_{x_3}^{\alpha_3} \partial_{v_1}^{\beta_1} \partial_{v_2}^{\beta_2} \partial_{v_3}^{\beta_3} \,.$$

We denote $\alpha \leq \tilde{\alpha}$ if each component of $\alpha \in \mathbb{N}^3$ is not greater than that of $\tilde{\alpha}$. The symbol $\alpha < \tilde{\alpha}$ means $\alpha \leq \tilde{\alpha}$ and $|\alpha| < |\tilde{\alpha}|$, where $| \alpha | = \alpha_1 + \alpha_2 + \alpha_3$. Next, we introduce the Sobolev spaces $H^N_x$, $H^N_x L^2_v$, $H^N_x L^2_v (\nu)$, $H^N_{x,v}$ and $H^N_{x,v} (\nu)$ which endowed with the norms
\begin{equation*}
  \begin{aligned}
    & \| \phi \|_{H^N_x} = \Big( \sum_{|\alpha| \leq N} \| \partial^\alpha_x \phi \|_{L^2_x}^2 \Big)^\frac{1}{2} \,, \ \| g \|_{H^N_x L^2_v} = \Big( \sum_{|\alpha| \leq N} \| \partial^\alpha_x g \|^2_{L^2_{x,v}} \Big)^\frac{1}{2} \,, \\
    & \| g \|_{H^N_x L^2_v (\nu)} = \Big( \sum_{|\alpha| \leq N} \| \partial^\alpha_x g \|^2_{L^2_{x,v}(\nu)} \Big)^\frac{1}{2} \,, \ \| g \|_{H^N_{x,v}} = \Big( \sum_{|\alpha| + |\beta| \leq N} \| \partial^\alpha_x \partial^\beta_v g \|^2_{L^2_{x,v}} \Big)^\frac{1}{2} \,, \\
    & \| g \|_{H^N_{x,v}(\nu)} = \Big( \sum_{|\alpha| + |\beta| \leq N} \| \partial^\alpha_x \partial^\beta_v g \|^2_{L^2_{x,v}(\nu)} \Big)^\frac{1}{2} \,.
  \end{aligned}
\end{equation*}

In order to easily state our main results, we use a capital symbol $G$ denote by a column vector in $\mathbb{R}^{2}$ so that $G=(f,g)^{T}$ and the projection operator $\mathbb{P}=(P_{1},P_{2})^{T}, \mathbb{I}-\mathbb{P}=(I-P_{1},I-P_{2})^{T}$. For instance, in this paper, we use
$G_{0}=(f_{0},g_{0})^{T}$,  $G^{n}_{\varepsilon}=(f^{n}_{\varepsilon},g^{n}_{\varepsilon})^{T}$,$\mathbb{P}G=(P_{1}f,P_{2}g)^{T}$, $\|\mathbb{P}G\|_{L_{x,v}^{2}}=\|P_{1}f\|_{L^{2}_{x,v}}^{2}+\|P_{2}g\|_{L^{2}_{x,v}}^{2}$ and $\|\nabla_{x}\mathbb{P}G\|_{L^{2}_{x,v}}^{2}=\|\nabla_{x}P_{1}f\|_{L^{2}_{x,v}}^{2}+\|\nabla_{x}P_{2}g\|_{L^{2}_{x,v}}^{2}$.

Then, we introduce the following energy functional
\begin{equation}\label{Energy-E}
  \begin{aligned}
    \mathcal{E}_N (G, \phi ) = \| G\|^2_{H^{N}_{x} L^2_{v}} +\| \nabla_x \phi \|^2_{H^N_x} +\| (\mathbb{I} - \mathbb{P})G \|^2_{H^N_{x,v}},
  \end{aligned}
\end{equation}
and the energy dissipative rate functional
\begin{equation}\label{Energy-D}
  \begin{aligned}
    \mathcal{D}_{N,\varepsilon} (G) &= \frac{1}{\varepsilon^2}\| (\mathbb{I}-\mathbb{P})G \|^2_{H^N_{x,v}(\nu)}+ \|\nabla_{x}\mathbb{P}G\|^2_{H^{N-1}_x L^2_v} .
  \end{aligned}
\end{equation}

In our paper, we prove two main theorems. The first theorem is give a global-in-time solution $(f_{\varepsilon},g_{\varepsilon}, \phi_{\varepsilon})$ of the two-species Vlasov-Poisson-Boltzmann system for any given the Knudsen number $\varepsilon\in(0,1]$ near global equilibrium and the second theorem is on the two-fluid incompressible Navier-Stoker-Fourier-Poisson system limit $\varepsilon\rightarrow 0$ take in the solutions $f_{\varepsilon}, g_{\varepsilon}$ of the VPB system \eqref{VPB-g}-\eqref{f's intial} which is constructed in the first theorem.
\begin{theorem}\label{Thm-global}
Assume integer $N \geq 4$, $ 0 < \varepsilon \leq1 $, for any given $\varepsilon$, we assume $(g_{\varepsilon,0}^{+}, g_{\varepsilon,0}^{-})\in H^{N}_{x,v},\phi_{\varepsilon,0}\in H^{N}_{x}$ and satisfying \eqref{f's intial}. Then  there exists $\tau>0$, independent of $\varepsilon$, such that if $\mathcal{E}_N (G_{\varepsilon,0}, \phi_{\varepsilon, 0}) \leq \tau$, then the Cauchy problem of \eqref{VPB-g}-\eqref{f's intial} admits a global solution $(g_{\varepsilon}^{+}, g_{\varepsilon}^{-}, \phi_{\varepsilon} )$ satisfying $(g_{\varepsilon}^{+}, g_{\varepsilon}^{-})\in L^{\infty}(\R^{+};H^{N}_{x,v})$, $\nabla_{x} \phi_{\varepsilon} \in L^{\infty}(\R^{+};H^{N}_{x})$ and the global estimate
	\begin{equation}\label{Uniform-Global-Bnd}
	  \begin{aligned}
	    \sup_{t \geq 0} \mathcal{E}_N (G_{\varepsilon}(t), \phi_\varepsilon (t)) + c_0 \int_0^\infty \mathcal{D}_{N, \varepsilon} (G_{\varepsilon}(t)) \d t \lesssim \mathcal{E}_N (G_{\varepsilon,0}, \phi_{\varepsilon, 0} ) \,,
	  \end{aligned}
	\end{equation}
	where $c_0 > 0$ is independent of $\varepsilon$. Moreover, $\phi_\varepsilon (t,x)$ and
	\begin{equation*}
    f^{\pm}_{\varepsilon}(t,x,v) = M (v) ( 1 + \varepsilon g^{\pm}_\varepsilon (t,x,v))\geq0
	\end{equation*}
	obeys the Cauchy problem \eqref{VPB-g}-\eqref{f's intial}.
\end{theorem}

The next theorem is about the limit to the two-fluid incompressible NSFP system with Ohm's law. This system is a
macroscopic description of a fluid based on its fluctuations of mass density $\rho(t,x)$, bulk velocity $u(t,x)$, temperature $\theta(t,x)$, the electric charge $n(t,x)$, the electric current $j(t,x)$ and the internal electric energy $w(t,x)$ in the electric filed $\nabla_{x}\phi(t,x)$
\begin{equation}\label{NSFP}
  \begin{cases}
    & \partial_t u + u \cdot \nabla_x u + \nabla_{x} p = \mu \Delta_x u +\frac{1}{2}n\nabla_{x}\phi \,,\\
    & \partial_t \theta+ u \cdot \nabla_x \theta= \kappa \Delta_x \theta \,,\\
    &\partial_{t}n+u\cdot\nabla_{x}n+\sigma n=\frac{\sigma}{2}\Delta_{x}n\,,\\
    &j=nu+\sigma(\nabla_{x}\phi-\frac{1}{2}\nabla_{x}n)\,,\\
    & \Delta_x\phi =n  \,,w=n\theta\,, \rho+\theta=0\,, \text{div}_{x}u= 0,
  \end{cases}
\end{equation}
where $p(t,x)$ is the pressure, the viscosity $\mu>0$, the heat conductivity $\kappa>0$ and the electrical conductivity $\sigma>0$ are defined as
\begin{equation}\label{mu-kappa}
  \mu = \frac{1}{15} \l A_{ij}, \widehat{A}_{ij} \rangle_{L^2_v} \,, \quad \kappa = \frac{2}{15}
\l B_i,\widehat{B}_i \rangle_{L^2_v} \,,
\end{equation}
\begin{equation}\label{sigma-lambda}
\frac{1}{\sigma}=\frac{1}{2}\int_{\mathbb{R}^{3}}v\cdot\mathcal{L}_{1}(v,-v)Mdv\,,\quad
\end{equation}
\begin{equation*}
  A(v) = v \otimes v -\frac{|v|^2}{3}I,\quad\quad B(v) = (\frac{|v|^2}{2} - \frac{5}{2})v.
\end{equation*}

For the qualities $A(v), B(v)\in\mathcal{N}^{\bot}\mathcal{L}_{1}$ and there exist unique $\widehat{A},\widehat{B}\in L^2(Mdv)$ such that(see \cite{BGL1})
\begin{equation*}
  A  =   \mathcal{L}_{1} \widehat{A} \,\quad \text{and} \quad B =  \mathcal{L}_{1} \widehat{B} \,.
\end{equation*}

\begin{theorem}\label{Thm-Limit}
	Under the same assumptions as in the Theorem \ref{Thm-global}, suppose $ 0 <\varepsilon\leq 1$, $N \geq 4$ and $\tau > 0$ be mentioned in the Theorem \ref{Thm-global}. Furthermore, we assume that there exist functions $\nabla_{x}\phi_{0}(x), u_{0}(x), \theta_{0}(x)\in H^N_{x}$ and
\begin{equation}\no
\begin{aligned}
\rho(0,x)&=\rho_{0}(x)=-\theta_{0}(x), n(0,x)=n_{0}(x)=\Delta_{x}\phi_{0}(x), w(0,x)=w_{0}(x)=\Delta_{x}\phi_{0}(x)\theta_{0}(x),\\ j(0,x)&=j_{0}(x)=\Delta_{x}\phi_{0}(x)\mathcal{P}u_{0}(x)+\sigma(\nabla_{x}\phi_{0}(x)-\frac{1}{2}\nabla_{x}\Delta_{x}\phi_{0}(x)),
\end{aligned}
\end{equation}
satisfying
	\begin{equation}
	  \begin{aligned}
	    \frac{g_{\varepsilon,0}^{+}(x,v)+g_{\varepsilon,0}^{-}(x,v)}{2} \rightarrow \frac{g_{0}^{+}(x,v)+g_{0}^{-}(x,v)}{2} = \rho_{0} (x) + u_{0} (x) \cdot v + \theta_{0}(x) ( \frac{|v|^2}{2} - \frac{3}{2} ),
	  \end{aligned}
	\end{equation}
	strongly in $H^N_{x,v}$ as $\varepsilon \rightarrow 0$, where P is the Leray projection on $\R^{3}$. Moreover, we assume that
\begin{equation}
\begin{aligned}
&\l g_{\varepsilon,0}(x,v)-g_{\varepsilon,0}(x,v), 1\r_{L^{2}_{v}}\rightarrow n_{0}(x)=\Delta_{x}\phi_{0}(x),\\
&\frac{1}{\varepsilon}\l g_{\varepsilon,0}(x,v)-g_{\varepsilon,0}(x,v), v\r_{L^{2}_{v}}\rightarrow j_{0}(x),\\
&\frac{1}{\varepsilon}\l g_{\varepsilon,0}(x,v)-g_{\varepsilon,0}(x,v), \frac{|v|^{2}}{3}-1\r_{L^{2}_{v}}\rightarrow w_{0}(x),\\
\end{aligned}
\end{equation}
strongly in $H^N_{x}$ as $\varepsilon \rightarrow 0$.
Let $( g_{\varepsilon}^{+} (t,x,v), g_{\varepsilon}^{-} (t,x,v), \phi_{\varepsilon} (t,x) )$ be the family of solutions to the two-species VPB systems \eqref{VPB-g} which constructed in Theorem \ref{Thm-global}. Then,
	\begin{equation}
	  \frac{g_{\varepsilon}^{+}(t,x,v)+g_{\varepsilon}^{-}(t,x,v)}{2} \rightarrow \rho (t,x) + u(t,x) \cdot v + \theta (t,x) ( \frac{|v|^2}{2} - \frac{3}{2} ),
	\end{equation}
	weakly-$\star$ for $t \geq 0$, strongly in $H^{N-1}_{x,v} $ as $\varepsilon \rightarrow 0$. Moreover,
\begin{equation}
\begin{aligned}
&\l g_{\varepsilon}(t,x,v)-g_{\varepsilon}(t,x,v), 1\r_{L^{2}_{v}}\rightarrow n(t,x)=\Delta_{x}\phi(t,x),\\
&\frac{1}{\varepsilon}\l g_{\varepsilon}(t,x,v)-g_{\varepsilon}(t,x,v), v\r_{L^{2}_{v}}\rightarrow j(t,x),\\
&\frac{1}{\varepsilon}\l g_{\varepsilon}(t,x,v)-g_{\varepsilon}(t,x,v), \frac{|v|^{2}}{3}-1\r_{L^{2}_{v}}\rightarrow w(t,x),\\
\end{aligned}
\end{equation}
weakly-$\star$ for $t \geq 0$, strongly in $H^{N-1}_{x}$ as $\varepsilon \rightarrow 0$. Here $ (\rho, u, \theta, n, j, w, \nabla_{x}\phi) \in L^\infty (\R^+; H^N_x) $ with $ (\rho, u, \theta, n, j, w, \nabla_{x}\phi) \in C(\R^+; H^{N-1}_x) $ is the solution of the two-fluid incompressible NSFP system \eqref{NSFP} with Ohm's law with initial data:
\begin{equation}\no
\begin{aligned}
\rho(0,x)&=\rho_{0}(x)=-\theta_{0}(x), n(0,x)=n_{0}(x)=\Delta_{x}\phi_{0}(x), w(0,x)=w_{0}(x)=\Delta_{x}\phi_{0}(x)\theta_{0}(x),\\
j(0,x)&=j_{0}(x)=\Delta_{x}\phi_{0}(x)\mathcal{P}u_{0}(x)+\sigma(\nabla_{x}\phi_{0}(x)-\frac{1}{2}\nabla_{x}\Delta_{x}\phi_{0}(x)),\\
u(0,x)&=\mathcal{P}u_{0}(x), \theta(0,x)=\theta_{0}(x), \phi(0,x)=\phi_{0}(x).
\end{aligned}
\end{equation}
Moreover, the following convergence
\begin{equation}\no
\begin{aligned}
&\l\frac{g_{\varepsilon}^{+}(t,x,v)+g_{\varepsilon}^{-}(t,x,v)}{2}, 1\r_{L^{2}_{v}}\rightarrow \rho(t,x),\\
&\mathcal{P}\l\frac{g_{\varepsilon}^{+}(t,x,v)+g_{\varepsilon}^{-}(t,x,v)}{2}, v\r_{L^{2}_{v}}\rightarrow u(t,x),\\
&\l\frac{g_{\varepsilon}^{+}(t,x,v)+g_{\varepsilon}^{-}(t,x,v)}{2}, \frac{|v|^{2}}{3}-1\r_{L^{2}_{v}}\rightarrow \theta(t,x),\\
&\l g_{\varepsilon}^{+}(t,x,v)-g_{\varepsilon}^{-}(t,x,v), 1\r_{L^{2}_{v}}\rightarrow n(t,x),\\
\end{aligned}
\end{equation}
strongly in $C(\R^{+};H^{N-1}_{x})$, weakly-$\star$ in $t\geq0$, holds.
\end{theorem}

\subsection{Difficulties and ideas}
In our paper, the key point is to deduce a uniform energy bound in $\varepsilon \in (0,1]$ of the global-in-time solutions to the two-species perturbed VPB system \eqref{VPB-g} near equilibrium.  We introduce two functions $f_{\varepsilon}(t,x,v)=g^{+}_{\varepsilon}(t,x,v)+g_{\varepsilon}^{-}(t,x,v),g_{\varepsilon}(t,x,v)=g^{+}_{\varepsilon}(t,x,v)-g_{\varepsilon}^{-}(t,x,v)$ according the formal analysis of the two-species VPB system to the NSFP system with Ohm's law. Therefore, the two kinetic functions $f_{\varepsilon}, g_{\varepsilon}$ can be divided into two different parts of macroscopic part and microscopic part, respectively
\begin{equation*}
\begin{aligned}
f_{\varepsilon}&= P_{1}f_{\varepsilon}+ (I-P_{1})f_{\varepsilon},g_{\varepsilon}=P_{2}g_{\varepsilon}+(I-P_{2})g_{\varepsilon},
\end{aligned}
\end{equation*}
where $P_{1}f_\varepsilon \in \mathcal{N}_{1}$  and $P_{2}g_{\varepsilon}\in\mathcal{N}_{2}$ are so-called the fluid  part of $f_{\varepsilon},g_{\varepsilon} $, and $(I-P_{1}) f_\varepsilon\in \mathcal{N}_{1}^\perp, (I-P_{2}) g_\varepsilon\in \mathcal{N}_{2}^\perp$ are called the kinetic part of $f_{\varepsilon}, g_{\varepsilon}$, where $\mathcal{N}_{i}^\perp$ is the orthogonal space of $\mathcal{N}_{i} (i=1,2)$ in $L^{2}_{v}$, respectively. Plugging the above decomposition into the scaled equation \eqref{VPB-g}, we obtain
\begin{equation}\label{decomposition1}
  \begin{split}
    \varepsilon \partial_{t} ( a_{\varepsilon}+ b_{\varepsilon} \cdot v + c_{\varepsilon}|v|^2 ) + v \cdot \nabla_x ( a_{\varepsilon}+ b_{\varepsilon}\cdot v + c_{\varepsilon}|v|^2 ) = l + h + m \,,
  \end{split}
\end{equation}
\begin{equation}\label{decomposition2}
\begin{aligned}
\varepsilon\partial_{t}d_{\varepsilon}+v\cdot\nabla_{x}d_{\varepsilon}=\hat{l}+\hat{h}+\hat{m},
\end{aligned}
\end{equation}
where
\begin{equation}\label{l+h+m}
  \begin{aligned}
    l = & - \varepsilon \partial_{t} (I-P_{1}) f_{\varepsilon} - v \cdot \nabla_{x} (I-P_{1}) f_{\varepsilon}- \frac{1}{\varepsilon} \mathcal{L}_{1} (I-P_{1}) f_{\varepsilon}\,, \\
    h = & Q (f_{\varepsilon}, f_{\varepsilon}) \,, \\
    m = &\varepsilon (g_{\varepsilon}v \cdot \nabla_{x} \phi_{\varepsilon} - \nabla_{x} \phi_{\varepsilon}\cdot \nabla_{v} g_{\varepsilon}) \,,\\
    \hat{l}=&-\varepsilon\partial_{t}(I-P_{2})g_{\varepsilon}-v\cdot\nabla_{x}(I-P_{2})g_{\varepsilon}-\frac{1}{\varepsilon}\mathcal{L}_{2}(I-P_{2})g_{\varepsilon},\\
    \hat{h}=&Q(g_{\varepsilon},f_{\varepsilon}),\\
    \hat{m}=&2v\cdot\nabla_{x}\phi+\varepsilon f_{\varepsilon}v\cdot\nabla_{x}\phi_{\varepsilon}-\varepsilon\nabla_{v}f_{\varepsilon}\cdot\nabla_{x}\phi_{\varepsilon}.
  \end{aligned}
\end{equation}

Furthermore, $a_{\varepsilon}$, $b_{\varepsilon}$, $c_{\varepsilon}$ and $d_{\varepsilon}$ follow the local macroscopic balance laws of mass, moment, energy from $f_{\varepsilon}$ and balance law of mass from $g_{\varepsilon}$. In fact, we multiply the first $f_{\varepsilon},g_{\varepsilon}$-equation \eqref{VPB-g} of the VPB system by the collision invariant in \eqref{Coll-I} and integrate by parts over $v \in \mathbb{R}^3$, then
\begin{equation*}
\left\{
\begin{aligned}
  & \partial_{t} \int_{\mathbb{R}^3} f_{\varepsilon}Md v + \frac{1}{\varepsilon}\nabla_{x}\cdot \int_{\mathbb{R}^3} v f_{\varepsilon}M d v=0 \,, \\
  & \partial_{t} \int_{\mathbb{R}^3} f_{\varepsilon}v M d v + \frac{1}{\varepsilon}\nabla_{x}\cdot \int_{\mathbb{R}^3} f_{\varepsilon}v \otimes v M d v-\nabla_{x}\phi_{\varepsilon}\int_{\mathbb{R}^{3}}g_{\varepsilon}Md v=0 \,, \\
  & \partial_{t} \int_{\mathbb{R}^3} f_{\varepsilon}|v|^2 d v + \frac{1}{\varepsilon}\nabla_{x}\cdot \int_{\mathbb{R}^3} f_{\varepsilon}|v|^{2}v d v-2\nabla_{x}\phi_{\varepsilon}\int_{\mathbb{R}^{3}}g_{\varepsilon}v Md v=0\,,\\
  &\partial_{t}\int_{\mathbb{R}^{3}}g_{\varepsilon}Md v+\frac{1}{\varepsilon}\int_{\mathbb{R}^{3}}v\cdot\nabla_{x}g_{\varepsilon}Md v=0.
\end{aligned}
\right.
\end{equation*}

We obtain the macroscopic balance laws by the perturbed form $f_{\varepsilon}^{+}+f_{\varepsilon}^{-}= M ( 2 + \varepsilon f_{\varepsilon} ), f_{\varepsilon}^{+}-f_{\varepsilon}^{-}=\varepsilon M g_{\varepsilon}$ and the decomposition \eqref{Decom}
\begin{equation}\label{a-b-c}
\left\{
\begin{aligned}
&\partial_{t} (a_{\varepsilon}+ 3 c_{\varepsilon} ) + \frac{1}{\varepsilon} \nabla_{x}\cdot b_\varepsilon= 0  \\
& \partial_{t} b_{\varepsilon}+ \frac{1}{\varepsilon} \nabla_{x} (a_{\varepsilon} + 5 c_{\varepsilon}) + \frac{1}{\varepsilon} \l v \cdot \nabla_x (I-P_{1}) f_{\varepsilon} , v \rangle_{L^2_v}-d_{\varepsilon}\nabla_{x}\phi_{\varepsilon}=0,\\
&\partial_{t}c_{\varepsilon}+\frac{1}{3\varepsilon}\nabla_{x}\cdot b+\frac{1}{6\varepsilon}\l v\cdot\nabla_{x}(I-P_{1})f_{\varepsilon},|v|^{2}\rangle_{L^{2}_{v}}-\frac{1}{3}\nabla_{x}\phi_{\varepsilon}\cdot\l(I-P_{2})g_{\varepsilon},v\rangle_{L^{2}_{v}}=0,\\
&\partial_{t}d_{\varepsilon}+\frac{1}{\varepsilon}\l v\cdot\nabla_{x}(I-P_{2})g_{\varepsilon},1\rangle_{L^{2}_{v}}=0,\\
&\Delta_x \phi_\varepsilon =d_{\varepsilon}(t,x).
\end{aligned}
\right.
\end{equation}

We will divide into three steps to derive the uniform global energy estimates of the perturbed VPB system \eqref{VPB-g} :

Step one: to deduce the pure spatial derivative energy estimates. We will obtain the kinetic dissipation $\frac{1}{\varepsilon^2} \| (\mathbb{I}-\mathbb{P}) G_\varepsilon \|^2_{H^N_x L^2_v (\nu)}$ thank to the coercivity of linearize collision operator $\mathcal{L}_{i}(i=1,2)$ as shown in Lemma \ref{Lmm-Coercivity-L}. Furthermore, the singular term $\frac{1}{\varepsilon} v \cdot \nabla_{x}G_{\varepsilon}$ will disappear since $\frac{1}{\varepsilon} \l v \cdot \nabla_x \partial^\alpha_xG_\varepsilon , \partial^\alpha_x G_\varepsilon \r_{L^2_{x,v}}= 0$ for all $\alpha \in \mathbb{N}^3$.  The singularity $\frac{1}{\varepsilon} \partial^\alpha_x (\mathbb{I}-\mathbb{P}) G_\varepsilon$ will be controlled by the kinetic dissipation $\frac{1}{\varepsilon^2} \| (\mathbb{I}-\mathbb{P}) G_\varepsilon \|^2_{H^N_x L^2_v (\nu)}$. Next for the linear singular term $ \frac{1}{\varepsilon} v \cdot \nabla_x \phi_\varepsilon$, the local conservation law of mass $\partial_t \l g_\varepsilon , 1 \r_{L^2_v} + \frac{1}{\varepsilon} \l g_\varepsilon, v \r_{L^2_v} = 0$ and the Poisson equation $\Delta_x \phi_\varepsilon = \l g_{\varepsilon} , 1 \r_{L^2_v}$ will be used such that the singular quantity $\l  \frac{1}{\varepsilon} v \cdot \nabla_x \partial^\alpha_x \phi_\varepsilon, \partial^\alpha_x g_{\varepsilon} \r_{L^2_{x,v}} = \frac{1}{2} \frac{d}{d t} \| \nabla_x \partial^\alpha_x \phi_\varepsilon \|^2_{L^2_x} $. As for the nonlinear and singular term $\frac{1}{\varepsilon}Q(f_{\varepsilon},f_{\varepsilon})$, $\frac{1}{\varepsilon}Q(g_{\varepsilon},f_{\varepsilon})$, it's easy to derive that $\frac{1}{\varepsilon}\l \partial_{x}^{\alpha} Q(f_{\varepsilon},f_{\varepsilon}), \partial_{x}^{\alpha}f_{\varepsilon}\r_{L^{2}_{v}}=\frac{1}{\varepsilon}\l \partial_{x}^{\alpha} Q(f_{\varepsilon},f_{\varepsilon}), \partial_{x}^{\alpha}(I-P_{1})f_{\varepsilon}\r_{L^{2}_{v}},$
$ \frac{1}{\varepsilon}\l \partial_{x}^{\alpha} Q(g_{\varepsilon},f_{\varepsilon}), \partial_{x}^{\alpha}g_{\varepsilon}\r_{L^{2}_{v}}=\frac{1}{\varepsilon}\l \partial_{x}^{\alpha} Q(f_{\varepsilon},f_{\varepsilon}), \partial_{x}^{\alpha}(I-P_{2})g_{\varepsilon}\r_{L^{2}_{v}}$ since $Q(f_{\varepsilon},f_{\varepsilon})\in\mathcal{N}_{1}^{\perp}, Q(g_{\varepsilon},f_{\varepsilon})\in\mathcal{N}_{2}^{\perp}$. Then we use the decomposition of $G_{\varepsilon}=\mathbb{P}G_{\varepsilon}+(\mathbb{I}-\mathbb{P})G_{\varepsilon}$ and these items can be controlled by the kinetic dissipation $\frac{1}{\varepsilon^2} \| (\mathbb{I}-\mathbb{P})G_\varepsilon \|^2_{H^N_x L^2_v (\nu)}$.

Step two: to estimate the macroscopic energy to find a dissipative structure of the fluid part $\mathbb{P}G$ by employing the so-called macro-micro decomposition method, depending on the {\em thirteen moments} (see \cite{GY-06}). Therefore, the fluid dissipation $\| \nabla_x \mathbb{P}G_\varepsilon\|^2_{H^{N-1}_x L^2_v}$ can be obtained. Applying the divergence operator $\nabla_x \cdot $ on the balance law \eqref{a-b-c} for $b_\varepsilon$, a more damping effect $\| \l g_\varepsilon, 1 \r_{L^2_v} \|^2_{H^{N-1}_x} = \| a_\varepsilon+ 3 c_\varepsilon\|^2_{H^{N-1}_x}$ in the perturbed VPB system resulted from the Poisson equation $\Delta_x \phi_\varepsilon = \l g_\varepsilon, 1 \r_{L^2_v} = a_\varepsilon + 3 c_\varepsilon$ show that
\begin{equation*}
  \begin{aligned}
    - \Delta_x (a_\varepsilon + 3 c_\varepsilon)= \textrm{ other terms} \,.
  \end{aligned}
\end{equation*}

Applying the method of {\em thirteen moments}, the term $b_{\varepsilon}(t,x)$(see \eqref{3.3.5}) can be yielded
\begin{equation}\no
-\Delta_{x}b_{i,\varepsilon}-\tfrac{1}{3}\partial_{i}\nabla_{x}\cdot b_{\varepsilon}=\sum_{j=1}^{3}\l l+h+m, \zeta_{ij}\r_{L^{2}_{v}}+\textrm{ other terms},
\end{equation}
where $\zeta_{ij}$ is a certain linear combination of the basis in \eqref{e 13}. Analogously, the term $c_{\varepsilon}(t,x)$( see \eqref{3.3.12})can be yielded
\begin{equation}\no
-\Delta_{x}c_{\varepsilon}=\sum_{i=1}^{3}\l l+h+m, \zeta_{i}\r_{L^{2}_{v}},
\end{equation}
where $\zeta_{i}$ is a certain linear combination of the basis in \eqref{e 13}. Moreover, as for the term $d_{\varepsilon}$, we take inner product with $\l\cdot,v\r_{L^{2}_{v}}$ in \eqref{decomposition2} and yield that
\begin{equation}\no
\nabla_{x}d=\l\hat{l}+\hat{h}+\hat{m},v\r_{L^{2}_{v}}.
\end{equation}

Noticed, no singular terms are generated in deriving the macroscopic energy estimates. However, a unsigned {\em interactive energy} quantity $E_N^{int} (G)$ defined in \eqref{Interactive-Energy} will appear. Fortunately, the unsigned interactive energy $E_N^{int} (G)$ is so small size that it can be dominated by the energy $\mathcal{E}_N (G, \phi_\varepsilon)$ (see Remark \ref{Rmk-MM}).

Step three: to derive the $(x,v)$-mixed derivatives estimates to closed the energy inequality. The uncontrolled quantity in the spatial derivative and macroscopic energy estimates are in terms of the $v$-derivatives of the kinetic part $(\mathbb{I}-\mathbb{P}) G_{\varepsilon}$. We employ the microscopic projection $I-P_{1}$ to the $f$-equation of \eqref{VPB-g}and $I-P_{2}$ to the $g$-equation of \eqref{VPB-g}, then we have
\begin{equation}\no
  \begin{aligned}
    &\partial_t (I-P_{1})f_\varepsilon + \frac{1}{\varepsilon^2} \mathcal{L}_{1} (I-P_{1}) f_\varepsilon = \frac{1}{\varepsilon} Q (f_\varepsilon, f_\varepsilon) + \textrm{ other nonsingular terms},\\
    &\partial_t (I-P_{2})g_\varepsilon + \frac{1}{\varepsilon^2} \mathcal{L}_{2} (I-P_{2}) g_\varepsilon =\frac{1}{\varepsilon} Q (g_\varepsilon, f_\varepsilon) + \textrm{ other nonsingular terms}.\\
  \end{aligned}
\end{equation}

In the mixed derivatives about speed and spatial estimate,  we will obtain a kinetic dissipation $\frac{1}{\varepsilon^2} \| \partial^{\alpha}_{x} \partial^{\beta}_{v} (\mathbb{I}-\mathbb{P}) G_\varepsilon\|^2_{L^2_{x,v} (\nu)}$ for all $| \alpha | + |\beta| \leq N$ with $\beta \neq 0$ thank to the coercivity of $\mathcal{L}_{1},\mathcal{L}_{2}$. As a consequence, the singular term $ - \frac{1}{\varepsilon} \l \partial^\alpha_x \partial^\beta_v (I-P_{1}) (v \cdot \nabla_x f_\varepsilon) + Q (f_\varepsilon, f_\varepsilon) , \partial^\alpha_x \partial^\beta_v (I-P_{1}) f_\varepsilon\r_{L^2_{x,v}} $ and $ - \frac{1}{\varepsilon} \l \partial^\alpha_x \partial^\beta_v (I-P_{2}) (v \cdot \nabla_x g_\varepsilon) + Q (g_\varepsilon, f_\varepsilon) , \partial^\alpha_x \partial^\beta_v (I-P_{2})g_\varepsilon\r_{L^2_{x,v}} $ will be controlled by the kinetic dissipation $ \frac{1}{\varepsilon^2} \| \partial^{\alpha}_{x} \partial^{\beta}_{v} (\mathbb{I}-\mathbb{P}) G_\varepsilon\|^2_{L^2_{x,v} (\nu)}$. However, for fixed $|\alpha| + |\beta| \leq N$ with $\beta \neq 0$, the coercivity of $\mathcal{L}_{i}(i=1,2)$ under the $v$-derivatives $\partial^\beta_v$ will further generate a uncontrolled quantity $\frac{1}{\varepsilon^2} \sum_{\tilde{\beta}< \beta} \| \partial^\alpha_x \partial^{\tilde{\beta}}_v (\mathbb{I}-\mathbb{P}) G_\varepsilon\|^2_{L^2_{x,v} (\nu)}$ with two order singularity $\frac{1}{\varepsilon^2}$. Due to the order $\tilde{\beta}$ of $v$-derivatives in that quantity is strictly less than $\beta$ $(\neq 0)$, we can employ the induction to absorb this uncontrolled singular norm by the mixed derivative kinetic dissipation with lower order $v$-derivative. Therefore, we establish the global uniform energy estimates.

Finally, we take the limit from the perturbed VPB system \eqref{VPB-g} to the incompressible NSFP system with Ohm's law \eqref{NSFP} as $\varepsilon \rightarrow 0$ base on the global-in-time energy estimate uniformly in $\varepsilon \in (0,1]$. Then, we apply the Aubin-Lions-Simon Theorem to obtain enough compactness such that the limits valid.

\subsection{Historical progress in this field}
There has been tremendous progress on the well-posedness of kinetic equations. DiPerna and Lions \cite{D-L} obtained the global renormalized solutions to the Boltzmann equation for bounded initial data. Later, Lions applied this theory to the VPB system (\cite{Lions1994-1}). For the classical solutions, Ukai \cite{Ukai} first considered the hard potential collision kernels.  Guo developed nonlinear energy estimation to prove the existence of global-in-time classical solutions to the Boltzmann equations near equilibrium \cite{guo-1, Guo-2002-VPB}. Later, there were more results on different collision kernels, among which we only listed a few results on VPB \cite{D-Y-Z(hard)-2011, D-Y-Z(soft)-2011, XXZhard, XXZsoft}.

One of the most important features of kinetic equations (i.e. Boltzmann-type equation) is their connection to fluid equations in the regime where Knudsen number $\varepsilon$ is very small. Hydrodynamic limits from kinetic equations have been an active research field from late 70's.

Most results in the context of classical solutions are obtained by Hilbert expansion. In \cite{Caflisch, Nishida}, Nishida and Caflisch used Hilbert expansion on the compressible Euler limit. Combining with nonlinear energy method and Hilbert expansion,  Guo-Jang-Jiang justified the acoustic limit \cite{GJJ-KRM2009,GJJ-CPAM2010}. Furthermore, Guo proved the incompressible Navier-Stokes limit \cite{GY-06}. All these results were based in Hilbert expansion. On the other direction, i.e. without employing Hilbert expansion, Bardos and Ukai \cite{b-u} proved the convergence for small data classical solutions from Boltzmann equation with hard potential to incompressible Navier-Stokes equations using the semigroup method and spectrum analysis of linearized Boltzmann operator. For general collision kernels, Briant, Jiang-Xu-Zhao and Gallagher-Tristani also proved this incompressible Navier-Stokes limits recently \cite{Briant, BMMouhot, Gallagher-Tristani, Jiang-Xu-Zhao-2018-Indiana}.

The BGL program (named after Bardos-Golse-Levermore's work \cite{BGL1, BGL2}) is to justify weak limit starting from DiPerna-Lions' renormalized solutions of Boltzmann equations to weak solutions of incompressible Navier-Stokes. This program are completed by  Golse and Saint-Raymond with cutoff Maxwell collision kernel in \cite{Go-Sai04}. Later, this convergence result was extended to soft potentials cases, non-cutoff and bounded domain cases, etc (see \cite{Arsenio, MSRM-CPAM2003, JM-CPAM2017}).

For VPB system, Guo and Jang prived the limit from the scaled VPB system to compressible Euler-Poisson system with hard-sphere interaction by the Hilbert expansion \cite{GJ}. Recently, Jiang and Zhang considered the sensitivity analysis method and energy estimates to justify the incompressible Navier-Stokes-Poisson limit and obtain the precise convergence rate without employing any results
based on Hilbert expansion for the first time \cite{Jiang-ZhangX}. In \cite{GJL}, the authors have prove the limit of the one-species VPB system to incompressible Navier-Stokes-Fourier-Poisson system by the approach of the nonlinear energy method.

In our current paper is {\em not} to take a Hilbert expansion approach in the context of classical solutions for two-species VPB system under the hard sphere potential.  We are extending the work of \cite{GJL} to the case of two particles and consider the two particles in a strong interspecies collisions. We derive the uniform bounded in $\varepsilon$ of the species sequence classical solutions $g^{\pm}_{\varepsilon}$ to the perturbed VPB system $(f^{\pm}_{\varepsilon}=M+\varepsilon Mg_{\varepsilon}^{\pm})$ near equilibrium and then rigorously justify the limit to the incompressible NSFP equations with Ohm's law.

In our paper, the two types of particles will interact when we consider two-species VPB system \eqref{scaled VPB}, such as the collisions $B(f^{+}_{\varepsilon}, f^{-}_{\varepsilon})$ and $ B(f^{-}_{\varepsilon}, f^{+}_{\varepsilon})$ have been added to the right-hand sides of equation \eqref{scaled VPB} due to interspecies collisions. As a consequence, we will encounter some difficulties which will not arise when considering the one-specie in \cite{GJL}. In the two-species perturbed VPB system \eqref{VPB-g}, $\frac{1}{\varepsilon}Q(g_{\varepsilon}, f_{\varepsilon})$ is  added since the appearance of interspecies collisions $\frac{1}{\varepsilon}B(f^{\pm}_{\varepsilon}, f^{\mp}_{\varepsilon})$. Recalling \eqref{scaled VPB}, we can derive the equation of $f_{\varepsilon}$ and $g_{\varepsilon}$ as follow
\begin{equation}\label{VPB}
\begin{cases}
&\partial_{t}f_{\varepsilon}+\frac{1}{\varepsilon}v\cdot\nabla_{x}f_{\varepsilon}- g_{\varepsilon}v\cdot\nabla_{x}\phi_{\varepsilon}+\nabla_{v}f_{\varepsilon}\cdot\nabla_{x}\phi_{\varepsilon}+\frac{1}{\varepsilon^{2}}\mathcal{L}_{1}f_{\varepsilon}=\frac{1}{\varepsilon}Q(f_{\varepsilon}, f_{\varepsilon})\\
&\partial_{t}g_{\varepsilon}+\frac{1}{\varepsilon}v\cdot\nabla_{x}g_{\varepsilon}-\frac{2}{\varepsilon}v\cdot\nabla_{x}\phi_{\varepsilon}- f_{\varepsilon}v\cdot\nabla_{x}\phi_{\varepsilon}+\nabla_{v}g_{\varepsilon}\cdot\nabla_{x}\phi_{\varepsilon}+\frac{1}{\varepsilon^{2}}\mathcal{L}_{2}g_{\varepsilon}
=\frac{1}{\varepsilon}Q(g_{\varepsilon}, f_{\varepsilon})\\
&\Delta_{x}\phi_{\varepsilon}=\int_{\mathbb{R}^{3}}g_{\varepsilon}Mdv.
\end{cases}
\end{equation}
Then we define two linear operator $\mathcal{L}_{1}$ and $\mathcal{L}_{2}$ in \eqref{linear operator} and two projection operators $P_{i}:L^{2}(Mdv)\rightarrow \text{Ker} \mathcal{L}_{i}(i=1,2)$ in \eqref{Decom}. On the one hand,  in order to deduce the fluid part of $g_{\varepsilon}$, we use the decomposition $g_{\varepsilon}=P_{2}g_{\varepsilon}+(I-P_{2})g_{\varepsilon}$ to the $g_{\varepsilon}$ in \eqref{VPB}. Therefore, we have equation of $\partial_{t}d+v\cdot\nabla_{x}d=\hat{l}+\hat{h}+\hat{m}$ in \eqref{decomposition2}, where $d(t,x)=P_{2}g_{\varepsilon}$. And we take inner product $L^{2}_{v}$ with $v$ in this equation and the $\partial_{t}d$ is vanish. Then, we obtain $\nabla_{x}b=\l \hat{l}+\hat{h}+\hat{m},v\r_{L^{2}_{v}}$. Finally, the dissipation of the fluid part $\|\nabla_{x}P_{2}g_{\varepsilon}\|_{L^{2}_{x}}^{2}$ is obtained. On the other hand, we can estimate the term $\frac{1}{\varepsilon}Q(g_{\varepsilon}, f_{\varepsilon})$ thanks to the  to the coercivity
of linearized collision operator $\mathcal{L}_{2}$  when we derive the $(x,v)$-mixed derivatives estimates. To derive the microscopic part, we apply the operator $I-P_{2}$ to the $g_{\varepsilon}$ in \eqref{VPB} and apply the kinetic dissipation of $\frac{1}{\varepsilon^{2}}\|(I-P_{2})g_{\varepsilon}\|^{2}_{L^{2}_{x,v}(\nu)}$.
The coercivity of the linearized operator $\mathcal{L}_{2}$ and the fluid part $d(t,x)$ of $\mathcal{L}_{2}$ have dissipative effect.
These are the most important steps and the main
novelty of this paper, see Section \ref{local-in-time} and Section \ref{Sec: Uniform-Bnd}.

Our paper is organized as follows. In Section \ref{local-in-time}, we construct a local-in-time solution $(g^{+}_{\varepsilon},g^{-}_{\varepsilon}, \phi_{\varepsilon})$ of the perturbed VPB system with small initial data for any given $\varepsilon\in(0,1]$. In Section \ref{Sec: Uniform-Bnd}, we gain the uniform bounded in $\varepsilon$ to the $(g^{+}_{\varepsilon},g^{-}_{\varepsilon}, \phi_{\varepsilon})$ and extend the local-in-time solution in Section \ref{local-in-time} into a global-in-time solution under the small size of the initial data. In Section \ref{Sec: Limits}, we rigorously derive the limit from the perturbed VPB system \eqref{VPB-g} to the two-fluid incompressible NSFP equations \eqref{NSFP} with Ohm'law as $\varepsilon \rightarrow 0$.

\section{Construction of Local Solutions}\label{local-in-time}

In this section, we will prove that the perturbed VPB system \eqref{VPB-g}-\eqref{f's intial} has a unique local-in-time solution for all $0 < \varepsilon \leq 1$ by employing an iterative schedule. Before doing that, we will first do some preparatory work.

\subsection{Some Lemmas}
In order to derive the equations of $n,j$ and $w$, we introduce two functions in $\mathcal{L}_{2}$'s range as follow
\begin{equation}
\Phi(v)=v, \Psi(v)=\frac{|v|^{2}}{2}-\frac{3}{2}\in L^{2}(Mdv).
\end{equation}
Thus, there are inverses $\hat{\Phi}\in L^{2}(Mdv)$ and $\hat{\Psi}\in L^{2}(Mdv)$ such that
\begin{equation}\label{hat{PP}}
\Phi=\mathcal{L}_{2}\hat{\Phi}, \Psi=\mathcal{L}_{2}\hat{\Psi},
\end{equation}
which can be uniquely determined by the fact that they are orthogonal to the kernel of $\mathcal{N}_{2}$ (see the Section 2.4.5 of \cite{Saint-Raymond-2009-Boltzmann}).

\begin{lemma}\label{a-b}
There exist two scalar-valued functions $\alpha,\beta:[0,+\infty)\to \mathbb{R}$  such that
\begin{center}
$\hat{\Phi}(v)=\alpha(|v|)\Phi(v)$ and $\hat{\Psi}(v)=\beta(|v|)\Psi(v)$,
\end{center}
moreover, the functions $\alpha$ and $\beta$ satisfy furthermore the growth estimate
\begin{equation}
|\alpha(|v|)|+|\beta(|v|)|\lesssim 1+|v|.
\end{equation}
\end{lemma}
\begin{proof}
The proof can be justified by the similar arguments in the Propositition 6.5 of \cite{Golse and Saint-Raymond 2005} and the Section 2.4.5 of \cite{Saint-Raymond-2009-Boltzmann}.
\end{proof}

Then, we consider of the linearized Boltzmann collisional operator $\mathcal{L}_{1}$ and another linearized operator $\mathcal{L}_{2}$ defined in \eqref{linear operator}. $\mathcal{L}_{1}$ and $\mathcal{L}_{2}$ give us the dissipative structure of the kinetic equation thanks to the coercivity of $\mathcal{L}_{1}, \mathcal{L}_{2}$ (see Lemma 3.3 in \cite{GY-06} for more detail).
\begin{lemma}\label{Lmm-Coercivity-L}
	For any $f\in L^{2}(Mdv)$, there exists a  $\delta>0$ such that
	\begin{equation}
	  \l \mathcal{L}_{i}f, f \rangle_{L^2_v} \geq 2\delta \| (I-P_{i}) f \|_{L^2_v (\nu)}^2 \,.
	\end{equation}
	Moreover, there exist two positive numbers $\delta_1 , \delta_2 > 0$ such that
	\begin{equation}
	  \begin{aligned}
	    \big\l \partial^\beta_v (I-P_{i})\mathcal{L}_{i} f , \partial^\beta_v (I-P_{i}) f \big\rangle_{L^2_v} \geq 2\delta_1 \| \partial^\beta_v ( I -P_{i} ) f \|^2_{L^2_v (\nu)} - \delta_2 \sum_{\tilde{\beta}< \beta} \| \partial^{\tilde{\beta}}_v (I-P_{i}), f \|^2_{L^2_v}
	  \end{aligned}
	\end{equation}
	for all multi-indexes $\beta \in \mathbb{N}^3$, where $P_{i}(i=1,2)$ are defined \eqref{Decom}.
\end{lemma}

For the bilinear symmetric operator $Q$ defined in \eqref{linear operator}, we have the following estimates, and relevant proofs can be found in Lemma 3.3 of \cite{GY-06}.
\begin{lemma}\label{Lmm-Q}
	Let $g_i (x, v) \,, (i=1,2,3) $ be smooth functions, then we have
	\begin{equation}
	\big| \l \partial^\beta_v Q (g_1, g_2), g_3 \r_{L^2_{x,v}} \big| \lesssim \sum_{\substack{\beta_1 + \beta_2 \leq \beta }} \int_{\R^3}
	\big(\| \partial_v^{\beta_1} g_1 \|_{L^2_v (\nu)} \| \partial^{\beta_2}_v g_2 \|_{L^2_v}+\| \partial_v^{\beta_1} g_1 \|_{L^2_v} \| \partial^{\beta_2}_v g_2 \|_{L^2_v (\nu)}\big) \| g_3 \|_{L^2_v (\nu)} d x,
	\end{equation}
	for any $\beta \in \mathbb{N}^3$.
\end{lemma}

For the term $P_{1}f$, we have the following  lemma
\begin{lemma}\label{P1f}
Suppose $h(v)$ be a polynomial of $v$, then for any $|\beta|\geq 0$, we have
\begin{equation*}
\begin{aligned}
\|\partial_{v}^{\beta}P_{1}f|h(v)|\|_{L^{2}_{x,v}}\lesssim\|P_{1}f\|_{L^{2}_{x,v}}.
\end{aligned}
\end{equation*}
\begin{proof}
It is obvious true if $|\beta|=0$ or $|\beta|\geq 3$ since $P_{1}f=a(t,x)+b(t,x)\cdot v+c(t,x)|v|^{2}$, where $a(t,x)=\int_{\R^{3}} f(\frac{5}{2}-\frac{|v|^{2}}{2})Md v$, $b(t,x)=\int_{\R^{3}} fvMd v$ and $c(t,x)=\int_{\R^{3}} f(\frac{|v|^{2}}{6}-\frac{1}{2})Md v$.

We only prove the case of $|\beta|=1$ since $|\beta|=2$ is as similar as the case of $|\beta|=1$. We calculate directly that
\begin{equation*}
\begin{aligned}
\|\partial_{v}P_{1}f|h(v)|\|^{2}_{L^{2}_{x,v}}\lesssim \int_{\R^{3}}\int_{\R^{3}}|b|^{2}|h(v)|^{2}Md vd x+ \int_{\R^{3}}\int_{\R^{3}}|c|^{2}|v^{2}||h(v)|^{2}Md vd x\lesssim \|b|_{L^{2}_{x}}^{2}+\|c\|_{L^{2}_{x}}^{2},
\end{aligned}
\end{equation*}
since $\partial_{v}P_{1}f=b+2c\cdot v$ and $h(v)$ be a polynomial of $v$. As the term $\|P_{1}f\|_{L^{2}_{x,v}}^{2}$, we also have
\begin{equation*}
\begin{aligned}
\|P_{1}f\|^{2}_{L^{2}_{x,v}}&=\int_{\R^{3}}\int_{\R^{3}}|a+b\cdot v+c|v|^{2}|^{2}Md vd x=\|a\|^{2}_{L^{2}_{x}}+\|b\|^{2}_{L^{2}_{x}}+15\|c\|_{L^{2}_{x}}^{2}+3\int_{\R^{3}}acd x\\
&\geq \frac{1}{2}\|a\|_{L^{2}_{x}}^{2}+\|b\|_{L^{2}_{x}}^{2}+\frac{21}{2}\|c\|_{L^{2}_{x}}^{2},
\end{aligned}
\end{equation*}
where we mark use  Young inequality $3a c\leq\frac{1}{2}|a|^{2}+\frac{9}{2}|c|^{2}$. As a consequence, the proof is completed.
\end{proof}

\end{lemma}

\subsection{Local-in-time solution}

In this subsection, we will prove that the perturbed VPB system \eqref{VPB-g} for all $0 < \varepsilon\leq 1$ has a unique  local-in-time solution under small size of the initial data. The proof  is divided into the following three steps. The first step is to construct the approximation equation. According to reference \cite{Jiang-Xu-Zhao-2018-Indiana}, we can know that the linear approximate system has a solution for the fixed $\varepsilon\in (0,1]$, the second step is to obtain the energy estimate of the uniform bound of $\varepsilon$ of the approximate system, the third step is to obtain that the perturbed  VPB system has a local-in-time solution under small size of the initial data by compactness analysis. To simplify the estimation, we introduce a new dissipative term as
\begin{equation}\label{New D}
\tilde{\mathcal{D}}_{N,\varepsilon}(G)= \frac{1}{\varepsilon^2}\| (\mathbb{I}-\mathbb{P}) G \|^2_{H^N_{x,v}(\nu)}.
\end{equation}

Then we have the following lemma.

\begin{lemma}\label{Lmm-Local}
	There exist $ 0 < \tau \leq 1 $ and $ 0 < T \leq 1 $ such that for any $ 0 < \varepsilon \leq 1 $, $(g_{\varepsilon, 0}^{+}(x,v),g_{\varepsilon,0}^{-}(x,v))\in H^N_{x,v}, \nabla_{x}\phi_{\varepsilon,0}(x)\in H^{N}_{x}( N \geq 4)$ with $ \mathcal{E}_N (G_{\varepsilon, 0} , \phi_{\varepsilon, 0} ) \leq \tau$, the system \eqref{VPB-g}-\eqref{f's intial} admits a unique solution $(g_\varepsilon^{+},g_{\varepsilon}^{-}) \in L^\infty (0,T; H^N_{x,v}) \cap L^2(0,T; H^N_{x,v} (\nu))$ and $ \phi_\varepsilon \in L^\infty (0, T; H^{N+1}_x)$ with uniform energy bound
	\begin{equation}\label{2.3}
	  \sup_{t\in [0,T]} \mathcal{E}_N (G_\varepsilon (t) , \phi_\varepsilon (t) ) + \int_0^T \tilde{\mathcal{D}}_{N,\varepsilon}(G_\varepsilon (t)) d t \leq C,
	\end{equation}
	for some constant $C > 0$ independent of $\varepsilon$, where $\mathcal{E}_N (G_\varepsilon(t) , \phi_\varepsilon (t))$ and $\tilde{\mathcal{D}}_{N,\varepsilon}(G_\varepsilon (t))$ defined in \eqref{Energy-E} and \eqref{New D} respectively. Furthermore, $f_{\varepsilon}^{\pm}(t,x,v)=M(1+\varepsilon g_{\varepsilon}^{\pm}(t,x,v))\geq0$ .
\end{lemma}
\begin{proof}[Proof of Lemma \ref{Lmm-Local}]
	For any fixed $\varepsilon \in (0,1]$, we consider the following linear iterative approximate system \eqref{Iter-Approx-Syst} with initial data \eqref{IC-Iter-Appro}
	\begin{equation}\label{Iter-Approx-Syst}
	\left\{
	\begin{array}{l} \partial_{t}f^{n+1}_{\varepsilon}+\frac{1}{\varepsilon}v\cdot\nabla_{x}f^{n+1}_{\varepsilon}-g^{n+1}_{\varepsilon}v\cdot\nabla_{x}\phi^{n}_{\varepsilon}+\nabla_{v}g^{n+1}_{\varepsilon}\cdot\nabla_{x}\phi_{n}+\frac{1}{\varepsilon}\mathcal{L}_{1}f^{n+1}_{\varepsilon}=\frac{1}{\varepsilon}Q(f^{n}_{\varepsilon},f^{n}_{\varepsilon})\,,\\
\partial_{t}g^{n+1}_{\varepsilon}+\frac{1}{\varepsilon}v\cdot\nabla_{x}g^{n+1}_{\varepsilon}-\frac{2}{\varepsilon}v\cdot\nabla_{x}\phi^{n+1}_{\varepsilon}-f^{n+1}_{\varepsilon}v\cdot\nabla_{x}\phi^{n}_{\varepsilon}+\nabla_{v}f^{n+1}_{\varepsilon}\cdot\nabla_{x}\phi^{n}_{\varepsilon}
+\frac{1}{\varepsilon^{2}}\mathcal{L}_{2}g^{n+1}_{\varepsilon}=\frac{1}{\varepsilon}Q(g^{n}_{\varepsilon},f^{n}_{\varepsilon})\,,\\
	\Delta_x \phi^{n+1}_\varepsilon = \l 1, g^{n+1}_\varepsilon\r_{L^2_v},
	\end{array}
	\right.
	\end{equation}
	with initial data
	\begin{equation}\label{IC-Iter-Appro}
	\begin{aligned}
	G^{n+1}_{\varepsilon} |_{t=0} = G_{\varepsilon, 0}(x,v).
	\end{aligned}
	\end{equation}

	We start with $G_\varepsilon^{0} (t,x,v) =G_{\varepsilon, 0} (x,v)$ and $f_{\varepsilon}^{0,\pm}=M(1+\varepsilon g_{\varepsilon}^{0, \pm})=M(1+\varepsilon \frac{f_{\varepsilon,0}(x,v)\pm g_{\varepsilon,0}(x,v)}{2})\geq0 $  for all $t \geq 0$. The existence of $(f^{n+1}_{\varepsilon}, g^{n+1}_{\varepsilon}, \phi^{n+1}_\varepsilon)$ for the linear Cauchy problem \eqref{Iter-Approx-Syst}-\eqref{IC-Iter-Appro} is assured above by employing the standard linear theory (see \cite{Jiang-Xu-Zhao-2018-Indiana}), once given initial data $f^{n}_{\varepsilon}|_{t=0} = f_{\varepsilon,0} (x,v)$ and $g^{n}_{\varepsilon}|_{t=0} = g_{\varepsilon,0} (x,v)$.

	Next, we derive the uniform energy estimates for $\varepsilon$ of the iterative approximate system \eqref{Iter-Approx-Syst}. For simplicity, we take $f^{n+1}_{\varepsilon}, g^{n+1}_{\varepsilon}$ as $f^{n+1}, g^{n+1}$ in what follows. For $ \alpha \in \mathbb{N}^3 $ with $|\alpha| \leq N$, we first act the derivative operator $\partial_x^\alpha$ on the first $f^{n+1}, g^{n+1}$-equation of \eqref{Iter-Approx-Syst}, respectively, and take $L^2_{x,v}$-inner product with $\partial_x^\alpha f^{n+1},$ $\partial_x^\alpha g^{n+1}$,  then we gain

\begin{equation}\label{2.4}
\begin{aligned}
&\frac{1}{2}\frac{d}{dt}\big(\| \partial_x^\alpha G^{n+1} \|^2_{L^2_{x,v}}+2\|\partial_{x}^{\alpha}\nabla_{x}\phi^{n+1}\|^{2}_{L^{2}_{x}}\big)+\frac{\delta}{\varepsilon^{2}}\| \partial_x^\alpha (\mathbb{I}-\mathbb{P})G^{n+1} \|_{L^2_{x,v}(\nu)}^2\\
&\leq\underbrace{-\l\nabla_{x}\phi^{n}\cdot\nabla_{v}\partial_{x}^{\alpha}g^{n+1}\cdot\partial_{x}^{\alpha}f^{n+1}+\nabla_{x}\phi^{n}\cdot\nabla_{v}\partial_{x}^{\alpha}f^{n+1}\cdot\partial_{x}^{\alpha}g^{n+1},1\r_{L^{2}_{x,v}}}_{D_{1}}\\
&\underbrace{-\sum_{0\neq\tilde{\alpha}\leq\alpha}\l\partial_{x}^{\tilde{\alpha}}\nabla_{x}\phi^{n}\cdot\nabla_{v}\partial_{x}^{\alpha-\tilde{\alpha}}g^{n+1}\cdot\partial_{x}^{\alpha}f^{n+1}
+\partial_{x}^{\tilde{\alpha}}\nabla_{x}\phi^{n}\cdot\nabla_{v}\partial_{x}^{\alpha-\tilde{\alpha}}f^{n+1}\cdot\partial_{x}^{\alpha}g^{n+1},1\r_{L^{2}_{x,v}}}_{_{D_{2}}}\\
&+\underbrace{\l\partial_{x}^{\alpha}(g^{n+1}v\cdot\nabla_{x}\phi^{n})\cdot\partial_{x}^{\alpha}f^{n+1}+\partial_{x}^{\alpha}(f^{n+1}v\cdot\nabla_{x}\phi^{n})\cdot\partial_{x}^{\alpha}g^{n+1},1\r_{L^{2}_{x,v}}}_{D_{3}}\\
&+\underbrace{\frac{1}{\varepsilon}\l \partial_{x}^{\alpha}Q(f^{n},f^{n}),\partial_{x}^{\alpha}f^{n+1}\r_{L^{2}_{x,v}}}_{D_{4}}+\underbrace{\frac{1}{\varepsilon}\l \partial_{x}^{\alpha}Q(g^{n},f^{n}),\partial_{x}^{\alpha}g^{n+1}\r_{L^{2}_{x,v}}}_{D_{5}}\\
\end{aligned}
\end{equation}
where we make use of Lemma \ref{Lmm-Coercivity-L}, $\partial_{t}(a^{n+1}+3c^{n+1})+\frac{1}{\varepsilon}\nabla_{x}\cdot b^{n+1}=0$ and Poisson equation of $\Delta_{x}\phi^{n+1}=\int g^{n+1}Md v$. 	

Now we estimate the secondary term $D_{1}$.
\begin{equation}
\begin{aligned}\label{APP2}
D_{1}=&-\l v\nabla_{x}\cdot\phi^{n},\partial_{x}^{\alpha}f^{n+1}\cdot\partial_{x}^{\alpha}g^{n+1}\r_{L^{2}_{x,v}}\\
=&\underbrace{-\l v\cdot\nabla_{x}\phi^{n},\partial_{x}^{\alpha}P_{1}f^{n+1}\cdot\partial_{x}^{\alpha}P_{2}g^{n+1}\r_{L^{2}_{x,v}}}_{D_{11}}\underbrace{-\l v\cdot\nabla_{x}\phi^{n},\partial_{x}^{\alpha}P_{1}f^{n+1}\cdot\partial_{x}^{\alpha}(I-P_{2})g^{n+1}\r_{L^{2}_{x,v}}}_{D_{12}}\\
&\underbrace{-\l v\cdot\nabla_{x}\phi^{n},\partial_{x}^{\alpha}(I-P_{1})f^{n+1}\cdot\partial_{x}^{\alpha}P_{2}g^{n+1}\r_{L^{2}_{x,v}}}_{D_{13}}\underbrace{-\l v\cdot\nabla_{x}\phi^{n},\partial_{x}^{\alpha}(I-P_{1})f^{n+1}\cdot\partial_{x}^{\alpha}(I-P_{2})g^{n+1}\r_{L^{2}_{x,v}}}_{D_{14}},
\end{aligned}
\end{equation}
where we make use  integral by part, the decomposition of $G=\mathbb{P}G+(\mathbb{I}-\mathbb{P})G$.

We have estimate the term $D_{11}$ by $N\geq 4$, the H\"older inequality,  the Sobolev embedding $H^2_x \hookrightarrow L^\infty_x$ and $\|\partial_{x}^{\alpha}P_{1}f\|_{L^{2}_{x,v}(\nu)}\lesssim\|\partial_{x}^{\alpha}f\|_{L^{2}_{x,v}}, \|\partial_{x}^{\alpha}P_{2}g\|_{L^{2}_{x}}\lesssim \|\partial_{x}^{\alpha}g\|_{L_{x,v}^{2}}$ that
\begin{equation*}
\begin{aligned}
D_{11}&=-\l v \cdot \nabla \phi^n, \partial_{x}^{\alpha}P_{2}g^{n+1}\partial_{x}^{\alpha}P_{1}f^{n+1} \r_{L^2_{x,v}} \leq\int_{\mathbb{R}^{3}}|\nabla_{x}\phi^{n}||\partial_{x}^{\alpha}P_{2}g^{n+1}|\|\partial_{x}^{\alpha}P_{1}f^{n+1}\|_{L^{2}_{v}}\|v\|_{L_{v}^{2}}dx\\
&\lesssim \|\nabla_{x}\phi^{n}\|_{L^{\infty}_{x}}\|\partial_{x}^{\alpha}P_{2}g^{n+1}\|_{L_{x}^{2}}\|\partial_{x}^{\alpha}P_{1}f^{n+1}\|_{L_{x,v}^{2}}\lesssim \|\nabla_{x}\phi^{n}\|_{H^{N}_{x}}\|g^{n+1}\|_{H^{N}_{x}L^{2}_{v}}\|f^{n+1}\|_{H^{N}_{x}L^{2}_{v}}\\
&\lesssim\mathcal{E}_{N}^{\frac{1}{2}}(G^{n},\phi^{n})\mathcal{E}_{N}(G^{n+1},\phi^{n+1}).
\end{aligned}
\end{equation*}

For the quantity $D_{12}$, we derive that
	\begin{equation*}
	  \begin{split}
	    D_{12} &\lesssim\|\nabla_{x}\phi^{n}\|_{H^{N}_{x}}\|\partial_{x}^{\alpha}(I-P_{2})g^{n+1}\|_{L_{x,v}^{2}(\nu)}\|\partial_{x}^{\alpha}P_{1}f^{n+1}\|_{L_{x,v}^{2}(\nu)}\\
&\lesssim\|\nabla_{x}\phi^{n}\|_{H^{N}_{x}}\|(I-P_{2})g^{n+1}\|_{H_{x}^{N}L_{v}^{2}(\nu)}\|f^{n+1}\|_{H^{N}_{x}L_{v}^{2}}\\
&\lesssim\varepsilon\mathcal{E}_{N}^{\frac{1}{2}}(G^{n},\phi^{n})\mathcal{E}_{N}^{\frac{1}{2}}(G^{n+1},\phi^{n+1})\tilde{\mathcal{D}}_{N,\varepsilon}^{\frac{1}{2}}(G^{n+1}),
	  \end{split}
	\end{equation*}
 where we make use of the Sobolev embedding $H^2_x \hookrightarrow L^{\infty}_x$ and \eqref{nu}.

	By a similar argument we can estimate $D_{13}$ and $D_{14}$ as
	\begin{equation*}
	  \begin{aligned}   D_{13}&\lesssim\|\nabla_{x}\phi^{n}\|_{H^{N}_{x}}\|g^{n+1}\|_{H^{N}_{x}L^{2}_{v}}\|(I-P_{1})f^{n+1}\|_{H^{N}_{x}L_{v}^{2}(\nu)}\lesssim\varepsilon\mathcal{E}_{N}^{\frac{1}{2}}(G^{n},\phi^{n})\mathcal{E}_{N}(G^{n+1},\phi^{n+1}),\\
D_{14}&\lesssim\|\nabla_{x}\phi^{n}\|_{H^{N}_{x}}\|(I-P_{1})g^{n+1}\|_{H^{N}_{x}L^{2}_{v}(\nu)}\|(I-P_{1})f^{n+1}\|_{H^{N}_{x}L_{v}^{2}(\nu)}\lesssim\varepsilon^{2}\mathcal{E}_{N}^{\frac{1}{2}}(G^{n},\phi^{n})\tilde{\mathcal{D}}_{N,\varepsilon}(G^{n+1}).
	  \end{aligned}
	\end{equation*}

	Thus, we thereby estimate $D_{1}$ following
	\begin{equation}\label{I2-alph1}
	  \begin{aligned}   D_{1}&\lesssim\mathcal{E}_{N}^{\frac{1}{2}}(G^{n},\phi^{n})\mathcal{E}_{N}(G^{n+1},\phi^{n+1})+\mathcal{E}_{N}^{\frac{1}{2}}(G^{n},\phi^{n})\tilde{\mathcal{D}}_{N,\varepsilon}(G^{n+1})\\
&+\mathcal{E}_{N}^{\frac{1}{2}}(G^{n},\phi^{n})\mathcal{E}_{N}^{\frac{1}{2}}(G^{n+1},\phi^{n+1})\tilde{\mathcal{D}}_{N,\varepsilon}^{\frac{1}{2}}(G^{n+1},\phi^{n+1}),
	   \end{aligned}
	\end{equation}
	where we make use of $0<\varepsilon\leq1$.

Next we estimate the term $D_{2}$ as
\begin{equation*}
\begin{aligned}
D_{2}=\underbrace{-\sum_{0\neq\tilde{\alpha}\leq\alpha}\l\partial_{x}^{\tilde{\alpha}}\nabla_{x}\phi^{n}\cdot\nabla_{v}\partial_{x}^{\alpha-\tilde{\alpha}}g^{n+1},\partial_{x}^{\alpha}\r_{L_{x,v}^{2}}}_{D_{21}}
\underbrace{-\sum_{0\neq\tilde{\alpha}\leq\alpha}\l\partial_{x}^{\tilde{\alpha}}\nabla_{x}\phi^{n}\cdot\nabla_{v}\partial_{x}^{\alpha-\tilde{\alpha}}f^{n+1},\partial_{x}^{\alpha}g^{n+1}\r_{L^{2}_{x,v}}}_{D_{22}}.
\end{aligned}
\end{equation*}

If $0<|\tilde{\alpha}|\leq|\alpha|-2$, we estimate $D_{21}$ as
\begin{equation*}
\begin{aligned}
D_{21}&\lesssim\sum_{0\neq\tilde{\alpha}\leq\alpha}\|\partial_{x}^{\tilde{\alpha}}\nabla_{x}\phi^{n}\|_{L_{x}^{\infty}}\|\partial_{x}^{\tilde{\alpha}-\alpha}\nabla_{v}(I-P_{2})g^{n+1}\|_{L_{x,v}^{2}}\|\partial_{x}^{\alpha}f^{n+1}\|_{L^{2}_{x,v}}\\
&\lesssim\|\nabla_{x}\phi^{n}\|_{H^{N}_{x}}\|(I-P_{2})g^{n+1}\|_{H^{N}_{x,v}(\nu)}\|f^{n+1}\|_{H^{N}_{x}L^{2}_{v}}\\
&\lesssim\varepsilon\mathcal{E}_{N}^{\frac{1}{2}}(G^{n},\phi^{n})\mathcal{E}_{N}^{\frac{1}{2}}(G^{n+1},\phi^{n+1})\tilde{\mathcal{D}}_{N,\varepsilon}^{\frac{1}{2}}(G^{n+1}),
\end{aligned}
\end{equation*}
where we use \eqref{nu}, $P_{2}g(t,x,v)=d(t,x)$ and the Sobolev embedding $H^2_x \hookrightarrow L^{\infty}_x$.

If $\alpha|-1\leq|\tilde{\alpha}|\leq|\alpha|$, we estimate $D_{21}$ as
\begin{equation*}
\begin{aligned}
D_{21}&\lesssim\sum_{0\neq\tilde{\alpha}\leq\alpha}\|\partial_{x}^{\tilde{\alpha}}\nabla_{x}\phi^{n}\|_{L_{x}^{4}}\|\partial_{x}^{\tilde{\alpha}-\alpha}\nabla_{v}(I-P_{2})g^{n+1}\|_{L_{x}^{4}L_{v}^{2}}\|\partial_{x}^{\alpha}f^{n+1}\|_{L^{2}_{x,v}}\\
&\lesssim\|\nabla_{x}\phi^{n}\|_{H^{N}_{x}}\|(I-P_{2})g^{n+1}\|_{H^{N}_{x,v}(\nu(v))}\|f^{n+1}\|_{H^{N}_{x}L^{2}_{v}}\\
&\lesssim\varepsilon\mathcal{E}_{N}^{\frac{1}{2}}(f^{n},g^{n},\phi^{n})\mathcal{E}_{N}^{\frac{1}{2}}(f^{n+1},g^{n+1},\phi^{n+1})\tilde{\mathcal{D}}_{N,\varepsilon}^{\frac{1}{2}}(f^{n+1},g^{n+1}),
\end{aligned}
\end{equation*}
where we use \eqref{nu}, $P_{2}g(t,x,v)=d(t,x)$ and the Sobolev embedding $H^1_x \hookrightarrow L^{4}_x$.

In a result, we have
\begin{equation*}
D_{21}\lesssim\varepsilon\mathcal{E}_{N}^{\frac{1}{2}}(G^{n},\phi^{n})\mathcal{E}_{N}^{\frac{1}{2}}(G^{n+1},\phi^{n+1})\tilde{\mathcal{D}}_{N,\varepsilon}^{\frac{1}{2}}(G^{n+1}).
\end{equation*}

Here, we estimate $D_{22}$ as follow
\begin{equation*}
\begin{aligned}
D_{22}&=-\sum_{0\neq\tilde{\alpha}\leq\alpha}\l\partial_{x}^{\tilde{\alpha}}\nabla_{x}\phi^{n}\cdot\nabla_{v}\partial_{x}^{\alpha-\tilde{\alpha}}P_{1}f^{n+1}, \partial_{x}^{\alpha}P_{2}g^{n+1}\r_{L^{2}_{x,v}}\\
&-\sum_{0\neq\tilde{\alpha}\leq\alpha}\l\partial_{x}^{\tilde{\alpha}}\nabla_{x}\phi^{n}\cdot\nabla_{v}\partial_{x}^{\alpha-\tilde{\alpha}}P_{1}f^{n+1}, \partial_{x}^{\alpha}(I-P_{2})g^{n+1}\r_{L^{2}_{x,v}}\\
&-\sum_{0\neq\tilde{\alpha}\leq\alpha}\l\partial_{x}^{\tilde{\alpha}}\nabla_{x}\phi^{n}\cdot\nabla_{v}\partial_{x}^{\alpha-\tilde{\alpha}}(I-P_{1})f^{n+1}, \partial_{x}^{\alpha}P_{2}g^{n+1}\r_{L^{2}_{x,v}}\\
&-\sum_{0\neq\tilde{\alpha}\leq\alpha}\l\partial_{x}^{\tilde{\alpha}}\nabla_{x}\phi^{n}\cdot\nabla_{v}\partial_{x}^{\alpha-\tilde{\alpha}}(I-P_{1})f^{n+1}, \partial_{x}^{\alpha}(I-P_{2})g^{n+1}\r_{L^{2}_{x,v}}.
\end{aligned}
\end{equation*}

Similar argument as $D_{1}$, we have
\begin{equation*}
	  \begin{aligned}   D_{22}&\lesssim\mathcal{E}_{N}^{\frac{1}{2}}(G^{n},\phi^{n})\mathcal{E}_{N}(G^{n+1},\phi^{n+1})+\mathcal{E}_{N}^{\frac{1}{2}}(G^{n},\phi^{n})\tilde{\mathcal{D}}_{N,\varepsilon}(G^{n+1})\\
&+\mathcal{E}_{N}^{\frac{1}{2}}(G^{n},\phi^{n})\mathcal{E}_{N}^{\frac{1}{2}}(G^{n+1},\phi^{n+1})\tilde{\mathcal{D}}_{N,\varepsilon}^{\frac{1}{2}}(G^{n+1}).
	   \end{aligned}
	\end{equation*}

Consequently, we obtain
\begin{equation}\label{I2-alph2}
\begin{aligned}   D_{2}&\lesssim\mathcal{E}_{N}^{\frac{1}{2}}(G^{n},\phi^{n})\mathcal{E}_{N}(G^{n+1},\phi^{n+1})+\mathcal{E}_{N}^{\frac{1}{2}}(f^{n},g^{n},\phi^{n})\tilde{\mathcal{D}}_{N,\varepsilon}(G^{n+1})\\
&+\mathcal{E}_{N}^{\frac{1}{2}}(G^{n},\phi^{n})\mathcal{E}_{N}^{\frac{1}{2}}(G^{n+1},\phi^{n+1})\tilde{\mathcal{D}}_{N,\varepsilon}^{\frac{1}{2}}(G^{n+1}).
	   \end{aligned}
\end{equation}

By the same argument as $D_{1}$, the quality $D_{3}$ can estimated as
\begin{equation}\label{I2-alph3}
\begin{aligned}   D_{3}&\lesssim\mathcal{E}_{N}^{\frac{1}{2}}(G^{n},\phi^{n})\mathcal{E}_{N}(G^{n+1},\phi^{n+1})+\mathcal{E}_{N}^{\frac{1}{2}}(f^{n},g^{n},\phi^{n})\tilde{\mathcal{D}}_{N,\varepsilon}(G^{n+1})\\
&+\mathcal{E}_{N}^{\frac{1}{2}}(G^{n},\phi^{n})\mathcal{E}_{N}^{\frac{1}{2}}(G^{n+1},\phi^{n+1})\tilde{\mathcal{D}}_{N,\varepsilon}^{\frac{1}{2}}(G^{n+1}).
	   \end{aligned}
\end{equation}

Next, we estimate $D_{4}$ and $D_{5}$
\begin{equation*}
\begin{aligned}
D_{4}&=\underbrace{\frac{1}{\varepsilon}\l \partial_{x}^{\alpha}Q(P_{1}f^{n},P_{1}f^{n}),\partial_{x}^{\alpha}(I-P_{1})f^{n+1}\r_{L_{x,v}^{2}}}_{D_{41}}
+\underbrace{\frac{1}{\varepsilon}\l \partial_{x}^{\alpha}Q((I-P_{1})f^{n},P_{1}f^{n}),\partial_{x}^{\alpha}(I-P_{1})f^{n+1}\r_{L_{x,v}^{2}}}_{D_{42}}\\
&+\underbrace{\frac{1}{\varepsilon}\l \partial_{x}^{\alpha}Q(P_{1}f^{n},(I-P_{1})f^{n}),\partial_{x}^{\alpha}(I-P_{1})f^{n+1}\r_{L_{x,v}^{2}}}_{D_{43}}+\underbrace{\frac{1}{\varepsilon}\l \partial_{x}^{\alpha}Q((I-P_{1})f^{n},(I-P_{1})f^{n}),\partial_{x}^{\alpha}(I-P_{1})f^{n+1}\r_{L_{x,v}^{2}}}_{D_{44}},
\end{aligned}
\end{equation*}
where we use the fact $\l Q(f,f),1\r_{L_{v}^{2}}=\l Q(f,f),v\r_{L_{v}^{2}}=\l Q(f,f),|v|^{2}\r_{L_{v}^{2}}=0$.

We estimate $D_{41}$ as follow
\begin{equation*}
\begin{aligned}
D_{41}&\lesssim\frac{1}{\varepsilon}\sum_{\tilde{\alpha}\leq\alpha}\int\|\partial_{x}^{\tilde{\alpha}}P_{1}f^{n}\|_{L^{2}_{v}}\|\partial_{x}^{\alpha-\tilde{\alpha}}P_{1}f^{n}\|_{L_{v}^{2}(\nu)}\|\partial_{x}^{\alpha}(I-P_{1})f^{n+1}\|_{L^{2}_{v}(\nu)}dx\\
&+\frac{1}{\varepsilon}\sum_{\tilde{\alpha}\leq\alpha}\int\|\partial_{x}^{\tilde{\alpha}}P_{1}f^{n}\|_{L^{2}_{v}(\nu)}\|\partial_{x}^{\alpha-\tilde{\alpha}}P_{1}f^{n}\|_{L_{v}^{2}}\|\partial_{x}^{\alpha}(I-P_{1})f^{n+1}\|_{L^{2}_{v}(\nu)}dx.
\end{aligned}
\end{equation*}

By the same argument as $D_{21}$, we have
\begin{equation*}
\begin{aligned}
D_{41}&\lesssim\frac{1}{\varepsilon}\|f^{n}\|_{H^{N}_{x}L^{2}_{v}}^{2}\|(I-P_{1})f^{n+1}\|_{H_{x,v}(\nu)}\lesssim\varepsilon\mathcal{E}_{N}(G^{n},\phi^{n})\tilde{\mathcal{D}}_{N,\varepsilon}^{\frac{1}{2}}(G^{n+1}),
\end{aligned}
\end{equation*}
where we use the fact $\|P_{1}f\|_{L^{2}_{x,v}(\nu)}\lesssim\|f\|_{L_{x,v}^{2}}$ and \eqref{nu}. Analogously, we have
\begin{equation*}
\begin{aligned}
D_{42}+D_{43}+D_{44}&\lesssim\varepsilon\mathcal{E}_{N}^{\frac{1}{2}}(G^{n},\phi^{n})\tilde{\mathcal{D}}_{N,\varepsilon}^{\frac{1}{2}}(G^{n})\tilde{\mathcal{D}}_{N,\varepsilon}^{\frac{1}{2}}(G^{n+1}),
\end{aligned}
\end{equation*}
where we use the fact $\|(I-P_{1})f\|_{L_{x,v}^{2}}\lesssim\|f\|_{L^{2}_{x,v}}$.

Consequently, we have
\begin{equation}\label{I2-alph4}
\begin{aligned}
D_{4}&\lesssim\mathcal{E}_{N}(G^{n},\phi^{n})\tilde{\mathcal{D}}_{N,\varepsilon}^{\frac{1}{2}}(G^{n+1})+\mathcal{E}_{N}^{\frac{1}{2}}(G^{n},\phi^{n})\tilde{\mathcal{D}}_{N,\varepsilon}^{\frac{1}{2}}(G^{n})\tilde{\mathcal{D}}_{N,\varepsilon}^{\frac{1}{2}}(G^{n+1}),
\end{aligned}
\end{equation}
where we mark use of $0<\varepsilon\leq1$.

To apply the fact that $P_{2}g(t,x,v)=d(t,x)$ and $\l Q(f,g),1\r_{L^{2}_{v}}$, we have the same estimate like $D_{4}$ as
\begin{equation}\label{I2-alph5}
\begin{aligned}
D_{5}&\lesssim\mathcal{E}_{N}(G^{n},\phi^{n})\tilde{\mathcal{D}}_{N,\varepsilon}^{\frac{1}{2}}(G^{n+1})+\mathcal{E}_{N}^{\frac{1}{2}}(G^{n},\phi^{n})\tilde{\mathcal{D}}_{N,\varepsilon}^{\frac{1}{2}}(G^{n})\tilde{\mathcal{D}}_{N,\varepsilon}^{\frac{1}{2}}(G^{n+1}).
\end{aligned}
\end{equation}

From plugging the inequalities \eqref{I2-alph1}, \eqref{I2-alph2}, \eqref{I2-alph3}, \eqref{I2-alph4} and \eqref{I2-alph5} to \eqref{2.4}, we have the estimate as follow
\begin{equation}\label{Iter-Spatial-Bnd}
\begin{aligned}
&\frac{1}{2}\frac{d}{dt}\big(\| \partial_x^\alpha G^{n+1} \|^2_{L^2_{x,v}}+2\|\partial_{x}^{\alpha}\nabla_{x}\phi^{n+1}\|^{2}_{L^{2}_{x}}\big)+\frac{2\delta}{\varepsilon^{2}}\| \partial_x^\alpha(\mathbb{I}-\mathbb{P})G^{n+1} \|_{L^2_{x,v}(\nu)}^2\\
&\lesssim\mathcal{E}_{N}^{\frac{1}{2}}(G^{n},\phi^{n})\mathcal{E}_{N}(G^{n+1},\phi^{n+1})+\mathcal{E}_{N}^{\frac{1}{2}}(G^{n},\phi^{n})\tilde{\mathcal{D}}_{N,\varepsilon}(G^{n+1})+\mathcal{E}_{N}^{\frac{1}{2}}(G^{n},\phi^{n})\tilde{\mathcal{D}}_{N,\varepsilon}^{\frac{1}{2}}(f^{n},g^{n})\tilde{\mathcal{D}}_{N,\varepsilon}^{\frac{1}{2}}(G^{n+1})\\
&+\mathcal{E}_{N}^{\frac{1}{2}}(G^{n},\phi^{n})\mathcal{E}_{N}^{\frac{1}{2}}(G^{n+1},\phi^{n+1})\tilde{\mathcal{D}}_{N,\varepsilon}^{\frac{1}{2}}(G^{n+1},\phi^{n+1})+\mathcal{E}_{N}(G^{n},\phi^{n})\tilde{\mathcal{D}}_{N,\varepsilon}^{\frac{1}{2}}(G^{n+1}).
\end{aligned}
\end{equation}

Next we need to control the terms $\| \nabla_v (\mathbb{I}-\mathbb{P})G^{n+1} \|_{H^{N-1}_x L^2_v}$ in \eqref{Iter-Spatial-Bnd}. By applying the microscopic projection $(I-P_{1}),(I-P_{2})$ to the $f^{n+1}$, $g^{n+1}$-equation of \eqref{Iter-Approx-Syst}, respectively , we have device the microscopic evolution equation
	\begin{equation}\label{2.81}
	  \begin{aligned}
	     \partial_t (I-P_{1}) f^{n+1} + \frac{1}{\varepsilon^2} \mathcal{L}_{1} (I-P_{1}) f^{n+1} =&-\frac{1}{\varepsilon}(I-P_{1})(v\cdot\nabla_{x}f^{n+1})+(I-P_{1})(g^{n+1}v\cdot\nabla_{x}\phi^{n})\\
&-(I-P_{1})(\nabla_{x}\phi^{n}\cdot\nabla_{v}(I-P_{2})g^{n+1})+\frac{1}{\varepsilon}Q(f^{n},f^{n}),
	  \end{aligned}
	\end{equation}
\begin{equation}\label{2.82}
	  \begin{aligned}
	     \partial_t (I-P_{2}) g^{n+1}  + \frac{1}{\varepsilon^2}\mathcal{L}_{2} (I-P_{2}) g^{n+1}=&-\frac{1}{\varepsilon}(I-P_{2})(v\cdot\nabla_{x}g^{n+1})+(I-P_{2})(f^{n+1}v\cdot\nabla_{x}\phi^{n})\\
&-(I-P_{2})(\nabla_{x}\phi^{n}\cdot\nabla_{v}f^{n+1})+\frac{1}{\varepsilon}Q(g^{n},f^{n}),
	  \end{aligned}
	\end{equation}
	where we make use of the relations
	$\nabla_{v}P_{2}g=\nabla_{v}d(t,x)= (I-P_{2})(v\cdot\nabla_{x}\phi^{n})=0.$

	For any fixed $\alpha \,, \beta \in  \mathbb{N}^3$ with $ \beta \neq 0$ and $|\alpha|+|\beta|\leq N (N\geq 4)$, we first take mixed derivative operator $ \partial_x^\alpha \partial_v^\beta$ on \eqref{2.81} and \eqref{2.82}, taking $L^2_{x,v}$-inner product by dot with $\partial_x^\alpha \partial_v^\beta (I-P_{1}) f^{n+1}$ and $\partial_x^\alpha \partial_v^\beta (I-P_{2}) g^{n+1}$, respectively, integrating by parts over $\mathbb{R}^3 \times \mathbb{R}^3 $ and sum up over $f^{n+1}$ and $g^{n+1}$. Then, there exist two constants $\delta_1, \delta_2 > 0$ such that
	\begin{equation}\label{2.9}
	  \begin{aligned}
	    & \frac{1}{2} \frac{d}{d t} \| \partial_x^\alpha \partial_v^\beta (\mathbb{I}-\mathbb{P})G^{n+1} \|^2_{L^2_{x,v}}+ \frac{2\delta_1}{\varepsilon^2} \| \partial_x^\alpha \partial_v^\beta (\mathbb{I}-\mathbb{P}) G^{n+1} \|_{L^2_{x,v}(\nu)}^2 - \frac{\delta_2}{\varepsilon^2} \sum_{\tilde{\beta}< \beta} \| \partial_x^\alpha \partial_v^{\tilde{\beta}} (\mathbb{I}-\mathbb{P})G^{n+1} \|_{L^2_{x,v}}^2\\
&\lesssim\underbrace{-\frac{1}{\varepsilon}\l\partial_{x}^{\alpha}\partial_{v}^{\beta}(I-P_{1})(v\cdot\nabla_{x}f^{n+1}),\partial_{x}^{\alpha}\partial_{v}^{\beta}(I-P_{1})f^{n+1}\r_{L^{2}_{x,v}}}_{S_{1}}\\
&\underbrace{-\frac{1}{\varepsilon}\l\partial_{x}^{\alpha}\partial_{v}^{\beta}(v\cdot\nabla_{x}g^{n+1}),\partial_{x}^{\alpha}\partial_{v}^{\beta}(I-P_{2})g^{n+1}\r_{L^{2}_{x,v}}}_{S_{2}}\\
&+\underbrace{\l\partial_{x}^{\alpha}\partial_{v}^{\beta}(I-P_{1})(g^{n+1}v\cdot\nabla_{x}\phi^{n}-\nabla_{x}\phi^{n}\cdot\nabla_{v}(I-P_{2})g^{n+1}),\partial_{x}^{\alpha}\partial_{v}^{\beta}(I-P_{1})f^{n+1}\r_{L^{2}_{x,v}}}_{S_{3}}\\
&+\underbrace{\l\partial_{x}^{\alpha}\partial_{v}^{\beta}(f^{n+1}v\cdot\nabla_{x}\phi^{n}-\nabla_{x}\phi^{n}\cdot\nabla_{v}f^{n+1}),\partial_{x}^{\alpha}\partial_{v}^{\beta}(I-P_{2})g^{n+1}\r_{L^{2}_{x,v}}}_{S_{4}}\\
&+\underbrace{\frac{1}{\varepsilon}\l\partial_{x}^{\alpha}\partial_{v}^{\beta}Q(f^{n},f^{n}),\partial_{x}^{\alpha}\partial_{v}^{\beta}(I-P_{1})f^{n+1}\r_{L^{2}_{x,v}}}_{S_{5}}
+\underbrace{\frac{1}{\varepsilon}\l\partial_{x}^{\alpha}\partial_{v}^{\beta}Q(g^{n}, f^{n}),\partial_{x}^{\alpha}\partial_{v}^{\beta}(I-P_{2})g^{n+1}\r_{L^{2}_{x,v}}}_{S_{6}},
	  \end{aligned}
	\end{equation}
	where Lemma \ref{Lmm-Coercivity-L} are used.

	Now we estimate terms $S_i (1 \leq i \leq 6)$ in \eqref{2.9}. We divided the term $S_{1}$ into three parts as follow
	\begin{equation*}
	  \begin{split} S_{1}&=\underbrace{\frac{1}{\varepsilon}\l\partial_{x}^{\alpha}\partial_{v}^{\beta}P_{1}(v\cdot\nabla_{x}f^{n+1}), \partial_{x}^{\alpha}\partial_{v}^{\beta}(I-P_{1})f^{n+1}}_{S_{11}}
\underbrace{-\frac{1}{\varepsilon}\l\partial_{x}^{\alpha}\partial_{v}^{\beta}(v\cdot\nabla_{x}P_{1}f^{n+1}),\partial_{x}^{\alpha}\partial_{v}^{\beta}(I-P_{1})f^{n+1}\r_{L^2_{x,v}}}_{S_{12}}\\
&\underbrace{-\frac{1}{\varepsilon}\l\partial_{x}^{\alpha}\partial_{v}^{\beta}(v\cdot\nabla_{x}(I-P_{1})f^{n+1}), \partial_{x}^{\alpha}\partial_{v}^{\beta}(I-P_{1})f^{n+1}\r_{L^2_{x,v}}}_{S_{13}}.
	  \end{split}
	\end{equation*}

We can estimate term $S_{11}$ by H\"older inequality, $\|\partial_{v}^{\beta}P_{1}f\|_{L^{2}_{x,v}(\nu)}\lesssim\|f\|_{L^{2}_{x,v}}$, $|\beta|\geq 1$, Young inequality, $0<\varepsilon\leq 1$ and $\nu(v) \sim 1+|v|$
\begin{equation*}
\begin{aligned}
S_{11}&=\frac{1}{\varepsilon}\l\partial_{x}^{\alpha}\partial_{v}^{\beta}P_{1}(v\cdot\nabla_{x}P_{1}f^{n+1}),\partial_{x}^{\alpha}\partial_{v}^{\beta}(I-P_{1})f^{n+1}\r_{L^{2}_{x,v}}+\frac{1}{\varepsilon}\l\partial_{x}^{\alpha}\partial_{v}^{\beta}P_{1}(v\cdot\nabla_{x}(I-P_{1})f^{n+1}),\partial_{x}^{\alpha}\partial_{v}^{\beta}(I-P_{1})f^{n+1}\r_{L^{2}_{x,v}}\\
&\lesssim\frac{1}{\varepsilon}\|\nabla_{x}\partial_{x}^{\alpha}(I-P_{1})f^{n+1}\|_{L^{2}_{x,v}(\nu)}\|\partial_{x}^{\alpha}\partial_{v}^{\beta}(I-P_{1})f^{n+1}\|_{L^{2}_{x,v}(\nu)}+\frac{1}{\varepsilon}\|\nabla_{x}\partial_{x}^{\alpha}f^{n+1}\|_{L^{2}_{x,v}}\|\partial_{x}^{\alpha}\partial_{v}^{\beta}(I-P_{1})f^{n+1}\|_{L^{2}_{x,v}(\nu)}\\
&\leq\frac{\delta_{1}}{16\varepsilon^{2}}\|\partial_{x}^{\alpha}\partial_{v}^{\beta}(I-P_{1})f^{n+1}\|_{L^{2}_{x,v}(\nu)}^{2}+\frac{\delta}{16\varepsilon^{2}}\|\partial_{x}^{\alpha+1}(I-P_{1})f^{n+1}\|_{L^{2}_{x,v}(\nu)}^{2}+C\mathcal{E}_{N}^{\frac{1}{2}}(G^{n+1},\phi^{n+1})\tilde{\mathcal{D}}_{N,\varepsilon}^{\frac{1}{2}}(G^{n+1}).
\end{aligned}
\end{equation*}

Analogously, we can also estimate $S_{12}, S_{13}$ following
\begin{equation*}
\begin{aligned}
S_{12}&=-\frac{1}{\varepsilon}\l v\cdot\nabla_{x}\partial_{x}^{\alpha}\partial_{v}^{\beta}P_{1}f^{n+1},\partial_{x}^{\alpha}\partial_{v}^{\beta}(I-P_{1})f^{n+1}\r_{L^{2}_{x,v}}\\
&-\frac{1}{\varepsilon}\l \nabla_{x}\partial_{x}^{\alpha}\partial_{v}^{\beta-1}P_{1}f^{n+1},\partial_{x}^{\alpha}\partial_{v}^{\beta}(I-P_{1})f^{n+1}\r_{L^{2}_{x,v}}\lesssim\mathcal{E}_{N}^{\frac{1}{2}}(G^{n+1},\phi^{n+1})\tilde{\mathcal{D}}_{N,\varepsilon}^{\frac{1}{2}}(G^{n+1}),\\
S_{13}&=-\frac{1}{\varepsilon}\l\nabla_{x}\partial_{x}^{\alpha}\partial_{v}^{\beta-1}(I-P_{1})f^{n+1},\partial_{x}^{\alpha}\partial_{v}^{\beta}(I-P_{1})f^{n+1}\r_{L^{2}_{x,v}}\\
&\lesssim\frac{1}{\varepsilon}\|f^{n+1}\|_{H_{x}^{N}L^{2}_{v}}\|(I-P_{1})f^{n+1}\|_{H^{N}_{x,v}(\nu)}\lesssim\mathcal{E}_{N}^{\frac{1}{2}}(G^{n+1},\phi^{n+1})\tilde{\mathcal{D}}_{N,\varepsilon}^{\frac{1}{2}}(G^{n+1}),
\end{aligned}
\end{equation*}
where we use $\l v\cdot\nabla_{x}\partial_{x}^{\alpha}\partial_{v}^{\beta}(I-P_{1})f^{n+1},\partial_{x}^{\alpha}\partial_{v}^{\beta}(I-P_{1})f^{n+1}\r_{L^{2}_{x,v}}=0$ and $\nu(v)\sim1+|v|$. Therefore, we have estimate $S_{1}$ as
\begin{equation}\label{S1}
\begin{aligned}
S_{1}&-\frac{\delta_{1}}{16\varepsilon^{2}}\|\partial_{x}^{\alpha}\partial_{v}^{\beta}(I-P_{1})f^{n+1}\|_{L^{2}_{x,v}(\nu)}^{2}-\frac{\delta}{16\varepsilon^{2}}\|\partial_{x}^{\alpha+1}(I-P_{1})f^{n+1}\|_{L^{2}_{x,v}(\nu)}^{2}\\
&\lesssim\mathcal{E}_{N}^{\frac{1}{2}}(G^{n+1},\phi^{n+1})\tilde{\mathcal{D}}_{N,\varepsilon}^{\frac{1}{2}}(G^{n+1}).
\end{aligned}
\end{equation}

By the similar argument as $S_{1}$, we can estimate $S_{2}$ as
\begin{equation}\label{S2}
\begin{aligned}
S_{2}&-\frac{\delta_{1}}{16\varepsilon^{2}}\|\partial_{x}^{\alpha}\partial_{v}^{\beta}(I-P_{2})g^{n+1}\|_{L^{2}_{x,v}(\nu)}^{2}-\frac{\delta}{16\varepsilon^{2}}\|\partial_{x}^{\alpha+1}(I-P_{2})g^{n+1}\|_{L^{2}_{x,v}(\nu)}^{2}\\
&\lesssim\mathcal{E}_{N}^{\frac{1}{2}}(G^{n+1},\phi^{n+1})\tilde{\mathcal{D}}_{N,\varepsilon}^{\frac{1}{2}}(G^{n+1}).
\end{aligned}
\end{equation}

Next, we estimate $S_{3}$,
\begin{equation*}
\begin{aligned}
S_{3}&=\underbrace{\l\partial_{x}^{\alpha}\partial_{v}^{\beta}(g^{n+1}v\cdot\nabla_{x}\phi^{n}),\partial_{x}^{\alpha}\partial_{v}^{\beta}(I-P_{1})f^{n+1}\r_{L^{2}_{x,v}}}_{S_{31}}\\
&\underbrace{-\l\partial_{x}^{\alpha}\partial_{v}^{\beta}(\nabla_{x}\phi^{n}\cdot\nabla_{v}(I-P_{2})g^{n+1}),\partial_{x}^{\alpha}\partial_{v}^{\beta}(I-P_{1})f^{n+1}\r_{L^{2}_{x,v}}}_{S_{32}}\\
&\underbrace{-\l\partial_{x}^{\alpha}\partial_{v}^{\beta}P_{1}(g^{n+1}v\cdot\nabla_{x}\phi^{n}),\partial_{x}^{\alpha}\partial_{v}^{\beta}(I-P_{1})f^{n+1}\r_{L^{2}_{x,v}}}_{S_{33}}\\
&\underbrace{+\l\partial_{x}^{\alpha}\partial_{v}^{\beta}P_{1}(\nabla_{x}\phi^{n}\cdot\nabla_{v}(I-P_{2})g^{n+1}),\partial_{x}^{\alpha}\partial_{v}^{\beta}(I-P_{1})f^{n+1}\r_{L^{2}_{x,v}}}_{S_{34}}.
\end{aligned}
\end{equation*}

Then, we estimate $S_{3i}(i=1,2,3,4)$ one by one
\begin{equation*}
\begin{aligned}
S_{31}&=\underbrace{\l v\cdot\nabla_{x}\phi^{n}\cdot\partial_{x}^{\alpha}\partial_{v}^{\beta}(I-P_{2})g^{n+1},\partial_{x}^{\alpha}\partial_{v}^{\beta}(I-P_{1})f^{n+1}\r_{L^{2}_{x,v}}}_{S_{311}}\\
&+\underbrace{\l\partial_{x}^{\alpha}\partial_{v}^{\beta}(P_{2}g^{n+1}v\cdot\nabla_{x}\phi^{n}),\partial_{x}^{\alpha}\partial_{v}^{\beta}(I-P_{1})f^{n+1}\r_{L^{2}_{x,v}}}_{S_{312}}\\
&+\underbrace{\sum_{0\neq\tilde{\alpha}\leq\alpha}\l v\cdot\nabla_{x}\partial_{x}^{\tilde{\alpha}}\phi^{n}\cdot\partial_{x}^{\alpha-\tilde{\alpha}}\partial_{v}^{\beta}(I-P_{2})g^{n+1},\partial_{x}^{\alpha}\partial_{v}^{\beta}(I-P_{1})f^{n+1}\r_{L^{2}_{x,v}}}_{S_{313}}\\
&+\underbrace{\sum_{0\neq\tilde{\alpha}\leq\alpha}\l \nabla_{x}\partial_{x}^{\tilde{\alpha}}\phi^{n}\cdot\partial_{x}^{\alpha-\tilde{\alpha}}\partial_{v}^{\beta-1}(I-P_{2})g^{n+1},\partial_{x}^{\alpha}\partial_{v}^{\beta}(I-P_{1})f^{n+1}\r_{L^{2}_{x,v}}}_{S_{314}}.
\end{aligned}
\end{equation*}

We estimate $S_{311}$ by H\"older inequality, Sobolev embedding inequality $H_{x}^{2}\hookrightarrow L^{\infty}_{x}$ and $\nu(v)\sim1+|v|$ as
\begin{equation*}
\begin{aligned}
S_{311}&\lesssim\|\nabla_{x}\phi^{n}\|_{H^{N}_{x}}\|(I-P_{2})g^{n+1}\|_{H_{x,v}(\nu)}\|(I-P_{2})f^{n+1}\|_{H_{x,v}(\nu)}\\
&\lesssim\varepsilon^{2}\mathcal{E}_{N}^{\frac{1}{2}}(G^{n},\phi^{n})\tilde{\mathcal{D}}_{N,\varepsilon}(G^{n+1}).
\end{aligned}
\end{equation*}

If $|\beta|=1$, then $S_{312}$ can be control as
\begin{equation*}
\begin{aligned}
S_{312}&=\sum_{\tilde{\alpha}\leq\alpha}\l\nabla_{x}\partial_{x}^{\tilde{\alpha}}\phi^{n}\cdot\partial_{x}^{\alpha-\tilde{\alpha}}P_{2}g^{n+1},\partial_{x}^{\alpha}\partial_{v}^{\beta}(I-P_{1})f^{n+1}\r_{L^{2}_{x,v}}\\
&\lesssim\sum_{\tilde{\alpha}\leq\alpha-2}\|\nabla_{x}\partial_{x}^{\tilde{\alpha}}\phi^{n}\|_{L^{\infty}_{x}}\|\partial_{x}^{\alpha-\tilde{\alpha}}g^{n+1}\|_{L^{2}_{x,v}}\|(I-P_{1})f^{n+1}\|_{H^{2}_{x,v}}\\
&+\sum_{\alpha-1\leq\tilde{\alpha}\leq\alpha}\|\nabla_{x}\partial_{x}^{\tilde{\alpha}}\phi^{n}\|_{L^{4}_{x}}\|\partial_{x}^{\alpha-\tilde{\alpha}}g^{n+1}\|_{L^{4}_{x}L^{2}_{v}}\|(I-P_{1})f^{n+1}\|_{H^{2}_{x,v}}\\
&\lesssim\varepsilon\mathcal{E}_{N}^{\frac{1}{2}}(G^{n},\phi^{n})\mathcal{E}_{N}^{\frac{1}{2}}(G^{n+1},\phi^{n+1})\tilde{\mathcal{D}}_{N,\varepsilon}^{\frac{1}{2}}(G^{n+1}),
\end{aligned}
\end{equation*}
where we use $\|P_{2}g\|_{L_{x}^{2}}\lesssim\|g\|_{L^{2}_{x,v}}$ and Sobolev imbedding inequality $H_{x}^{2}\hookrightarrow L^{\infty}_{x}, H_{x}^{1}\hookrightarrow L^{4}_{x}$.

If $|\beta|\geq2$, then $S_{312}=0$.

The terms $S_{313}, S_{314}$ can be estimated as
\begin{equation*}
\begin{aligned}
S_{313}&\lesssim\sum_{0\neq\tilde{\alpha}\leq\alpha-2}\|\nabla_{x}\partial_{x}^{\tilde{\alpha}}\phi^{n}\|_{L^{\infty}_{x}}\|\partial_{x}^{\alpha-\tilde{\alpha}}\partial_{v}^{\beta}(I-P_{2})g^{n+1}\|_{L^{2}_{x,v}(\nu)}\|(I-P_{1})f^{n+1}\|_{H^{N}_{x,v}(\nu)}\\
&+\sum_{\alpha-1\leq\tilde{\alpha}\leq\alpha-2}\|\nabla_{x}\partial_{x}^{\tilde{\alpha}}\phi^{n}\|_{L^{4}_{x}}\|\partial_{x}^{\alpha-\tilde{\alpha}}\partial_{v}^{\beta}(I-P_{2})g^{n+1}\|_{L^{4}_{x}L^{2}_{v}(\nu(v))}\|(I-P_{1})f^{n+1}\|_{H^{N}_{x,v}(\nu)}\\
&\lesssim\varepsilon^{2}\mathcal{E}_{N}^{\frac{1}{2}}(f^{n},g^{n},\phi^{n})\tilde{\mathcal{D}}_{N,\varepsilon}(f^{n+1},g^{n+1}),\\
S_{314}&\lesssim\sum_{0\neq\tilde{\alpha}\leq\alpha-2}\|\nabla_{x}\partial_{x}^{\tilde{\alpha}}\phi^{n}\|_{L^{\infty}_{x}}\|\partial_{x}^{\alpha-\tilde{\alpha}}\partial_{v}^{\beta-1}(I-P_{2})g^{n+1}\|_{L^{2}_{x,v}}\|(I-P_{1})f^{n+1}\|_{H^{N}_{x,v}(\nu)}\\
&+\sum_{\alpha-1\leq\tilde{\alpha}\leq\alpha-2}\|\nabla_{x}\partial_{x}^{\tilde{\alpha}}\phi^{n}\|_{L^{4}_{x}}\|\partial_{x}^{\alpha-\tilde{\alpha}}\partial_{v}^{\beta-1}(I-P_{2})g^{n+1}\|_{L^{4}_{x}L^{2}_{v}}\|(I-P_{1})f^{n+1}\|_{H^{N}_{x,v}(\nu)}\\
&\lesssim\varepsilon^{2}\mathcal{E}_{N}^{\frac{1}{2}}(G^{n},\phi^{n})\tilde{\mathcal{D}}_{N,\varepsilon}(G^{n+1}),\\
\end{aligned}
\end{equation*}
where we use $\nu(v)\sim1+|v|$.

In summary, we have
\begin{equation*}
\begin{aligned}
S_{31}&\lesssim\mathcal{E}_{N}^{\frac{1}{2}}(G^{n},\phi^{n})\mathcal{E}_{N}^{\frac{1}{2}}(G^{n+1},\phi^{n+1})\tilde{\mathcal{D}}_{N,\varepsilon}^{\frac{1}{2}}(G^{n+1})+\mathcal{E}_{N}^{\frac{1}{2}}(G^{n},\phi^{n})\tilde{\mathcal{D}}_{N,\varepsilon}(G^{n+1}),
\end{aligned}
\end{equation*}
where $0<\varepsilon\leq 1$ be used.

Analogously, we can also estimate $S_{32}$ that
\begin{equation*}
S_{32}\lesssim\mathcal{E}_{N}^{\frac{1}{2}}(G^{n},\phi^{n})\mathcal{E}_{N}^{\frac{1}{2}}(G^{n+1},\phi^{n+1})\tilde{\mathcal{D}}_{N,\varepsilon}^{\frac{1}{2}}(G^{n+1})+\mathcal{E}_{N}^{\frac{1}{2}}(G^{n},\phi^{n})\tilde{\mathcal{D}}_{N,\varepsilon}(G^{n+1}).
\end{equation*}


Therefore, the term $S_{3}$ can be control by
\begin{equation}\label{S3}
\begin{aligned}
S_{3}&\lesssim\mathcal{E}_{N}^{\frac{1}{2}}(G^{n},\phi^{n})\mathcal{E}_{N}^{\frac{1}{2}}(G^{n+1},\phi^{n+1})\tilde{\mathcal{D}}_{N,\varepsilon}^{\frac{1}{2}}(G^{n+1})+\mathcal{E}_{N}^{\frac{1}{2}}(G^{n},\phi^{n})\tilde{\mathcal{D}}_{N,\varepsilon}(G^{n+1}).
\end{aligned}
\end{equation}

We estimate $S_{4}$ by the same argument as $S_{3}$ that
\begin{equation}\label{S4}
\begin{aligned}
S_{4}&\lesssim\mathcal{E}_{N}^{\frac{1}{2}}(G^{n},\phi^{n})\mathcal{E}_{N}^{\frac{1}{2}}(G^{n+1},\phi^{n+1})\tilde{\mathcal{D}}_{N,\varepsilon}^{\frac{1}{2}}(G^{n+1})+\mathcal{E}_{N}^{\frac{1}{2}}(G^{n},\phi^{n})\tilde{\mathcal{D}}_{N,\varepsilon}(G^{n+1}).
\end{aligned}
\end{equation}

For the term $S_{5}$, we divide it into four parts that
\begin{equation*}
\begin{aligned}
S_{5}&=\underbrace{\frac{1}{\varepsilon}\l\partial_{x}^{\alpha}\partial_{v}^{\beta}Q(P_{1}f^{n},P_{1}f^{n}),\partial_{x}^{\alpha}\partial_{v}^{\beta}(I-P_{1})f^{n+1}\r_{L^{2}_{x,v}}}_{S_{51}}\\
&+\underbrace{\frac{1}{\varepsilon}\l\partial_{x}^{\alpha}\partial_{v}^{\beta}Q(P_{1}f^{n},(I-P_{1})f^{n}),\partial_{x}^{\alpha}\partial_{v}^{\beta}(I-P_{1})f^{n+1}\r_{L^{2}_{x,v}}}_{S_{52}}\\
&+\underbrace{\frac{1}{\varepsilon}\l\partial_{x}^{\alpha}\partial_{v}^{\beta}Q((I-P_{1})f^{n},P_{1}f^{n}),\partial_{x}^{\alpha}\partial_{v}^{\beta}(I-P_{1})f^{n+1}\r_{L^{2}_{x,v}}}_{S_{53}}\\
&+\underbrace{\frac{1}{\varepsilon}\l\partial_{x}^{\alpha}\partial_{v}^{\beta}Q((I-P_{1})f^{n},(I-P_{1})f^{n}),\partial_{x}^{\alpha}\partial_{v}^{\beta}(I-P_{1})f^{n+1}\r_{L^{2}_{x,v}}}_{S_{54}}.
\end{aligned}
\end{equation*}

Applying Lemma \ref{Lmm-Q} and $\|\partial_{x}^{\alpha}\partial_{v}^{\beta}P_{1}f\|_{L_{x,v}(\nu(v))}\lesssim\|\partial_{x}^{\alpha}f\|_{L^{2}_{x,v}}$, we have estimate $S_{51}$ as
\begin{equation*}
	  \begin{aligned}
	    S_{51} &\lesssim\|P_{1}f^{n}\|_{H^{N}_{x,v}}\|P_{1}f^{n}\|_{H_{x,v}^{N}(\nu(v))}\|(I-P_{1})f^{n+1}\|_{H^{N}_{x,v}(\nu)}\lesssim\|f^{n}\|_{H^{N}_{x}L_{v}^{2}}\|(I-P_{1})f^{n+1}\|_{H^{N}_{x,v}(\nu)}\\
&\lesssim\mathcal{E}_{N}(G^{n},\phi^{n})\tilde{\mathcal{D}}_{N,\varepsilon}^{\frac{1}{2}}(G^{n+1}).
	  \end{aligned}
	\end{equation*}

For the term $S_{52},S_{53}$ and $S_{54}$, we gain the estimate by the same argument as $S_{51}$
\begin{equation*}
	  \begin{aligned} &S_{52}+S_{53}+S_{54}\lesssim\frac{1}{\varepsilon}\|f^{n}\|_{H^{N}_{x}L_{v}^{2}}\|(I-P_{1})f^{n}\|_{H^{N}_{x,v}(\nu)}\|(I-P_{1})f^{n+1}\|_{H^{N}_{x,v}(\nu)}\\
&+\frac{1}{\varepsilon}\|(I-P_{1})f^{n}\|_{H^{N}_{x,v}}\|(I-P_{1})f^{n}\|_{H^{N}_{x,v}(\nu)}\|(I-P_{1})f^{n+1}\|_{H^{N}_{x,v}(\nu)}\\
&\lesssim\varepsilon\mathcal{E}_{N}^{\frac{1}{2}}(G^{n},\phi^{n})\tilde{\mathcal{D}}_{N,\varepsilon}^{\frac{1}{2}}(f^{n},g^{n})\tilde{\mathcal{D}}_{N,\varepsilon}^{\frac{1}{2}}(G^{n+1}).
	  \end{aligned}
\end{equation*}

Consequently, we estimate $S_{5}$ that
\begin{equation}\label{S5}
\begin{aligned}
S_{5}&\lesssim\mathcal{E}_{N}(G^{n},\phi^{n})\tilde{\mathcal{D}}_{N,\varepsilon}^{\frac{1}{2}}(G^{n+1})+\mathcal{E}_{N}^{\frac{1}{2}}(G^{n},\phi^{n})\tilde{\mathcal{D}}_{N,\varepsilon}^{\frac{1}{2}}(G^{n})\tilde{\mathcal{D}}_{N,\varepsilon}^{\frac{1}{2}}(G^{n+1}).
\end{aligned}
\end{equation}

Analogously, the term $S_{6}$ can be control by
\begin{equation}\label{S6}
\begin{aligned}
S_{6}&\lesssim\mathcal{E}_{N}(G^{n},\phi^{n})\tilde{\mathcal{D}}_{N,\varepsilon}^{\frac{1}{2}}(G^{n+1})+\mathcal{E}_{N}^{\frac{1}{2}}(G^{n},\phi^{n})\tilde{\mathcal{D}}_{N,\varepsilon}^{\frac{1}{2}}(G^{n})\tilde{\mathcal{D}}_{N,\varepsilon}^{\frac{1}{2}}(G^{n+1}).
\end{aligned}
\end{equation}

From the estimates \eqref{S1}, \eqref{S2}, \eqref{S3}, \eqref{S4}, \eqref{S5}, \eqref{S6} and \eqref{2.9}, we deduce that
	\begin{equation}\label{2.1.1}
	  \begin{aligned}
	    & \frac{1}{2}\frac{d}{d t}\| \partial_x^\alpha \partial_v^\beta (\mathbb{I}-\mathbb{P})G^{n+1} \|^2_{L^2_{x,v}}+\frac{\delta_1}{\varepsilon^2} \| \partial_x^\alpha \partial_v^\beta (\mathbb{I}-\mathbb{P}) G^{n+1} \|_{L^2_{x,v} (\nu)}^2
-\frac{\delta_{2}}{\varepsilon^2} \sum_{\tilde{\beta}< \beta}\| \partial^\alpha_x \partial^{\tilde{\beta}}_v (\mathbb{I}-\mathbb{P}) G^{n+1} \|^2_{L^2_{x,v} (\nu)}\\
&-\frac{\delta}{8\varepsilon^{2}}\|\partial_{x}^{\alpha+1}(\mathbb{I}-\mathbb{P})G^{n+1}\|_{L^{2}_{x,v}(\nu)}^{2}\lesssim\mathcal{E}_{N}^{\frac{1}{2}}(G^{n+1},\phi^{n+1})\tilde{\mathcal{D}}_{N,\varepsilon}^{\frac{1}{2}}(G^{n+1})+\mathcal{E}_{N}^{\frac{1}{2}}(G^{n},\phi^{n})\tilde{\mathcal{D}}_{N,\varepsilon}(G^{n+1})\\
&+\mathcal{E}_{N}(G^{n},\phi^{n})\tilde{\mathcal{D}}_{N,\varepsilon}^{\frac{1}{2}}(G^{n+1})+\mathcal{E}_{N}^{\frac{1}{2}}(G^{n},\phi^{n})\tilde{\mathcal{D}}_{N,\varepsilon}^{\frac{1}{2}}(G^{n})\tilde{\mathcal{D}}_{N,\varepsilon}^{\frac{1}{2}}(G^{n+1})\\
&+\mathcal{E}_{N}^{\frac{1}{2}}(G^{n},\phi^{n})\mathcal{E}_{N}^{\frac{1}{2}}(G^{n+1},\phi^{n+1})\tilde{\mathcal{D}}_{N,\varepsilon}^{\frac{1}{2}}(G^{n+1}).
\end{aligned}
\end{equation}

Combining \eqref{Iter-Spatial-Bnd} and \eqref{2.1.1}, we can derive from the induction of $0 \leq |\beta| = k \leq N$ that there exist some positive constants $a_{|\beta|}$, $b_k$ and $c_k$ such that
	\begin{equation}\label{2.20}
	  \begin{aligned}
& \frac{1}{2} \frac{d}{d t} \big( \| G^{n+1} \|^2_{H^N_x L^2_v} + \| \nabla_x \phi^{n+1} \|^2_{H^N_x} \big)+\frac{1}{2}\frac{d}{d t}\sum_{\substack{ |\alpha| + |\beta| \leq N \\ 1 \leq |\beta| \leq k }} c_{|\beta|}  \| \partial_x^\alpha \partial_v^\beta (\mathbb{I}-\mathbb{P})G^{n+1} \|^2_{L^2_{x,v}}+ \frac{b_k}{\varepsilon^2} \| (\mathbb{I}-\mathbb{P}) G^{n+1} \|^2_{H^N_x L^2_v (\nu)}\\
&+ \frac{c_k}{\varepsilon^2} \sum_{\substack{ |\alpha| + |\beta| \leq N \\ 1 \leq |\beta| \leq k }}\| \partial_x^\alpha \partial_v^\beta (\mathbb{I}-\mathbb{P}) G^{n+1} \|^2_{L^2_{x,v} (\nu)} \lesssim\mathcal{E}_{N}^{\frac{1}{2}}(G^{n+1},\phi^{n+1})\tilde{\mathcal{D}}_{N,\varepsilon}^{\frac{1}{2}}(G^{n+1})+\mathcal{E}_{N}^{\frac{1}{2}}(G^{n},\phi^{n})\tilde{\mathcal{D}}_{N,\varepsilon}(G^{n+1})\\
&+\mathcal{E}_{N}^{\frac{1}{2}}(G^{n},\phi^{n})\mathcal{E}_{N}(G^{n+1},\phi^{n+1})+\mathcal{E}_{N}(G^{n},\phi^{n})\tilde{\mathcal{D}}_{N,\varepsilon}^{\frac{1}{2}}(G^{n+1})+\mathcal{E}_{N}^{\frac{1}{2}}(G^{n},\phi^{n})\tilde{\mathcal{D}}_{N,\varepsilon}^{\frac{1}{2}}(G^{n})\tilde{\mathcal{D}}_{N,\varepsilon}^{\frac{1}{2}}(G^{n+1})\\
&+\mathcal{E}_{N}^{\frac{1}{2}}(G^{n},\phi^{n})\mathcal{E}_{N}^{\frac{1}{2}}(G^{n+1},\phi^{n+1})\tilde{\mathcal{D}}_{N,\varepsilon}^{\frac{1}{2}}(G^{n+1}),
	  \end{aligned}
	\end{equation}
	for all $0 \leq |\beta| = k \leq N$.

Then, there exist some positive numbers  $a_{N}=a(N, c_{|\beta|})>0$, $b_{N}=b(N, b_{k})>0$, $c_{N}=c(N, c_{k})>0$ such that we can define the so-called instant iterating energy $\mathscr{E}_{N,\varepsilon} (G, \phi )$ that
	\begin{equation}
	  \begin{split}
	    \mathscr{E}_{N,\varepsilon} (G, \phi) &=  \| G \|^2_{H^N_x L^2_v} + \| \nabla_x \phi \|^2_{H^N_x} + a_{N}\| (\mathbb{I}-\mathbb{P}) G\|^2_{H^{N}_{x,v}}+  \frac{b_{N}}{\varepsilon^2} \int_0^t \| (\mathbb{I}-\mathbb{P}) G \|^2_{H^N_x L^2_v (\nu)} d \xi \\
&+ \frac{ c_{N}}{\varepsilon^2} \int_0^t \| (\mathbb{I}-\mathbb{P}) G \|^2_{H^{N}_{x,v} (\nu)}d\xi \,,
	  \end{split}
	\end{equation}
     and the initial energy $\mathscr{E}_{N} ( G_0, \phi_0 )$ is given as
	\begin{equation}
	  \begin{aligned}
	    \mathscr{E}_{N} ( G_{0}, \phi_0 ) =  \| G_{0} \|^2_{H^N_x L^2_v}+ \| \nabla_x \phi_{0} \|^2_{H^N_x} + a_{N}\| (\mathbb{I}-\mathbb{P}) G_{0} \|^2_{H^{N}_{x,v}}\,.
	  \end{aligned}
	\end{equation}

	It is easy to derive that
	\begin{equation*}
	  \begin{aligned}
	    \mathscr{E}_{N, \varepsilon} (G, \phi) \thicksim \mathcal{E}_N (G, \phi) + \int_0^t  \tilde{\mathcal{D}}_{N,\varepsilon}(G)d \xi \,.
	  \end{aligned}
	\end{equation*}
	
Now we claim that there exist two small positive numbers $\tau, T \in (0, 1]$, independent of $\varepsilon$, such that if $0 < T \leq 1 $, $ \mathscr{E}_N (G_{\varepsilon,0}, \phi_{\varepsilon,0}) \leq C_\# \mathcal{E}_N (G_{\varepsilon,0}, \phi_{\varepsilon,0}) \leq C_\# \delta \leq \tau $ for some constant $C_\# > 0$ and
	$$\sup_{0 \leq t \leq T} \mathscr{E}_{N,\varepsilon} (G^{n}(t), \phi^n (t)) \leq 2 \tau \,, $$
	then, we gain
	\begin{equation}\label{Claim-Unif-Bnd}
	  \begin{aligned}
	    \sup_{0 \leq t \leq T} \mathscr{E}_{N,\varepsilon} (G^{n+1} (t), \phi^{n+1} (t)) \leq 2 \tau \,.
	  \end{aligned}
	\end{equation}

	Indeed, from taking $k=N$ in the inequality \eqref{2.20} and integrating over $[0,t]$, we derive that for all $t \in [0, T]$
\begin{equation*}
\begin{aligned}
&\int_{0}^{t}\mathcal{E}_{N}^{\frac{1}{2}}(G^{n+1}(\xi),\phi^{n+1}(\xi))\tilde{\mathcal{D}}_{N,\varepsilon}^{\frac{1}{2}}(G^{n+1}(\xi))d\xi\lesssim\sup_{0 \leq t \leq T}\mathcal{E}_{N}^{\frac{1}{2}}(G^{n+1}(t),\phi^{n+1}(t))\sqrt{T}(\int_{0}^{t}\tilde{\mathcal{D}}_{N,\varepsilon}(G^{n+1}(\xi))d\xi)^{\frac{1}{2}}\\
&\lesssim\sqrt{T}\sup_{0 \leq t \leq T}\mathscr{E}_{N, \varepsilon} (G^{n+1} (t), \phi^{n+1} (t)),\\
\end{aligned}
\end{equation*}

\begin{equation*}
\begin{aligned}
&\int_{0}^{t}\mathcal{E}_{N}^{\frac{1}{2}}(G^{n}(\xi),\phi^{n}(\xi))\tilde{\mathcal{D}}_{N,\varepsilon}(G^{n+1}(\xi))d\xi\lesssim\sup_{0\leq t\leq T}\mathcal{E}_{N}^{\frac{1}{2}}(G^{n}(t),\phi^{n}(t))\int_{0}^{t}\tilde{\mathcal{D}}_{N,\varepsilon}(G^{n+1}(\xi))d\xi\\
&\lesssim \sup_{0\leq t\leq T}\mathscr{E}_{N, \varepsilon}^{\frac{1}{2}} (G^{n} (t), \phi^{n} (t))\sup_{0\leq t\leq T}\mathscr{E}_{N, \varepsilon} (G^{n+1} (t), \phi^{n+1} (t))\lesssim 2\tau\sup_{0\leq t\leq T}\mathscr{E}_{N, \varepsilon} (G^{n+1} (t), \phi^{n+1} (t)),
\end{aligned}
\end{equation*}

\begin{equation*}
\begin{aligned}
&\int_{0}^{t}\mathcal{E}_{N}^{\frac{1}{2}}(G^{n}(\xi),\phi^{n}(\xi))\mathcal{E}_{N,\varepsilon}(G^{n+1}(\xi),\phi^{n+1}(\xi))d\xi\lesssim \sup_{0\leq t\leq T}\mathcal{E}_{N}(G^{n}(t),\phi^{n}(t))\int_{0}^{t}\mathcal{E}_{N}(G^{n+1}(\xi),\phi^{n+1}(\xi))d\xi\\
&\lesssim T\sup_{0\leq t\leq T}\mathscr{E}_{N, \varepsilon}^{\frac{1}{2}} (G^{n} (t), \phi^{n} (t))\sup_{0\leq t\leq T}\mathscr{E}_{N, \varepsilon} (G^{n+1} (t), \phi^{n+1} (t))\lesssim\sqrt{2\tau}T\sup_{0\leq t\leq T}\mathscr{E}_{N, \varepsilon} (G^{n+1} (t), \phi^{n+1} (t)),
\end{aligned}
\end{equation*}

\begin{equation*}
\begin{aligned}
&\int_{0}^{t}\mathcal{E}_{N}(G^{n}(\xi),\phi^{n}(\xi))\tilde{\mathcal{D}}^{\frac{1}{2}}_{N,\varepsilon}(G^{n+1}(\xi))d\xi\lesssim\sqrt{T}\sup_{0\leq t\leq T}\mathscr{E}_{N,\varepsilon}(G^{n}(t),\phi^{n}(t))\sup_{0\leq t\leq T}\mathscr{E}_{N,\varepsilon}^{\frac{1}{2}}(G^{n+1}(t),\phi^{n+1}(t))\\
&\lesssim\sqrt{T}\sup_{0\leq t\leq T}\mathscr{E}_{N,\varepsilon}^{2}(G^{n}(t),\phi^{n}(t))+\sqrt{T}\sup_{0\leq t\leq T}\mathscr{E}_{N,\varepsilon}(G^{n+1}(t),\phi^{n+1}(t))\\
&\lesssim \sqrt{T}(2\tau)^{2}+\sqrt{T}\sup_{0\leq t\leq T}\mathscr{E}_{N,\varepsilon}(G^{n+1}(t),\phi^{n+1}(t)),
\end{aligned}
\end{equation*}

\begin{equation*}
\begin{aligned}
&\int_{0}^{t}\mathcal{E}_{N}^{\frac{1}{2}}(G^{n}(\xi),\phi^{n}(\xi))\tilde{\mathcal{D}}_{N,\varepsilon}^{\frac{1}{2}}(G^{n}(\xi))\tilde{\mathcal{D}}_{N,\varepsilon}^{\frac{1}{2}}(G^{n+1}(\xi))d\xi\lesssim\sup_{0\leq t\leq T}\mathscr{E}_{N,\varepsilon}(G^{n}(t),\phi^{n}(t))\sup_{0\leq t\leq T}\mathscr{E}_{N,\varepsilon}^{\frac{1}{2}}(G^{n+1}(t),\phi^{n+1}(t))\\
&\leq C\sup_{0\leq t\leq T}\mathscr{E}_{N,\varepsilon}^{2}(G^{n}(t),\phi^{n}(t))+\frac{1}{16}\sup_{0\leq t\leq T}\mathscr{E}_{N,\varepsilon}(G^{n+1}(t),\phi^{n+1}(t))\leq C(2\tau)^{2}+\frac{1}{16}\sup_{0\leq t\leq T}\mathscr{E}_{N,\varepsilon}(G^{n+1}(t),\phi^{n+1}(t)),
\end{aligned}
\end{equation*}

\begin{equation*}
\begin{aligned}
&\int_{0}^{t}\mathcal{E}_{N}^{\frac{1}{2}}(G^{n}(\xi),\phi^{n}(\xi))\mathcal{E}_{N}^{\frac{1}{2}}(G^{n+1}(\xi),\phi^{n+1}(\xi))\tilde{\mathcal{D}}_{N,\varepsilon}^{\frac{1}{2}}(G^{n+1}(\xi))d\xi\lesssim\sqrt{T}\sup_{0\leq t\leq T}\mathscr{E}^{\frac{1}{2}}_{N,\varepsilon}(G^{n}(t),\phi^{n}(t))\\
&\cdot\sup_{0\leq t\leq T}\mathscr{E}_{N,\varepsilon}(G^{n+1}(t),\phi^{n+1}(t))\lesssim\sqrt{2T\tau}\sup_{0\leq t\leq T}\mathscr{E}_{N,\varepsilon}(G^{n+1}(t),\phi^{n+1}(t)).
\end{aligned}
\end{equation*}
where we use the Young inequality and the H\"older inequality. In summary, we have for some positive constant $C$ that
\begin{equation}
	  \begin{aligned}
	    \big[ \frac{15}{16} - C(\sqrt{T}+2\tau+\sqrt{2\tau}T+\sqrt{2\tau T}) \big] \sup_{0 \leq t \leq T} \mathscr{E}_{N, \varepsilon} (G^{n+1} (t) , \phi^{n+1} (t) ) \leq \tau + 4C ( 1 + \sqrt{T} ) \tau^2 \,.
	  \end{aligned}
\end{equation}

	Therefore, we take $T \in (0,1]$ and $\tau \in (0,1]$ some small such that $ \frac{15}{16} - C(\sqrt{T}+2\tau+\sqrt{2\tau}T+\sqrt{2\tau T}) \geq \frac{5}{8} $ and $\tau + 4C ( 1 + \sqrt{T} ) \tau^2\leq\frac{5}{4}\tau$.  Consequently, it is very easy to derive that
	\begin{equation}
	  \begin{aligned}
	    \sup_{0 \leq t \leq T} \mathscr{E}_{N,\varepsilon} (G^{n+1} (t) , \phi^{n+1} (t) ) \leq 2\tau,
	  \end{aligned}
	\end{equation}
	which immediately implies \eqref{Claim-Unif-Bnd} holds.
	
	Consequently,  it is easy to obtain a solution $(g_{\varepsilon}^{+} ,g_{\varepsilon}^{-}, \phi_{\varepsilon} )$ to \eqref{VPB-g}-\eqref{f's intial} for any fixed $0 < \varepsilon \leq 1$ by employing the standard compactness arguments if we take $n \rightarrow \infty$. Moreover, the uniform bound \eqref{Claim-Unif-Bnd} implies the energy bound \eqref{2.3}. Since $f_{\varepsilon}^{0,\pm}\geq0$, from a simple induction over $n$, we deduce $f_{\varepsilon}^{n,\pm}\geq0$. This implies $f_{\varepsilon}^{\pm}\geq0$. Thus, we have completed the proof of Lemma \ref{Lmm-Local}.
\end{proof}

\section{Uniform estimates with $\varepsilon$ and global-in-time solutions}\label{Sec: Uniform-Bnd}

In this section, we will extend the local solution in Lemma \ref{Lmm-Local} to a global-in-time solution of  \eqref{VPB-g}-\eqref{f's intial} by deriving an energy estimates uniformly in $\varepsilon \in (0, 1]$ under small size of the initial data. For simplicity, we will drop the lower index $\varepsilon$ of $f_{\varepsilon},g_{\varepsilon} $ and $\phi_\varepsilon$ in the perturbed VPB system \eqref{VPB-g} that
\begin{equation}\label{VPB-g-drop-eps}
  \left\{
    \begin{array}{l}
 \partial_{t} f + \frac{1}{\varepsilon} v \cdot \nabla_x f + \frac{1}{\varepsilon^2} \mathcal{L}_{1}f - gv\cdot\nabla_{x}\phi+\nabla_{x}\phi\cdot\nabla_{v}g = \frac{1}{\varepsilon} Q (f , f ) \,, \\
 \partial_{t} g + \frac{1}{\varepsilon} v \cdot \nabla_{x} g +\frac{1}{\varepsilon^2} \mathcal{L}_{2} g - \frac{2}{\varepsilon} v \cdot \nabla_{x} \phi -  fv \cdot \nabla_{x} \phi + \nabla_x \phi\cdot \nabla_v f = \frac{1}{\varepsilon} Q ( g , f ) \,, \\
      \Delta_x \phi =\l g , 1 \r_{L^2_{v}} \,.
    \end{array}
  \right.
\end{equation}

\subsection{Pure spatial derivative estimates: kinetic dissipations}

In this subsection, we will consider the energy estimates on the pure spatial derivative of $G$.  We will prove the following lemma.
\begin{lemma}\label{Lmm-Spatial}
	Suppose that $(g_{\varepsilon}^{+}(t,x,v),g_{\varepsilon}^{-}(t,x,v),\phi_{\varepsilon}(t,x))$ is the solution to the perturbed VMB system \eqref{VPB-g}-\eqref{f's intial} constructed in Lemma \ref{Lmm-Local}. Then there exist two positive constants $\delta$ such that
	\begin{equation}\label{Spatial-Bnd}
	  \begin{aligned}
	    \frac{1}{2} \frac{d}{d t} \Big(\| G \|^2_{H^N_x L^2_v} +2\| \nabla_x \phi \|^2_{H^N_x} \Big) &+ \frac{\delta}{\varepsilon^2} \| (\mathbb{I}-\mathbb{P}) G \|^2_{H^N_x L^2_v (\nu)}\lesssim \mathcal{E}_N^\frac{1}{2} (G, \phi)  \mathcal{D}_{N,\varepsilon} (G),
	  \end{aligned}
	\end{equation}
	for all $0 < \varepsilon \leq 1$, where $\mathcal{E}_N (G, \phi)$, $\mathcal{D}_{N,\varepsilon} (G)$ are given in \eqref{Energy-E}, \eqref{Energy-D}.
\end{lemma}

\begin{proof}[Proof of Lemma \ref{Lmm-Spatial}]
	For any fixed $\alpha \in \mathbb{N}^3$ with $ | \alpha | \leq N $,  we apply the derivative operator $ \partial_x^\alpha $ to the $f,g$-equation in \eqref{VPB-g-drop-eps} and take $L^2_{x,v}$-inner product in the above equation by dot with $ \partial_{x}^{\alpha} f, \partial_{x}^{\alpha}g$, respectively, integrate by parts over $\R^3_{x} \times \R_{v}^3$. We thereby obtain
	\begin{equation}\label{3.2.1}
	  \begin{aligned}
	     & \frac{1}{2} \frac{d}{d t} \big(\| \partial_x^\alpha G \|_{L^{2}_{x,v}}^{2}+2\|\nabla_{x}\partial_{x}^{\alpha}\phi\|_{L^{2}_{x}}^{2}\big) + \frac{2\delta}{\varepsilon^2}\|\partial_{x}^{\alpha}(\mathbb{I}-\mathbb{P})G\|^{2}_{L_{x,v}^{2}(\nu)} \lesssim\underbrace{\frac{1}{\varepsilon}\l\partial_{x}^{\alpha}Q(f,f),\partial_{x}^{\alpha}f\r_{L^{2}_{x,v}}}_{I_{1}}+\underbrace{\frac{1}{\varepsilon}\l\partial_{x}^{\alpha}Q(g,f),\partial_{x}^{\alpha}g\r_{L^{2}_{x,v}}}_{I_{2}}\\
&\underbrace{-\l\partial_{x}^{\alpha}(\nabla_{x}\phi\cdot\nabla_{v}f)\cdot\partial_{x}^{\alpha}g+\partial_{x}^{\alpha}(\nabla_{x}\phi\cdot\nabla_{v}g)\cdot\partial_{x}^{\alpha}f,1\r_{L^{2}_{x,v}}}_{I_{3}}+\underbrace{\l\partial_{x}^{\alpha}(fv\cdot\nabla_{x}\phi)\cdot\partial_{x}^{\alpha}g+\partial_{x}^{\alpha}(gv\cdot\nabla_{x}\phi)\cdot\partial_{x}^{\alpha}f,1\r_{L^{2}_{x,v}}}_{I_{4}},
	  \end{aligned}
	\end{equation}
where we use  Lemma \ref{Lmm-Coercivity-L}, $\l\partial_{x}^{\alpha}(v\cdot\nabla_{x}f),\partial_{x}^{\alpha}f\r_{L^{2}_{x,v}}=\l\partial_{x}^{\alpha}(v\cdot\nabla_{x}g),\partial_{x}^{\alpha}g\r_{L^{2}_{x,v}}=0$, the local conservation law of mass $ \partial_t \l g, 1 \r_{L^2_v} + \frac{1}{\varepsilon} \nabla_{x}\cdot\l g , v \r_{L^2_v} = 0$ and $ \Delta_x \phi =\l g , 1 \r_{L^2_v} $.
	
	Then, we estimate the terms $I_{i}(i=1,...,4)$  on the right-hand side of \eqref{3.2.1} one by one. We make use of  $f = P_{1}f+ (I-P_{1}) f$, so that $ I_1 $ is divided into
	\begin{equation}
	  \begin{aligned}
	   & I_{1} = \underbrace{ \frac{1}{\varepsilon} \l\partial_x^\alpha Q ( P_{1}f , P_{1}f ) , \partial_x^\alpha (I-P_{1})f \r_{L^2_{x,v}} }_{I_{11}} + \underbrace{\frac{1}{\varepsilon} \l \partial_x^\alpha Q ( P_{1}f , (I-P_{1})f ) , \partial_x^\alpha (I-P_{1})f\r_{L^2_{x,v}} }_{I_{12}} \\
	    &+ \underbrace{ \frac{1}{\varepsilon} \l \partial_x^\alpha Q ( (I-P_{1})f , P_{1}f) , \partial_x^\alpha (I-P_{1})f\r_{L^2_{x,v}} }_{I_{13}}+ \underbrace{ \frac{1}{\varepsilon} \l \partial_x^\alpha Q ( (I-P_{1})f, (I-P_{1})f) , \partial_x^\alpha (I-P_{1})f \r_{L^2_{x,v}} }_{I_{14}} ,
	  \end{aligned}
	\end{equation}
	where we also utilize the fact that $Q (f, f) \in \mathcal{N}_{1}^\perp$. Then, we estimate $I_{1i}(1\leq i\leq 4)$.  On the other hand, we plug $P_{1}f = a + b\cdot v + c|v|^2$ into the term $I_{11}$ to gain
	\begin{equation}
	  \begin{aligned}
	    & I_{11}=\frac{1}{\varepsilon}\l\partial^\alpha_x Q ( P_{1}f, P_{1}f), \partial_x^\alpha (I-P_{1})f \r_{L^2_{x,v}}\\
	    \lesssim & \frac{1}{\varepsilon}\sum_{\alpha_1 \leq \alpha} \| \, | \partial^{\alpha_1}_x a | \, | \partial^{\alpha - \alpha_1}_x a| + | \partial^{\alpha_1}_x b | \, | \partial^{\alpha - \alpha_1}_x b| + | \partial^{\alpha_1}_x c | \, | \partial^{\alpha - \alpha_1}_x c| \|_{L^2_x} \| \partial_x^\alpha (I-P_{1})f\|_{L^2_{x,v} (\nu)} \,.
	  \end{aligned}
	\end{equation}

	If $\alpha \neq 0$, the first factor is bounded by the calculus inequality
	\begin{equation*}
	  \begin{aligned}
	    &\Big( \| \nabla_x a \|_{H^{N-1}_x} + \| \nabla_x b \|_{H^{N-1}_x} + \| \nabla_x c \|_{H^{N-1}_x} \Big) \Big( \| a \|_{H^N_x} + \| b \|_{H^N_x} + \| c \|_{H^N_x}  \Big) \\
	    &\lesssim \| \nabla_x P_{1}f\|_{H^{N-1}_x L^2_v} \| P_{1}f \|_{H^N_x L^2_v} \lesssim \| \nabla_x P_{1}f\|_{H^{N-1}_x L^2_v} \|f\|_{H^N_x L^2_v} \,.
	  \end{aligned}
	\end{equation*}

	If $\alpha = 0$, the first factor is bounded by the H\"older inequality and the Sobolev inequality
	\begin{equation*}
	  \begin{aligned}
	    \| a+|b|+c\|_{L^2_x} \lesssim & \big( \| a\|_{L^6_x} + \| b\|_{L^6_x} + \| c \|_{L^6_x} \big) \big( \| a \|_{L^3_x} + \| b\|_{L^3_x} + \| c\|_{L^3_x} \big) \\
	    \lesssim & \big( \| \nabla_x a\|_{L^2_x} + \| \nabla_x b\|_{L^2_x} + \| \nabla_x c \|_{L^2_x} \big) \big( \|a \|_{H^1_x} + \|b\|_{H^1_x} + \| c\|_{H^1_x} \big) \\
	    \lesssim & \| \nabla_x P_{1}f\|_{H^{N-1}_x L^2_v} \|f\|_{H^N_x L^2_v}.
	  \end{aligned}
	\end{equation*}

	Consequently, we have
	\begin{equation}\label{I11}
	  \begin{aligned}
	    I_{11}\lesssim \frac{1}{\varepsilon} \| \nabla_x P_{1}f\|_{H^{N-1}_x L^2_v} \|f\|_{H^N_x L^2_v} \| \partial_x^\alpha (I-P_{1})f\|_{L^2_{x,v} (\nu)} \lesssim \mathcal{E}_N^\frac{1}{2} (G, \phi) \mathcal{D}_{N, \varepsilon} (G) \,,
	  \end{aligned}
	\end{equation}
	where the functionals $ \mathcal{E}_N (G, \phi) $ and $ \mathcal{D}_{N, \varepsilon} (G) $ are defined in \eqref{Energy-E} and \eqref{Energy-D}. Moreover, by employing the similar arguments in estimates \eqref{S4}, it is easy to derive
	\begin{equation}\no
	  \begin{aligned}
	    I_{12}+I_{13}+ I_{14} \lesssim & \frac{1}{\varepsilon} \| f \|_{H^N_x L^2_v} \| (I-P_{1})f \|^2_{H^N_x L^2_v (\nu)} \lesssim \varepsilon \mathcal{E}_N^\frac{1}{2} (G, \phi) \mathcal{D}_{N, \varepsilon} (G) \,.
	  \end{aligned}
	\end{equation}

Consequently, we have estimates that
	\begin{equation}\label{I1}
	  \begin{aligned}
	   I_1 \lesssim \mathcal{E}_N^\frac{1}{2} (G, \phi) \mathcal{D}_{N, \varepsilon} (G),
	  \end{aligned}
	\end{equation}
	for all $0 < \varepsilon \leq 1$. As for the term $I_{2}$, by
employing the similar arguments in \eqref{I1} and the decomposition $g = P_{2}g+ (I-P_{2}) g$ that
	\begin{equation}\label{I2}
I_2 \lesssim \mathcal{E}_N^\frac{1}{2} (G, \phi) \mathcal{D}_{N, \varepsilon} (G).
	\end{equation}

We divided $I_{3}$ into two parts that
	\begin{equation*}
	  \begin{aligned}   I_{3}&=\underbrace{-\l\nabla_{x}\phi\cdot\nabla_{v}\partial_{x}^{\alpha}f\cdot\partial_{x}^{\alpha}g+\nabla_{x}\phi\cdot\nabla_{v}\partial_{x}^{\alpha}g\cdot\partial_{x}^{\alpha}f,1\r_{L^{2}_{x,v}}}_{I_{31}}
\underbrace{-\sum_{0\neq\tilde{\alpha}\leq\alpha}\l\nabla_{x}\partial_{x}^{\tilde{\alpha}}\phi\cdot\nabla_{v}\partial_{x}^{\alpha-\tilde{\alpha}}f,\partial_{x}^{\alpha}g\r_{L^{2}_{x,v}}}_{I_{32}}\\
&\underbrace{-\sum_{0\neq\tilde{\alpha}\leq\alpha}\l\nabla_{x}\partial_{x}^{\tilde{\alpha}}\phi\cdot\nabla_{v}\partial_{x}^{\alpha-\tilde{\alpha}}g,\partial_{x}^{\alpha}f\r_{L^{2}_{x,v}}}_{I_{33}}.
	  \end{aligned}
	\end{equation*}

Then, we have
\begin{equation*}
\begin{aligned}
I_{31}=&-\l v\cdot\nabla_{x}\phi,\partial_{x}^{\alpha}f\cdot\partial_{x}^{\alpha}g\r_{L_{x,v}^{2}}=\underbrace{-\l v\cdot\nabla_{x}\phi,\partial_{x}^{\alpha}P_{1}f\cdot\partial_{x}^{\alpha}P_{2}g\r_{L^{2}_{x,v}}}_{I_{311}}
\underbrace{-\l v\cdot\nabla_{x}\phi,\partial_{x}^{\alpha}(I-P_{1})f\cdot\partial_{x}^{\alpha}P_{2}g\r_{L^{2}_{x,v}}}_{I_{312}}\\
&\underbrace{-\l v\cdot\nabla_{x}\phi,\partial_{x}^{\alpha}P_{1}f\cdot\partial_{x}^{\alpha}(I-P_{2})g\r_{L^{2}_{x,v}}}_{I_{313}}
\underbrace{-\l v\cdot\nabla_{x}\phi,\partial_{x}^{\alpha}(I-P_{1})f\cdot\partial_{x}^{\alpha}(I-P_{2})g\r_{L^{2}_{x,v}}}_{I_{314}},
\end{aligned}
\end{equation*}
where $G=\mathbb{P}G+(\mathbb{I}-\mathbb{P})G$.

Next, we estimate $I_{311}$ to $I_{314}$ one by one. If $\alpha = 0$, the term $I_{311}$ is bounded by
	\begin{equation*}
	  \begin{aligned}
	    I_{311}&\lesssim\| \nabla_x \phi \|_{L^{\frac{3}{2}}_{x}} \|P_{1}f \|_{L^{6}_{x} L^{2}_{v}} \| P_{2}g \|_{L^{6}_{x}}\|v\|_{L^{2}_{v}}\lesssim \|\nabla_{x}\phi\|_{H^{N}_{x}}\|\nabla_{x}P_{1}f\|_{H^{N-1}_{x}L^{2}_{v}}\|\nabla_{x}P_{2}g\|_{H^{N-1}_{x}}\lesssim\mathcal{E}_{N}^{\frac{1}{2}}(G,\phi)\mathcal{D}_{N,\varepsilon}(G),
	  \end{aligned}
	\end{equation*}
where we use the Sobolev embedding inequality $\|f\|_{L^{6}_{x}(\mathbb{R}^{3})}\lesssim\|\nabla_{x}f\|_{L^{2}_{x}(\mathbb{R}^{3})}$, $P_{2}g=d(t,x)$, $H^{1}_{x}\hookrightarrow L^{\frac{3}{2}}_{x}$ and $N\geq 4$. If $\alpha \neq 0$, the term $I_{311}$ is bounded by
	\begin{equation}\no
	  \begin{aligned}
	    I_{311}&\lesssim\| \nabla_{x} \phi \|_{L^{\infty}_{x}} \|\partial^{\alpha}_{x} P_{1}f\|_{L^{2}_{x,v}} \| \partial^{\alpha}_{x}P_{2}g \|_{L^2_{x}}\|v\|_{L^{2}_{v}}\lesssim \| \nabla_{x} \phi \|_{H^{N}_{x}} \| \nabla_{x} P_{1}f \|_{H^{N-1}_{x} L^{2}_{v}} \|\nabla_{x}P_{2}g \|_{H^{N-1}_{x}} \lesssim \mathcal{E}_N^\frac{1}{2} (G, \phi) \mathcal{D}_{N, \varepsilon} (G) \,,
	  \end{aligned}
	\end{equation}
	where \eqref{nu} and the Sobolev embedding $H^2_x \hookrightarrow L^\infty_x$ are also used. In summary, we obtain
	\begin{equation*}
	  \begin{aligned}
	    I_{311} \lesssim \mathcal{E}_{N}^{\frac{1}{2}} (G, \phi) \mathcal{D}_{N, \varepsilon} (G) \,.
	  \end{aligned}
	\end{equation*}

If $\alpha= 0$, we can estimate $I_{312}$ as follow
\begin{equation*}
\begin{aligned}
I_{312}&\lesssim\|\nabla_{x}\phi\|_{L^{3}_{x}}\|(I-P_{1})f\|_{L^{2}_{x,v}(\nu)}\|P_{2}g\|_{L^{6}_{x}}\|\sqrt{v}\|_{L^{2}_{v}}\lesssim\|\nabla_{x}\phi\|_{H^{N}_{x}}\|(I-P_{1})f\|_{H^{N}_{x}L^{2}_{v}(\nu)}\|\nabla_{x}P_{2}g\|_{H^{N-1}_{x}}\\
&\lesssim\varepsilon\mathcal{E}_{N}^{\frac{1}{2}} (G, \phi) \mathcal{D}_{N, \varepsilon} (G)
\end{aligned}
\end{equation*}
where \eqref{nu}, $\|f\|_{L^{6}_{x}(\mathbb{R}^{3})}\lesssim\|\nabla_{x}f\|_{L^{2}_{x}(\mathbb{R}^{3})}$, $P_{2}g=d(t,x)$ and the Sobolev embedding $H^{1}_{x} \hookrightarrow L^{3}_{x}$ are also used. If $\alpha\neq 0$, we can estimate $I_{312}$ as follow
\begin{equation*}
\begin{aligned}
I_{312}&\lesssim\|\nabla_{x}\phi\|_{L^{\infty}_{x}}\|\partial_{x}^{\alpha}(I-P_{1})f\|_{L^{2}_{x,v}(\nu)}\|\partial_{x}P_{2}g\|_{L_{x}^{2}}\|\sqrt{v}\|_{L^{2}_{v}}\\
&\lesssim\|\nabla_{x}\phi\|_{H^{N}_{x}}\|(I-P_{1})f\|_{H^{N}_{x}L^{2}_{v}(\nu)}\|\nabla_{x}P_{2}g\|_{H^{N-1}_{x}}\lesssim\varepsilon\mathcal{E}_{N}^{\frac{1}{2}} (G, \phi) \mathcal{D}_{N, \varepsilon} (G),
\end{aligned}
\end{equation*}
where we mark use of \eqref{nu}, $P_{2}g=d(t,x)$ and the Sobolev embedding $H^{2}_{x} \hookrightarrow L^{\infty}_{x}$. In summary, we have
	\begin{equation*}
	  \begin{aligned}
	    I_{312} \lesssim \varepsilon\mathcal{E}_{N}^{\frac{1}{2}} (G, \phi) \mathcal{D}_{N, \varepsilon} (G) \,.
	  \end{aligned}
	\end{equation*}

The analogous arguments of $I_{312}$ tells us that $I_{313}$ and $I_{314}$ are bounded by
	\begin{equation}\no
	  \begin{aligned}
	    I_{313}+I_{314}\lesssim \varepsilon\mathcal{E}_{N}^{\frac{1}{2}} (G, \phi) \mathcal{D}_{N, \varepsilon} (G)\,.
	  \end{aligned}
	\end{equation}

As a result, we have
\begin{equation}\label{I31}
\begin{aligned}
I_{31}\lesssim\mathcal{E}_{N}^{\frac{1}{2}} (G, \phi) \mathcal{D}_{N, \varepsilon} (G),
\end{aligned}
\end{equation}
where we mark use of $0<\varepsilon\leq 1$. We make use of $G=\mathbb{P}G+(\mathbb{I}-\mathbb{P})G$ and Sobolev imbedding inequality, the terms $I_{32}$ and $I_{33}$ can be estimate as
	\begin{equation}
	  \begin{aligned}
I_{32}+I_{33}\lesssim\mathcal{E}_{N}^{\frac{1}{2}} (G, \phi) \mathcal{D}_{N, \varepsilon} (G).
	  \end{aligned}
	\end{equation}

In consequence, we have
\begin{equation}\label{I3}
I_{3}\lesssim\mathcal{E}_{N}^{\frac{1}{2}} (G, \phi) \mathcal{D}_{N, \varepsilon} (G).
\end{equation}

It is remain to control the term $I_{4}$. The analogous arguments of $I_{312}$ tell us the $I_{4}$ can be bounded by
\begin{equation}\label{I4}
\begin{aligned}
I_{4}\lesssim\mathcal{E}_{N}^{\frac{1}{2}} (G, \phi) \mathcal{D}_{N, \varepsilon} (G).
\end{aligned}
\end{equation}

Consequently, we have the bound \eqref{Spatial-Bnd} by plugging all estimates \eqref{I1}, \eqref{I2}, \eqref{I3}, \eqref{I4} into \eqref{3.2.1} and summing up for all $|\alpha| \leq N$. Then we complete the proof of Lemma \ref{Lmm-Spatial}.
\end{proof}

\subsection{Macroscopic energy estimates: fluid dissipations}

In this subsection, we will find a dissipative structure of the fluid part $P_{1}f$ by using the so-called micro-macro decomposition method for the two-species VPB system, depending on the so-called {\em thirteen moments} (\cite{GY-06}). More specifically, we will obtain the evolution equation of each coefficient $(a,b,c)$ of $P_{1}f$, where the basis $\{e_k\}_{k=1}^{13} \subseteq L^2_v$ of the coefficients consisting of
\begin{equation}\label{e 13}
  \big\{ 1, v_i, v_i^2, v_i |v|^2 \ (i=1,2,3) ; v_i v_j \ (1 \leq i < j \leq 3) \big\} \subseteq L^2_v \,.
\end{equation}

Let $\{e_k^*\}_{k=1}^{13}$ be a set of orthonormal basis in $L^{2}_{v}$ such that
\begin{equation*}
  \l e_k^* , e_j^* \r_{L^2_v} = \delta_{jk} \,, \ j,k=1,2\cdots, 13 \,,
\end{equation*}
where $\delta_{jk} = 0$ if $j \neq k$ and $\delta_{jk} = 1$ if $j=k$. Actually, $\{e_k^*\}_{k=1}^{13}$ is a given linear combination of \eqref{e 13}. From \eqref{decomposition1}, we have:
\begin{align}
  1 : & \qquad \varepsilon \partial_t a =l_a + h_a + m_a \,, \label{Coef-Evl-a}\\
  v_i : & \qquad \varepsilon \partial_t b_i + \partial_i a = l_{bi}+ h_{bi} + m_{bi} \label{Coef-Evl-b}\,, \\
  v_i^2 : & \qquad \varepsilon \partial_t c + \partial_i b_i= l_i + h_i + m_i  \,,\label{3.3.8} \\
  v_i v_j : & \qquad \partial_i b_j+ \partial_j b_i = l_{ij} + h_{ij}+ m_{ij} \quad (i \neq j) \,, \label{3.3.7} \\
  v_i |v|^2 : & \qquad \partial_i c = l_{ci} + h_{ci} + m_{ci} \,, \label{3.3.9}
\end{align}
where all terms on the right-hand side above are the coefficients of $l$, $h$, $m$, defined in \eqref{l+h+m}, in terms of the basis \eqref{e 13}.

As for $g$,  we have
 \begin{equation}
\varepsilon\partial_{t}d+v\cdot\nabla_{x}d=\hat{l}+\hat{h}+\hat{m},
\end{equation}
where $\hat{l},\hat{h},\hat{m}$ are defined in \eqref{decomposition2}.

Our goal is to find a macroscopic dissipation. The high order derivatives of the fluid coefficients $(a(t,x),b(t,x),c(t,x))$ and $d(t,x)$ are dissipative from the balance laws \eqref{a-b-c} and \eqref{Coef-Evl-a}-\eqref{3.3.9}, Which is similar to the case of the Boltzmann equation.  More specifically, we give the following lemma.

\begin{lemma}\label{Lmm-MM-Decomp}
Suppose $(g_{\varepsilon}^{+}(t,x,v), g_{\varepsilon}^{-}(t,x,v),\phi_{\varepsilon}(t,x))$ is the solution of the perturbed VPB system \eqref{VPB-g}-\eqref{f's intial} constructed in Lemma \ref{Lmm-Local}. Then, we have
	\begin{equation}\label{MM-Inq}
	  \begin{split}
\varepsilon \frac{d}{d t} E^{int}_N (G) + \| \nabla_x b\|^2_{H^{N-1}_x} &+ \| \nabla_x c\|^2_{H^{N-1}_x} + \| \nabla_x (a+3c) \|^2_{H^{N-1}_x} +\|\nabla_{x}d\|_{H^{N-1}_{x}}^{2} \lesssim\mathcal{E}_N^\frac{1}{2} (G, \phi) \mathcal{D}_{N, \varepsilon} (G)+\mathcal{E}_{N}^\frac{1}{2} (G, \phi) \mathcal{D}^\frac{1}{2}_{N,\varepsilon}(G),
	  \end{split}
	\end{equation}
	for all $0 < \varepsilon\leq 1$ and $N \geq 4$, where the so-called interactive energy functional $E_N^{int} (G)$ is given by
	\begin{equation}\label{Interactive-Energy}
	  \begin{aligned}
	    E_N^{int} (G)= \sum_{|\alpha| \leq N-1} \Bigg\{&-\sum_{i,j=1}^{3} \l\partial^\alpha_x (I-P_{1}) f , \partial^\alpha_{x}\nabla_{x}\cdot b\zeta_i (v) \r_{L^2_{x,v}}\\
&+17\sum_{i=1}^{3}\l\partial_{i}\partial_{x}^{\alpha}(I-P_{1})f,\partial_{x}^{\alpha}c\zeta_{i}(v)\r_{L^{2}_{x,v}} +\l\partial_{x}^{\alpha}(I-P_{2})g,v\cdot\nabla_{x}\partial_{x}^{\alpha}d\r_{L^{2}_{x,v}} \Bigg\} \,,
	  \end{aligned}
	\end{equation}
	where $\zeta_i(v), \zeta_{ij}(v)$ are some linear combinations of the basis \eqref{e 13}.
\end{lemma}

\begin{remark}\label{Rmk-MM}
	We have notice that
	\begin{equation}\label{Interactive-Energy-Bnd}
	  \begin{aligned}
	    |E_{N}^{int} (G)| \lesssim\|G \|^{2}_{H^{N}_{x} L^{2}_{v}} \lesssim \mathcal{E}_{N} (G, \phi),
	  \end{aligned}
	\end{equation}
	and
	\begin{equation}\no
	  \begin{aligned}
	   \| \nabla_x b\|^{2}_{H^{N-1}_{x}} + \| \nabla_{x} c\|^2_{H^{N-1}_x} + \| \nabla_{x} (a+3c) \|^{2}_{H^{N-1}_{x}} + \|\nabla_{x}d\|^{2}_{H^{N-1}_{x}}\thicksim\| \nabla_{x}\mathbb{P}G\|^{2}_{H^{N-1}_{x} L^{2}_{v}} \,.
	  \end{aligned}
	\end{equation}
	Then the inequality \eqref{MM-Inq} can be simplified as
	\begin{equation}\label{MM-Inq-Simple}
	  \begin{aligned}
	    c_{0} \varepsilon \frac{d}{d t} E^{int}_N (G) + \| \nabla_{x}\mathbb{P}G\|^2_{H^{N-1}_x L^2_v} &\lesssim \mathcal{E}_N^\frac{1}{2} (G, \phi) \mathcal{D}_{N, \varepsilon} (G)+\mathcal{E}_N^\frac{1}{2} (G, \phi) \mathcal{D}_{N,\varepsilon}^\frac{1}{2}(G),
	  \end{aligned}
	\end{equation}
	for a positive constant $c_0 > 0$, which is independent of $\varepsilon$.
\end{remark}

\begin{proof}[Proof of Lemma \ref{Lmm-MM-Decomp}]
	We derive by taking divergence operator $\nabla_x \cdot $ on the balance law \eqref{a-b-c} for $b$, that
    \begin{equation}\label{3.3.1}
       \begin{aligned}
         &- \Delta_{x} (a+3c)=\varepsilon \partial_{t} \nabla_x \cdot b+2 \Delta_{x} c+ \nabla_x \cdot \l v \cdot \nabla_x (I-P_{1})f, v \r_{L^2_v}-\varepsilon\nabla_{x}\cdot[d\nabla_{x}\phi] \,.
       \end{aligned}
    \end{equation}

    For any multi-index $\alpha \in \mathbb{N}^3$ with $  |\alpha| \leq N-1$, from applying the derivative operator $\partial_x^\alpha$ to the equation \eqref{3.3.1}, multiplying by $\partial^\alpha_{x} (a+3c)$, integrating by parts over $x \in \R^3$, we deduces  that
    \begin{equation}\label{R1+R2+R3+R4}
      \begin{aligned}
        & \|\nabla_{x} \partial_{x}^\alpha (a+3c) \|^{2}_{L_{x}^{2}} =\underbrace{\varepsilon\l\partial_{x}^{\alpha}\partial_{t}(\nabla_{x}\cdot b),\partial_{x}^{\alpha}(a+3c)\r_{L^{2}_{x}}}_{R_{1}}+\underbrace{2\l\partial_{x}^{\alpha}\Delta_{x}c,\partial_{x}^{\alpha}(a+3c)\r_{L^{2}_{x}}}_{R_{2}}\\
        &+\underbrace{\l\nabla_{x}\cdot\partial_{x}^{\alpha}[v\cdot\nabla_{x}(I-P_{1})f],\partial_{x}^{\alpha}(a+3c)v\r_{L^{2}_{x,v}}}_{R_{3}}
        \underbrace{-\varepsilon\l\partial_{x}^{\alpha}[\nabla_{x}\cdot(d\nabla_{x}\phi)],\partial_{x}^{\alpha}(a+3c)\r_{L^{2}_{x}}}_{R_{4}}.
      \end{aligned}
    \end{equation}

    Next, we estimate $R_{i}(1\leq i\leq 4)$. For the term $R_{1}$, it can estimated by the balance law \eqref{a-b-c} for $a+3c$ and integration by parts over $x \in \R^3$ that
    \begin{equation}\label{R1}
      \begin{split}
         R_{1} & = \frac{d}{d t} \l \varepsilon\nabla_{x} \cdot \partial_{x}^\alpha b^{+} , \partial_{x}^{\alpha} (a+3c) \r_{L^{2}_{x}} - \l \partial_{x}^\alpha \nabla_x \cdot b, \varepsilon \partial_{x}^{\alpha} \partial_{t} (a+3c) \r_{L^2_x} = \varepsilon \frac{d}{d t} \l \nabla_{x} \cdot \partial_{x}^{\alpha} b, \partial_{x}^{\alpha} (a+3c) \r_{L^{2}_{x}} + \|\partial_{x}^{\alpha} \nabla_{x}\cdot b\|^{2}_{L^{2}_{x}} \,.
      \end{split}
    \end{equation}

    For terms $R_{2}$, $R_{3}$, by applying integration by parts and the H\"older inequality, we have
    \begin{equation}\label{R2}
      \begin{split}
        R_{2} \leq 2 \| \nabla_{x} \partial_{x} ^\alpha c\|_{L^{2}_{x}} \| \nabla_{x} \partial_{x}^{\alpha} (a+3c) \|_{L^{2}_{x}} \leq 8 \| \nabla_{x} \partial_{x}^\alpha c\|^{2}_{L^{2}_{x}} + \frac{1}{8} \| \nabla_{x} \partial^{\alpha}_{x} (a+3c)\|^{2}_{L^{2}_{x}} \,,
      \end{split}
    \end{equation}
    and
    \begin{equation}\label{R3}
      \begin{split}
        R_{3} & \leq \| \nabla_{x} \partial^{\alpha}_{x} (I-P_{1})f\|_{L^2_{x,v}} \| \nabla_x \partial^\alpha_x (a+3c) \|_{L^2_x} \l |v|^4 , 1 \r_{L^2_v}^\frac{1}{2}  \lesssim\|f\|_{H^{N}_{x}L^{2}_{v}}\|\nabla_{x}P_{1}f\|_{H_{x}^{N-1}L_{v}^{2}}\lesssim\mathcal{E}_{N}^{\frac{1}{2}}(G,\phi)\mathcal{D}^{\frac{1}{2}}_{N,\varepsilon}(G).
      \end{split}
    \end{equation}

   The term $R_{4}$ can be estimated since $|\alpha| \leq N-1$
    \begin{equation}\label{R4}
      \begin{aligned}
        R_{4}&=\varepsilon\l\partial_{x}^{\alpha}(d\nabla_{x}\phi),\nabla_{x}\partial_{x}^{\alpha}(a+3c)\r_{L^{2}_{x}}=\varepsilon\sum_{\tilde{\alpha}\leq\alpha}\l\nabla_{x}\partial_{x}^{\tilde{\alpha}}\phi\cdot\partial_{x}^{\alpha-\tilde{\alpha}}d,\nabla_{x}\partial_{x}^{\alpha}(a+3c)\r_{L^{2}_{x}}\\
        &\lesssim\varepsilon\|\nabla_{x}\phi\|_{H^{N}_{x}}\|\nabla_{x}P_{2}g\|_{H^{N-1}_{x}}\|\nabla_{x}P_{1}f\|_{H^{N-1}_{x}L^{2}_{v}}\lesssim\varepsilon\mathcal{E}_{N}^{\frac{1}{2}}(G,\phi)\mathcal{D}_{N,\varepsilon}(G).
      \end{aligned}
    \end{equation}

Plugging the bounds \eqref{R1}, \eqref{R2}, \eqref{R3} and \eqref{R4} into \eqref{R1+R2+R3+R4}, we gain
    \begin{equation}\label{3.3.16}
      \begin{aligned}
        & \frac{7}{8}\|\nabla_{x}\partial_{x}^{\alpha}(a+3c)\|_{L^{2}_{x}}^{2}-8\|\nabla_{x}\partial_{x}^{\alpha}c\|_{L^{2}_{x}}^{2}\lesssim\varepsilon\mathcal{E}_{N}^{\frac{1}{2}}(G,\phi)\mathcal{D}_{N,\varepsilon}(G)+\mathcal{E}_{N}^{\frac{1}{2}}(G,\phi)\mathcal{D}^{\frac{1}{2}}_{N,\varepsilon}(G).
      \end{aligned}
    \end{equation}
    for all $|\alpha| \leq N-1$.

For any multi-indexes $\alpha \in \mathbb{N}^3 (|\alpha| \leq N-1)$, we derive from the $b$-equation \eqref{3.3.7} and $c$-evolution \eqref{3.3.8} that for any fixed index $i \in \{ 1,2,3 \}$
    \begin{equation}
      \begin{split}
       & - \Delta_{x} \partial_{x}^{\alpha} b_{i} =  - \sum_{j \neq i} \partial_{j} \partial_{j} \partial_{x}^{\alpha} b_{i} - \partial_{i} \partial_{i} \partial_{x}^{\alpha} b_{i}\\
        =&  \sum_{j \neq i} \partial_{j} \partial_{x}^{\alpha} ( \partial_{i} b_{j}- l_{ij}-h_{ij}- m_{ij}) + \partial_i \partial_{x}^{\alpha}(\varepsilon \partial_{t} c- l_{i}- h_{i}- m_{i}) \\
        = & \sum_{j \neq i} \partial_i \partial_x^\alpha ( - \varepsilon \partial_t c + l_j + h_j + m_j) - \sum_{j \neq i} \partial_j \partial_x^\alpha ( l_{ij}+ h_{ij}+ m_{ij})  + \partial_i \partial^\alpha_x ( \varepsilon \partial_t c - l_i - h_i- m_i) \\
        = & - \varepsilon\partial_t \partial_i \partial^\alpha_x c- \sum_{j \neq i} \partial_j \partial_x^\alpha ( l_{ij}+ h_{ij} + m_{ij})
       + \partial_i \partial^\alpha_x \Big[ \sum_{j \neq i} ( l_j+ h_j+ m_j) - (l_i+ h_i+ m_i) \Big] \,.
      \end{split}
    \end{equation}

There are exist a certain linear combination $\zeta_{ij}(v)$ of the basis \eqref{e 13} since $l_i$, $l_{ij}$, $h_i$, $h_{ij}$ and $m_i$, $m_{ij}$ are the coefficients of $l$, $h$ and $m$ , that
    \begin{equation}
      \begin{aligned}
        \partial_i \partial^\alpha_x \Big[ \sum_{j \neq i} ( l_j + h_j + m_j) &- (l_i + h_i+ m_i) \Big] - \sum_{j \neq i} \partial_j \partial_x^\alpha ( l_{ij}+ h_{ij}+ m_{ij})= \sum_{j=1}^3 \partial_j \partial^\alpha_x \l l+ h + m , \zeta_{ij}(v) \r_{L^2_v} \,.
      \end{aligned}
    \end{equation}

    Applying the balance law \eqref{a-b-c} for $c$, we derive
    \begin{equation}\label{3.3.5}
      \begin{aligned}
        - \Delta_x \partial^\alpha_x b_i- \frac{1}{3} \partial_i \partial^\alpha_x \nabla_x \cdot b&= \partial_i \partial^\alpha_x \Big( \frac{1}{6} \l v \cdot \nabla_x (I-P_{1})f , |v|^2 \r_{L^2_v}-\frac{\varepsilon}{3}\nabla_{x}\phi\cdot\l(I-P_{2})g,v\r_{L^{2}_{v}}\Big)\\
        &+ \sum_{j=1}^3 \partial_j \partial^\alpha_x \l l + h + m, \zeta_{ij} (v) \r_{L^2_v},
      \end{aligned}
    \end{equation}
    for any fixed index $i \in \{ 1,2,3 \}$ and all $\alpha \in \mathbb{N}^3$ with $|\alpha| \leq N-1$. We multiply \eqref{3.3.5} by $\partial_x^\alpha b_i$, sum up over the index $i$and integrate over $x \in \mathbb{R}^3$. We thereby gain
    \begin{equation}\label{T1+T2+T3+T4+T5}
      \begin{aligned}
        & \| \nabla_x \partial^\alpha_x b\|^2_{L^2_x}+ \frac{1}{3}\|\l\partial^\alpha_x \nabla_x \cdot b \|^2_{L^2_x}
        =  \underbrace{\sum_{i,j=1}^3 \l\partial_{i}\partial^\alpha_x l ,\partial^\alpha_x b_i \zeta_{ij}(v) \r_{L^2_{x,v}} }_{T_1} +\underbrace{ \sum_{i,j=1}^3 \l \partial_i \partial^\alpha_x h , \partial^\alpha_x b_i\zeta_{ij}(v) \r_{L^2_{x,v}} }_{T_2} \\
         &+\underbrace{\sum_{i,j=1}^3 \l \partial_i \partial^\alpha_x m , \partial^\alpha_x b_i\zeta_{ij}(v) \r_{L^2_{x,v}} }_{T_3}+  \underbrace{\frac{1}{6}\l\partial_{i}\partial_{x}^{\alpha}[v\cdot\nabla_{x}(I-P_{1})f],\partial_{x}^{\alpha}b_{i}|v|^{2}\r_{L^{2}_{x,v}}}_{T_4} \\
        &+\underbrace{ \frac{2\varepsilon}{3}\l\partial_{x}^{\alpha}[\nabla_{x}\phi\cdot(I-P_{2})g],\partial_{i}\partial_{x}^{\alpha}b_{i}v\r_{L^2_{x,v}} }_{T_5},
      \end{aligned}
    \end{equation}
    for $|\alpha| \leq N-1$.

 We next estimate these terms $T_i\ (1\leq i\leq 5)$ in \eqref{T1+T2+T3+T4+T5}. Recalling the definition of $l$ in \eqref{l+h+m}, we have
    \begin{equation}\label{T1=T11+T12+T13}
      \begin{aligned}
        &T_1 =  - \varepsilon \frac{d}{d t} \sum_{i,j=1}^3 \l \partial_i \partial^\alpha_x (I-P_{1})f , \partial^\alpha_x b_i\zeta_{ij}(v) \r_{L^2_{x,v}}+ \underbrace{\sum_{i,j=1}^3 \l \partial_i \partial^\alpha_x (I-P_{1}) f, \varepsilon \partial_t \partial^\alpha_x b_i\zeta_{ij}(v) \r_{L^2_{x,v}} }_{T_{11}} \\
        & + \underbrace{\sum_{i,j=1}^3 \l \l v \cdot \nabla_x \partial^\alpha_x (I-P_{1})f, \partial_i \partial^\alpha_x b_i\zeta_{ij}(v) \r_{L^2_{x,v}} }_{T_{12}}+ \underbrace{\sum_{i,j=1}^3 \l \frac{1}{\varepsilon} \mathcal{L}_{1} \partial^\alpha_x (I-P_{1})f, \partial_i \partial^\alpha_x b_i\zeta_{ij}(v) \r_{L^2_{x,v}} }_{T_{13}}.
      \end{aligned}
    \end{equation}

 The term $T_{11}$ can be estimated
    \begin{equation}\no
      \begin{aligned}
        T_{11} &=  \underbrace{ \sum_{i,j=1}^3 \l \partial_i \partial^\alpha_x (I-P_{1})f, \partial_{i}\partial_{x}^{\alpha}\partial_{x_{i}}(a+5c)\zeta_{ij}\r_{L^2_{x,v}} }_{T_{111}}+\underbrace{ \sum_{i,j=1}^3 \l \partial^\alpha_x (I-P_{1}) f , \varepsilon\partial_i \partial^\alpha_x ( d\partial_{x_{i}}\phi) \zeta_{ij} (v) \r_{L^2_{x,v}} }_{T_{112}} \\
        &+\underbrace{\sum_{i,j=1}^3 \l \partial^\alpha_x(I-P_{1}) f, \partial_i \partial^\alpha_x \l v \cdot \nabla_x (I-P_{1})f, v_i \r_{L^2_v} \zeta_{ij}(v) \r_{L^2_{x,v}} }_{T_{113}},
      \end{aligned}
    \end{equation}
   by using the equation $\eqref{a-b-c}$ of the secondary balance law. Then the term $T_{111}$ is bounded by
    \begin{equation}\no
      \begin{aligned}
        T_{111} &\leq\sum_{i,j=1}^{3}\l\|\nabla_{i}\partial_{x}^{\alpha}(I-P_{1})f\|_{L^{2}_{v}},|\partial_{x}^{\alpha}\partial_{x_{i}}(a+5c)|\|\zeta_{i,j}\|_{L^{2}_{v}}\r_{L^{2}_{x}}\lesssim\|\nabla_{x}\partial_{x}^{\alpha}(I-P_{1})f\|_{L^{2}_{x,v}}\|\nabla_{x}\partial_{x}^{\alpha}(a+5c)\|_{L^{2}_{x}}\\
        &\lesssim\|(I-P_{1})f\|_{H^{N}_{x}L^{2}_{v}}\|\nabla_{x}P_{1}f\|_{H^{N-1}_{x}L^{2}_{v}}\lesssim\varepsilon\mathcal{E}^\frac{1}{2}_N (G, \phi) \mathcal{D}^\frac{1}{2}_{N,\varepsilon} (G) \,,
      \end{aligned}
    \end{equation}
    where we make use of \eqref{nu} and the H\"older inequality. For the term $T_{152}$, we estimate that
    \begin{equation}\no
      \begin{aligned}
        T_{112}&=\varepsilon\sum_{i,j=1}^{3}\sum_{\tilde{\alpha}\leq\alpha}\l\partial_{i}\partial_{x}^{\alpha}(I-P_{1})f,\partial_{x}^{\tilde{\alpha}}\partial_{x_{i}}\phi\partial_{x}^{\alpha-\tilde{\alpha}}d\zeta_{i,j}\r_{L^{2}_{x,v}}\\
        &\lesssim\varepsilon\|(I-P_{1})f\|_{H^{N}_{x}L^{2}_{v}}\|\nabla_{x}\phi\|_{H^{N}_{x}}\|\nabla_{x}P_{2}g\|_{H^{N-1}_{x}}\lesssim \varepsilon\mathcal{E}^\frac{1}{2}_N (G, \phi) \mathcal{D}_{N,\varepsilon} (G),
      \end{aligned}
    \end{equation}
 where we mark use of $\|f\|_{L^{6}_{x}(\mathbb{R}^{3})}\lesssim\|\nabla_{x}f\|_{L^{2}_{x}(\mathbb{R}^{3})}$. Analogously, the terms $T_{113}$ are bounded by
    \begin{equation}\no
      \begin{aligned}
        T_{113}& \lesssim \|(I-P_{1})f\|^{2}_{H^{N}_{x}L^{2}_{v}}\lesssim\|f\|_{H^{N}_{x} L^{2}_{v}}\|(I-P_{1})f\|_{H^N_{x} L^2_{v}(\nu)}\lesssim\varepsilon \mathcal{E}^\frac{1}{2}_N (G, \phi) \mathcal{D}^{\frac{1}{2}}_{N,\varepsilon} (G),
      \end{aligned}
    \end{equation}
   where we make use of \eqref{nu} and the H\"older inequality. In summary, we derive that
    \begin{equation}\no
      \begin{aligned}
        T_{11} \lesssim\varepsilon \mathcal{E}^\frac{1}{2}_N (G, \phi) \mathcal{D}^{\frac{1}{2}}_{N,\varepsilon} (G)+\varepsilon \mathcal{E}_N (G, \phi) \mathcal{D}^{\frac{1}{2}}_{N,\varepsilon} (G).
      \end{aligned}
    \end{equation}

    Moreover, it is easily to estimate that
    \begin{equation}\no
      \begin{aligned}
        T_{12} \lesssim\|f\|_{H_{x}^{N}L_{v}^{2}}\|\nabla_{x}P_{1}f\|_{H^{N-1}_{x}L^{2}_{v}}\lesssim\mathcal{E}^\frac{1}{2}_N (G, \phi) \mathcal{D}^{\frac{1}{2}}_{N,\varepsilon}(G),
      \end{aligned}
    \end{equation}
moreover we have
    \begin{equation}\no
      \begin{aligned}
        T_{13}&=\frac{1}{\varepsilon}\sum_{i,j=1}^{3}\l\partial_{x}^{\alpha}(I-P_{1})f, \partial_{i}\partial_{x}^{\alpha}b_{i}\mathcal{L}_{1}\zeta_{ij}\r_{L^{2}_{x,v}}\lesssim\frac{1}{\varepsilon}\sum_{i,j=1}^{3}\l\|\partial_{x}^{\alpha}(I-P_{1})f\|_{L^{2}_{x,v}(\nu)}\|\zeta_{ij}\|_{L^{2}_{v}(\nu)},|\partial_{i}\partial_{x}^{\alpha}b_{i}|\r_{L^{2}_{x}}\\
        &\lesssim\frac{1}{\varepsilon}\|(I-P_{1})f\|_{H^{N}_{x}L^{2}_{v}(\nu)}\|\nabla_{x}P_{1}f\|_{H^{N-1}_{x}L^{2}_{v}}\lesssim\mathcal{E}_{N}^{\frac{1}{2}}(G,\phi)\mathcal{D}_{N,\varepsilon}^{\frac{1}{2}}(G),
      \end{aligned}
    \end{equation}
    where the self-adjoint property of the linearized Boltzmann collision operator $\mathcal{L}_{1}$ and Lemma\ref{Lmm-Q} are used. As a result, we gain
    \begin{equation}\label{T1}
    \begin{aligned}
    T_{1}+ \varepsilon \frac{d}{d t} \sum_{i,j=1}^3 \l \partial_i \partial^\alpha_x (I-P_{1})f , \partial^\alpha_x b_i\zeta_{ij}(v) \r_{L^2_{x,v}}\lesssim\mathcal{E}_{N}^{\frac{1}{2}}(G,\phi)\mathcal{D}_{N,\varepsilon}^{\frac{1}{2}}(G)+\mathcal{E}_{N}^{\frac{1}{2}}(G,\phi)\mathcal{D}_{N,\varepsilon}(G),
    \end{aligned}
    \end{equation}
    for all $0 < \varepsilon \leq 1$ and $|\alpha| \leq N-1$. We derive from the similar estimates in \eqref{I1} by recalling the definition of $h, m$ in \eqref{l+h+m}, that
    \begin{equation}\label{T2-T3}
      \begin{aligned}
        T_2+T_{3}\lesssim \mathcal{E}_N^\frac{1}{2} (G, \phi) \mathcal{D}_{N,\varepsilon} (G) \,.
      \end{aligned}
    \end{equation}

    Similar argument as \eqref{T2-T3}, the term $T_4$ and $T_{5}$ can be bounded by
    \begin{equation}\label{T4+T5}
      \begin{aligned}
        T_{4}+T_{5}\lesssim\varepsilon\mathcal{E}_{N}^\frac{1}{2} (G, \phi) \mathcal{D}_{N, \varepsilon}(G)+\varepsilon\mathcal{E}_{N}^\frac{1}{2} (G, \phi) \mathcal{D}^\frac{1}{2} _{N, \varepsilon}(G).
      \end{aligned}
    \end{equation}

In summary, we can estimate \eqref{T1+T2+T3+T4+T5} by estimates above and the Young's inequality that
    \begin{equation}\label{3.3.15}
      \begin{aligned}
        \| \nabla_{x} \partial^{\alpha}_{x} b\|_{L^{2}_{x}}^{2}+ \frac{1}{3}\| \partial^\alpha_x \nabla_x \cdot b\|^{2}_{L^{2}_{x}}+& \varepsilon\frac{d}{d t} \sum_{i,j=1}^{3} \l \partial_1 \partial^\alpha_x (I-P_{1})f, \partial^{\alpha}_{x} b_i\zeta_{ij}(v) \r_{L^2_{x,v}}
        \lesssim \mathcal{E}_N^\frac{1}{2} (G, \phi) \mathcal{D}_{N, \varepsilon} (G)+ \mathcal{E}_N^\frac{1}{2} (G, \phi) \mathcal{D}_{N, \varepsilon}^\frac{1}{2} (G),
      \end{aligned}
    \end{equation}
    for all $0 < \varepsilon \leq 1$ and $|\alpha| \leq N-1$. Next, we estimate the norm $\| \nabla_{x} \partial^{\alpha}_{x} c\|_{L^{2}_{x}}$ for all $|\alpha| \leq N-1$. From \eqref{3.3.9}, we have
    \begin{equation}\label{3.3.12}
      \begin{aligned}
        - \Delta_{x} \partial^{\alpha}_{x} c = - \sum_{i=1}^{3} \partial_i \partial^\alpha_x \big(  l_{ci}+ h_{ci}+ m_{ci}\big) = \sum_{i=1}^{3} \partial_{i} \partial^{\alpha}_{x} \l l+ h+ m, \zeta_{i}(v) \r_{L^2_v},
      \end{aligned}
    \end{equation}
    for some certain linear combinations $\zeta_{i} (v)$ of the basis in \eqref{e 13}. Multiplying \eqref{3.3.12} by $\nabla_{x}\partial_{x}^{\alpha}c$ yields
    \begin{equation}\label{M1+M2+M3}
      \begin{split}
        &\| \nabla_x \partial_x^\alpha c\|^2_{L^2_x}= \underbrace{\sum_{i=1}^{3} \l \partial_{i}\partial^{\alpha}_{x} l, \partial^{\alpha}_{x} c\zeta_{i}(v) \r_{L^{2}_{x,v}} }_{M_{1}}+\underbrace{ \sum_{i=1}^{3} \l \partial_{i}\partial^{\alpha}_{x} h,\partial^{\alpha}_{x} c \zeta_{i}(v) \r_{L^{2}_{x,v}} }_{M_{2}}+\underbrace{\sum_{i=1}^3 \l \partial_{i}\partial^{\alpha}_{x} m, \partial^{\alpha}_{x} c \zeta_{i}(v) \r_{L^{2}_{x,v}} }_{M_{3}},
      \end{split}
    \end{equation}
    for all $|\alpha| \leq N-1$. With the same analysis method as above for $T_{i}(1\leq i\leq 5)$, we can get the estimate of $M_{i}(1\leq i\leq 3)$.
Consequently, we have
    \begin{equation}\label{3.3.14}
      \begin{split}
        \| \nabla_{x} \partial_{x}^{\alpha}c\|^{2}_{L^{2}_{x}}+ \varepsilon \frac{d}{d t} \sum_{i=1}^3 \l \partial_i \partial^\alpha_x (I-P_{1})f, \partial^{\alpha}_{x} c\zeta_{i}(v) \r_{L^{2}_{x,v}}\lesssim \mathcal{E}^\frac{1}{2} (G, \phi) \mathcal{D}_{N,\varepsilon} (G)+\mathcal{E}_N^\frac{1}{2} (G, \phi) \mathcal{D}_{N,\varepsilon}^\frac{1}{2}(G),
      \end{split}
    \end{equation}
    for  $0 < \varepsilon \leq 1$ and $|\alpha| \leq N-1$.

It remains to estimate the norm $\|\nabla_{x}\partial_{x}^{\alpha}d\|_{L^{2}_{x}}$. Recalling \eqref{decomposition2} and \eqref{l+h+m}, we obtain
\begin{equation}\label{d2}
\begin{aligned}
\varepsilon\partial_{t}d+v\cdot\nabla_{x}d=\hat{l}+\hat{h}+\hat{m},
\end{aligned}
\end{equation}
where
\begin{equation}\label{hat{l,h,m}}
\begin{aligned}
\hat{l}=&-\varepsilon\partial_{t}(I-P_{2})g-v\cdot\nabla_{x}(I-P_{2})g-\frac{1}{\varepsilon}\mathcal{L}_{2}(I-P_{2})g,\\
\hat{h}=&Q(g,f),\\
\hat{m}=&2v\cdot\nabla_{x}\phi+\varepsilon fv\cdot\nabla_{x}\phi-\varepsilon\nabla_{v}f\cdot\nabla_{x}\phi.
\end{aligned}
\end{equation}

Then, we have
\begin{equation}\no
\begin{aligned}
\varepsilon\l\partial_{t}d,v\r_{L^{2}_{v}}+\l v\cdot\nabla_{x}d,v\r_{L^{2}_{v}}=\l\hat{l}+\hat{h}+\hat{m},v\r_{L^{2}_{v}},
\end{aligned}
\end{equation}
which implies that
\begin{equation}\label{d1}
\begin{aligned}
\nabla_{x}d=\l\hat{l}+\hat{h}+\hat{m},v\r_{L^{2}_{v}}.
\end{aligned}
\end{equation}

For any multi-index $\alpha\in\mathbb{N}^{3}$ with $|\alpha|\leq N-1$, from applying the derivative operator $\partial_{x}^{\alpha}$ to the equation \eqref{d1}, multiplying by $\nabla_{x}\partial_{x}^{\alpha}d$, integrating by parts over $x\in\mathbb{R}^{3}$, we yields that
\begin{equation}
\begin{aligned}
\|\nabla_{x}\partial_{x}^{\alpha}d\|_{L^{2}_{x}}^{2}=\l\partial_{x}^{\alpha}\hat{l},v\cdot\nabla_{x}\partial_{x}^{\alpha}d\r_{L^{2}_{x,v}}+\l\partial_{x}^{\alpha}\hat{h},v\cdot\nabla_{x}\partial_{x}^{\alpha}d\r_{L^{2}_{x,v}}
+\l\partial_{x}^{\alpha}\hat{m},v\cdot\nabla_{x}\partial_{x}^{\alpha}d\r_{L^{2}_{x,v}}.
\end{aligned}
\end{equation}

With the same analysis method as above for $T_{i}$, it is easy to yields that
\begin{equation}\label{d}
\begin{aligned}
\|\nabla_{x}\partial_{x}^{\alpha}d\|_{L^{2}_{x}}^{2}+\varepsilon\frac{d}{d t}\l\partial_{x}^{\alpha}(I-P_{2})g, v\cdot\nabla_{x}\partial_{x}^{\alpha}d\r_{L^{2}_{x,v}}\lesssim\mathcal{E}^{\frac{1}{2}}_{N}(G,\phi)\mathcal{D}_{N,\varepsilon}(G)+\mathcal{E}^{\frac{1}{2}}_{N}(G,\phi)\mathcal{D}^{\frac{1}{2}}_{N,\varepsilon}(G).
\end{aligned}
\end{equation}

    From adding the bounds \eqref{3.3.15}, $17$ times of \eqref{3.3.14} and \eqref{d} to $2$ times of \eqref{3.3.16}, summing up for $|\alpha| \leq N-1$ and the Young's inequality, then we have
    \begin{equation}
      \begin{aligned}
\varepsilon\frac{d}{d t} E^{int}_N (G) +\| \nabla_{x} b\|^{2}_{H^{N-1}_{x}} &+ \| \nabla_{x} c \|^{2}_{H^{N-1}_x} +\|\nabla_{x}(a+3c)\|_{H^{N-1}_{x}}+\| \nabla_{x} d\|^{2}_{H^{N-1}_{x}}\\
&\lesssim\mathcal{E}_N^\frac{1}{2} (G, \phi) \mathcal{D}_{N, \varepsilon} (G)+\mathcal{E}_N^\frac{1}{2} (G, \phi) \mathcal{D}_{N,\varepsilon}^\frac{1}{2}(G),
      \end{aligned}
    \end{equation}
    for all $0 < \varepsilon \leq 1$ and $N \geq 4$, where the interactive energy functional $E^{int}_N (G)$ is given in \eqref{Interactive-Energy}. Consequently, we have completed the proof of the Lemma \ref{Lmm-MM-Decomp}.
\end{proof}

\subsection{Energy estimates for $(x,v)$-mixed derivatives}

In this subsection,  we will  derive a closed energy estimate. Because of the above two subsections, we only require to estimate the energy of $(x,v)$-mixed derivatives of the kinetic part $(\mathbb{I}-\mathbb{P})G$. We give the following lemma.

\begin{lemma}\label{Lmm-Mixed-Est}
	Suppose that $(g^{+}_{\varepsilon} (t,x,v), g^{-}_{\varepsilon} (t,x,v), \phi_{\varepsilon}(t,x,v))$ is the solution to the perturbed VPB system \eqref{VPB-g}-\eqref{f's intial} constructed in Lemma \ref{Lmm-Local}. Let $N \geq 4$ be any fixed integer. For any given $0 \leq k \leq N$ and $|\beta| \leq k$, there exist some positive constants $\varrho_k$, $\varrho_k^*$, $C_{|\beta|}$, $r_k$, $r_k^*$, $ \xi_k $ and $\eta_0$, which are independent of $\varepsilon$, such that
	\begin{equation}\label{Mixed-Inq-Closed}
	  \begin{aligned}
	    &\frac{1}{2} \frac{d}{d t} \Bigg(\varrho_{k} \|G\|^{2}_{H^{N}_{x} L^{2}_{v}} + \varrho_{k} \| \nabla_{x} \phi \|^{2}_{H^{N}_{x}} + \sum_{\substack{|\alpha| + |\beta| \leq N \\ |\beta| \leq k}} C_{|\beta|} \eta^2 \| \partial^\alpha_x \partial^\beta_v (\mathbb{I}-\mathbb{P})G\|^2_{L^2_{x,v}}+ \varrho_k^* \eta E_N^{int} (G)\Bigg)\\
	    & + \frac{r_k}{\varepsilon^2} \| (\mathbb{I}-\mathbb{P})G\|^2_{H^N_x L^2_v (\nu)} + r_k^* \eta \| \nabla_{x}\mathbb{P} G\|^2_{H^{N-1}_x L^2_v} + \frac{\xi_k}{\varepsilon^2} \eta^2\sum_{\substack{|\alpha| + |\beta| \leq N \\ |\beta| \leq k}} \| \partial^\alpha_x \partial^\beta_v (\mathbb{I}-\mathbb{P})G\|^2_{L^2_{x,v} (\nu)}\\
&\lesssim \mathcal{E}_N^\frac{1}{2} (G, \phi) \mathcal{D}_{N, \varepsilon} (G)+\mathcal{E}_N^\frac{1}{2} (G, \phi) \mathcal{D}_{N,\varepsilon}^\frac{1}{2}(G),
	  \end{aligned}
	\end{equation}
\end{lemma}
for all $0 < \varepsilon \leq 1$ and $0 < \eta \leq \eta_0$.

\begin{remark}\label{Rmk-Mixed}
	We will chose $\eta \in (0, \eta_0 ]$ mentioned in Lemma \ref{Lmm-Mixed-Est} so smaller later that the unsigned interactive energy functional $\varrho_k^* \eta E_N^{int} (G)$ can be dominated by $\varrho_k \|G \|^2_{H^N_x L^2_v}$ due to the bound \eqref{Interactive-Energy-Bnd} in Remark \ref{Rmk-MM}. Moreover, we introduce the instant energy functional as follow
	\begin{equation}
	  \begin{aligned}
	    \mathbb{E}_{N, \eta} (G, \phi) = & \varrho_N \|G\|^2_{H^N_x L^2_v} + \varrho_N \| \nabla_x \phi \|^2_{H^N_x} + \varrho_N^* \eta E_N^{int} (G) +\sum_{\substack{|\alpha| + |\beta| \leq N \\ |\beta| \leq N}} C_{|\beta|} \eta^2 \| \partial^\alpha_x \partial^\beta_v (\mathbb{I}-\mathbb{P}) G\|^2_{L^2_{x,v}}
	  \end{aligned}
	\end{equation}
	and the instant energy dissipative rate functional
	\begin{equation}
	  \begin{aligned}
	    \mathbb{D}_{N, \varepsilon, \eta} (G) = & \frac{r_N}{\varepsilon^2} \| (\mathbb{I}-\mathbb{P})G \|^2_{H^N_x L^2_v (\nu)} + r_N^* \eta \| \nabla_x \mathbb{P}G\|^2_{H^{N-1}_x L^2_v}+ \frac{\xi_k}{\varepsilon^2} \eta^2\sum_{\substack{|\alpha| + |\beta| \leq N \\ |\beta| \leq N}} \| \partial^\alpha_x \partial^\beta_v (\mathbb{I}-\mathbb{P})G\|^2_{L^2_{x,v} (\nu)} \,.
	  \end{aligned}
	\end{equation}
	We can easy verify that there exist a small $\eta_1 \in (0, \eta_0 ]$, independent of $\varepsilon$, such that
	\begin{equation}\label{Equivalences-Energies}
	  \begin{aligned}
	    \mathbb{E}_{N, \eta_1} (G, \phi) \thicksim \mathcal{E}_N (G, \phi) \,, \quad \mathbb{D}_{N, \varepsilon, \eta_1} (G) \thicksim \mathcal{D}_{N, \varepsilon} (G) \,.
	  \end{aligned}
	\end{equation}
	Consequently, by letting $k = N$ in the inequality \eqref{Mixed-Inq-Closed}, we have
	\begin{equation}\label{Closed-Energy-Inq}
	  \begin{aligned}
	    \frac{1}{2} \frac{d}{d t} \mathbb{E}_{N, \eta_1} (G, \phi) + \mathbb{D}_{N, \varepsilon, \eta_1} (G) &\lesssim \mathbb{E}^\frac{1}{2}_{N, \eta_1} (G, \phi) \mathbb{D}_{N, \varepsilon, \eta_1} (G)+\mathbb{E}_{N, \eta_1} (G, \phi) \mathbb{D}^\frac{1}{2}_{N, \varepsilon, \eta_1} (G)\\
&+\mathbb{E}^\frac{1}{2}_{N, \eta_1} (G, \phi) \mathbb{D}^\frac{1}{2}_{N, \varepsilon, \eta_1} (G).
	  \end{aligned}
	\end{equation}
\end{remark}

\begin{proof}[Proof of Lemma \ref{Lmm-Mixed-Est}]

Applying the microscopic projection $(I-P_{1})$ to first $f$-equation of \eqref{VPB-g-drop-eps} and $(I-P_{2})$ to the second $g$-equation of \eqref{VPB-g-drop-eps}, then
\begin{equation}\label{3.4.2}
  \begin{aligned}
\partial_{t}(I-P_{1})f+\frac{1}{\varepsilon^{2}}\mathcal{L}_{1}(I-P_{1})f&=\frac{1}{\varepsilon}Q(f,f)-\frac{1}{\varepsilon}(I-P_{1})(v\cdot\nabla_{x}f)+(I-P_{1})(gv\cdot\nabla_{x}\phi-\nabla_{x}\phi\cdot\nabla_{v}g),\\
\partial_{t}(I-P_{2})g+\frac{1}{\varepsilon^{2}}\mathcal{L}_{2}(I-P_{2})g&=\frac{1}{\varepsilon}Q(g,f)-\frac{1}{\varepsilon}(I-P_{2})(v\cdot\nabla_{x}g)+(I-P_{2})(fv\cdot\nabla_{x}\phi-\nabla_{v}f\cdot\nabla_{x}\phi),
  \end{aligned}
\end{equation}
where the relations $ (I-P_{2}) (v \cdot \nabla_x \phi)= 0 $ are used.

For all multi-indexes $\alpha$, $\beta \in \mathbb{N}^3$ with $|\alpha| + |\beta| \leq N$ and $\beta \neq 0$, we employ the mixed derivatives operator $\partial_{x}^{\alpha}\partial_{v}^{\beta}$ to \eqref{3.4.2}, take $L^2_{x,v}$-inner product via multiplying by $\partial^{\alpha}_{x} \partial^{\beta}_{v} (I-P_{1})f$ and $\partial^\alpha_x \partial^\beta_v (I-P_{2})g$, respectively, integrate by parts over $(x,v) \in \R^3 \times \R^3$. Consequently, we deduce that there exist two positive constants $\delta_{1}$ and $\delta_{2}$ which independent of $\varepsilon$ by Lemma \ref{Lmm-Coercivity-L}, such that
\begin{equation}\label{W1+W2+W3+W4+W5+W6+W7}
  \begin{aligned}
    & \frac{1}{2} \frac{d}{d t} \| \partial_{x}^{\alpha} \partial_{v}^{\beta} (\mathbb{I}-\mathbb{P}) G\|^2_{L^2_{x,v}} + \frac{2\delta_1}{\varepsilon^2}\| \partial_{x}^{\alpha} \partial_{v}^{\beta}  (\mathbb{I}-\mathbb{P}) G \|_{L^2_{x,v}(\nu)}^2- \frac{\delta_{2}}{\varepsilon^{2}} \sum_{\tilde{\beta} < \beta} \| \partial_x^\alpha \partial_v^{\tilde{\beta}}  (\mathbb{I}-\mathbb{P}) G  \|_{L^2_{x,v}(\nu)}^2 \\
    &\lesssim\underbrace{\frac{1}{\varepsilon}\big(\l\partial_{x}^{\alpha}\partial_{v}^{\beta}Q(f,f),\partial_{x}^{\alpha}\partial_{v}^{\beta}(I-P_{1})f\r_{L^{2}_{x,v}}+\l\partial_{x}^{\alpha}\partial_{v}^{\beta}Q(g,f),\partial_{x}^{\alpha}\partial_{v}^{\beta}(I-P_{2})g\r_{L^{2}_{x,v}}\big)}_{W_{1}}\\
    &\underbrace{-\frac{1}{\varepsilon}\big(\l\partial_{x}^{\alpha}\partial_{v}^{\beta}(I-P_{1})(v\cdot\nabla_{x}f),\partial_{x}^{\alpha}\partial_{v}^{\beta}(I-P_{1})f\r_{L^{2}_{x,v}}+\l\partial_{x}^{\alpha}\partial_{v}^{\beta}(I-P_{2})(v\cdot\nabla_{x}g),\partial_{x}^{\alpha}\partial_{v}^{\beta}(I-P_{2})g\r_{L^{2}_{x,v}}\big)}_{W_{2}}\\
    &+\underbrace{\l\partial_{x}^{\alpha}\partial_{v}^{\beta}(I-P_{1})(gv\cdot\nabla_{x}\phi),\partial_{x}^{\alpha}\partial_{v}^{\beta}(I-P_{1})f\r_{L^{2}_{x,v}}+\l\partial_{x}^{\alpha}\partial_{v}^{\beta}(I-P_{2})(fv\cdot\nabla_{x}\phi),\partial_{x}^{\alpha}\partial_{v}^{\beta}(I-P_{2})g\r_{L^{2}_{x,v}}}_{W_{3}}\\
    &\underbrace{-\big(\l\partial_{x}^{\alpha}\partial_{v}^{\beta}(I-P_{1})(\nabla_{x}\phi\cdot\nabla_{v}g),\partial_{x}^{\alpha}\partial_{v}^{\beta}(I-P_{1})f\r_{L^{2}_{x,v}}+\l\partial_{x}^{\alpha}\partial_{v}^{\beta}(I-P_{2})(\nabla_{x}\phi\cdot\nabla_{v}f),\partial_{x}^{\alpha}\partial_{v}^{\beta}(I-P_{2})g\r_{L^{2}_{x,v}}\big)}_{W_{4}}
  \end{aligned}
  \end{equation}

Then, we estimate $W_{i}(1\leq i\leq4)$. It's easy to control $W_{1}$ by using the similar estimates \eqref{I2-alph4}, $G=\mathbb{P}G+(\mathbb{I}-\mathbb{P})G$ and Lemma \ref{Lmm-Q}
\begin{equation}\label{W1}
  \begin{aligned}
  W_{1}\lesssim\mathcal{E}_{N}^{\frac{1}{2}}(G,\phi)\mathcal{D}_{N,\varepsilon}(G).
  \end{aligned}
\end{equation}

We divide $W_{2}$ into six parts since $P_{2}g=d(t,x)$, $G=\mathbb{P}G+(\mathbb{I}-\mathbb{P})G$ and $|\beta|\geq1$
\begin{equation*}
\begin{aligned}
&W_{2}=\underbrace{-\frac{1}{\varepsilon}\l\partial_{x}^{\alpha}\partial_{v}^{\beta}(v\cdot\nabla_{x}P_{1}f),\partial_{x}^{\alpha}\partial_{v}^{\beta}(I-P_{1})f\r_{L^{2}_{x,v}}}_{W_{21}}\underbrace{-\frac{1}{\varepsilon}\l\partial_{x}^{\alpha}\partial_{v}^{\beta}(v\cdot\nabla_{x}(I-P_{1})f),\partial_{x}^{\alpha}\partial_{v}^{\beta}(I-P_{1})f\r_{L^{2}_{x,v}}}_{W_{22}}\\
&+\underbrace{\frac{1}{\varepsilon}\l\partial_{x}^{\alpha}\partial_{v}^{\beta}P_{1}(v\cdot\nabla_{x}P_{1}f),\partial_{x}^{\alpha}\partial_{v}^{\beta}(I-P_{1})f\r_{L^{2}_{x,v}}}_{W_{23}}+\underbrace{\frac{1}{\varepsilon}\l\partial_{x}^{\alpha}\partial_{v}^{\beta}P_{1}(v\cdot\nabla_{x}(I-P_{1})f),\partial_{x}^{\alpha}\partial_{v}^{\beta}(I-P_{1})f\r_{L^{2}_{x,v}}}_{W_{24}}\\
&\underbrace{-\frac{1}{\varepsilon}\l\partial_{x}^{\alpha}\partial_{v}^{\beta}(v\cdot\nabla_{x}P_{2}g),\partial_{x}^{\alpha}\partial_{v}^{\beta}(I-P_{2})g\r_{L^{2}_{x,v}}}_{W_{25}}\underbrace{-\frac{1}{\varepsilon}\l\partial_{x}^{\alpha}\partial_{v}^{\beta}(v\cdot\nabla_{x}(I-P_{2})g),\partial_{x}^{\alpha}\partial_{v}^{\beta}(I-P_{2})g\r_{L^{2}_{x,v}}}_{W_{26}}.
\end{aligned}
\end{equation*}

Next, we estimate $W_{2i}(i=1,\cdots,6)$ one by one.
\begin{equation}\no
\begin{aligned}
W_{21}&=-\frac{1}{\varepsilon}\l\nabla_{x}\partial_{x}^{\alpha}\partial_{v}^{\beta-1}P_{1}f,\partial_{x}^{\alpha}\partial_{v}^{\beta}(I-P_{1})f\r_{L^{2}_{x,v}}-\frac{1}{\varepsilon}\l v\cdot\nabla_{x}\partial_{x}^{\alpha}\partial_{v}^{\beta}P_{1}f,\partial_{x}^{\alpha}\partial_{v}^{\beta}(I-P_{1})f\r_{L^{2}_{x,v}}\\
&\lesssim\frac{1}{\varepsilon}\|f\|_{H^{N-1}_{x}L^{2}_{v}}\|(I-P_{1})f\|_{H^{N}_{x,v}(\nu)}\lesssim\mathcal{E}_{N}^{\frac{1}{2}}(G,\phi)\mathcal{D}_{N,\varepsilon}^{\frac{1}{2}}(G),
\end{aligned}
\end{equation}
where we mark use of \eqref{nu} and Lemma \ref{P1f}. For the term $W_{22}$, we obtain
\begin{equation*}
\begin{aligned}
W_{22}&=-\frac{1}{\varepsilon}\l\nabla_{x}\partial_{x}^{\alpha}\partial_{v}^{\beta-1}(I-P_{1})f,\partial_{x}^{\alpha}\partial_{v}^{\beta}(I-P_{1})f\r_{L^{2}_{x,v}}\lesssim\mathcal{E}_{N}^{\frac{1}{2}}(G,\phi)\mathcal{D}_{N,\varepsilon}^{\frac{1}{2}}(G),
\end{aligned}
\end{equation*}
since $\l v\cdot\nabla_{x}\partial_{x}^{\alpha}\partial_{v}^{\beta}(I-P_{1})f,\partial_{x}^{\alpha}\partial_{v}^{\beta}(I-P_{1})f\r_{L^{2}_{x,v}}=0$.

For the term $W_{23}$, we have $P_{1}(v\cdot\nabla_{x}P_{1}f)=(\nabla_{x}a+5\nabla_{x}c)\cdot v+\frac{1}{3}\nabla_{x}\cdot b|v|^{2}$ since $P_{1}f=a+b\cdot v+c|v|^{2}$. Therefore, we yields
\begin{equation*}
\begin{aligned}
W_{23}&\lesssim\frac{1}{\varepsilon}\|\nabla_{x}P_{1}f\|_{H^{N}_{x}L^{2}_{v}}\|(I-P_{1})f\|_{H^{N}_{x,v}(\nu)}\lesssim\mathcal{E}_{N}^{\frac{1}{2}}(G,\phi)\mathcal{D}_{N,\varepsilon}^{\frac{1}{2}}(G).
\end{aligned}
\end{equation*}

For the term $W_{24}$, we have
\begin{equation*}
\begin{aligned}
W_{24}&\lesssim\frac{1}{\varepsilon}\|P_{1}(v\cdot\nabla_{x}\partial_{x}^{\alpha}(I-P_{1})f)\|_{L^{2}_{x,v}}\|\partial_{x}^{\alpha}\partial_{v}^{\beta}(I-P_{1})f\|_{L^{2}_{x,v}}\lesssim\frac{1}{\varepsilon}\|(I-P_{1})f\|_{H_{x,v}^{N}(\nu)}\|(I-P_{1})f\|_{H^{N}_{x,v}}\\
&\lesssim\mathcal{E}_{N}^{\frac{1}{2}}(G,\phi)\mathcal{D}^{\frac{1}{2}}_{N,\varepsilon}(G),
\end{aligned}
\end{equation*}
where we use Lemma \ref{P1f} and \eqref{nu}. Furthermore, we estimate $W_{25},W_{26}$ by the similar arguments above that
\begin{equation*}
\begin{aligned}
W_{25}+W_{26}\lesssim\mathcal{E}_{N}^{\frac{1}{2}}(G,\phi)\mathcal{D}^{\frac{1}{2}}_{N,\varepsilon}(G).
\end{aligned}
\end{equation*}

As a result, we can estimate $W_{2}$ as follow
\begin{equation}\label{W2}
\begin{aligned}
W_{2}\lesssim\mathcal{E}_{N}^{\frac{1}{2}}(G,\phi)\mathcal{D}^{\frac{1}{2}}_{N,\varepsilon}(G).
\end{aligned}
\end{equation}

We divide $W_{3}$ into four parts and use the similar argument like $S_{3}$ that
\begin{equation}\label{W3}
\begin{aligned}
W_{3}\lesssim\mathcal{E}_{N}^{\frac{1}{2}}(G,\phi)\mathcal{D}_{N,\varepsilon}(G).
\end{aligned}
\end{equation}

For the term $W_{4}$, we divide it into seven parts and  use the similar argument like $S_{3}$ that
\begin{equation*}
\begin{aligned}
&W_{4}=\l v\cdot\nabla_{x}\phi,\partial_{x}^{\alpha}\partial_{v}^{\beta}(I-P_{2})g\cdot \partial_{x}^{\alpha}\partial_{v}^{\beta}(I-P_{1})f\r_{L^{2}_{x,v}}-\l\nabla_{x}\phi\cdot\partial_{x}^{\alpha}\partial_{v}^{\beta}P_{2}g,\partial_{x}^{\alpha}\partial_{v}^{\beta}(I-P_{1})f\r_{L^{2}_{x,v}}\\
&-\l \nabla_{x}\phi\cdot\nabla_{v}\partial_{x}^{\alpha}\partial_{v}^{\beta}P_{1}f,\partial_{x}^{\alpha}\partial_{v}^{\beta}(I-P_{2})g\r_{L^{2}_{x,v}}-\sum_{\tilde{\alpha}\leq\alpha}\l\nabla_{x}\partial_{x}^{\tilde{\alpha}}\phi\cdot\nabla_{v}\partial_{x}^{\alpha-\tilde{\alpha}}\partial_{v}^{\beta-1}g,\partial_{x}^{\alpha}\partial_{v}^{\beta}(I-P_{1})f\r_{L^{2}_{x,v}}\\
&-\sum_{\tilde{\alpha}\leq\alpha}\l\nabla_{x}\partial_{x}^{\tilde{\alpha}}\phi\cdot\nabla_{v}\partial_{x}^{\alpha-\tilde{\alpha}}\partial_{v}^{\beta-1}f,\partial_{x}^{\alpha}\partial_{v}^{\beta}(I-P_{2})g\r_{L^{2}_{x,v}}+\l \partial_{x}^{\alpha}\partial_{v}^{\beta}P_{1}(\nabla_{x}\phi\cdot\nabla_{v}g),\partial_{x}^{\alpha}\partial_{v}^{\beta}(I-P_{1})f\r_{L^{2}_{x,v}}\\
&+\l\partial_{x}^{\alpha}\partial_{v}^{\beta}P_{2}(\nabla_{x}\phi\cdot\nabla_{v}f),\partial_{x}^{\alpha}\partial_{v}^{\beta}(I-P_{2})g\r_{L^{2}_{x,v}}\lesssim\mathcal{E}_{N}^{\frac{1}{2}}(G,\phi)\mathcal{D}_{N,\varepsilon}(G).
\end{aligned}
\end{equation*}

 Consequently, by estimates above and the Young's inequality, we deduce that
\begin{equation}\label{Mixed-Inq-1}
  \begin{aligned}
   \frac{1}{2} \frac{d}{d t}\| \partial_{x}^{\alpha} \partial_{v}^{\beta} (\mathbb{I}-\mathbb{P})G \|^{2}_{L^{2}_{x,v}}&+\frac{2\delta_1}{\varepsilon^{2}}\| \partial_x^\alpha \partial_v^\beta (\mathbb{I}-\mathbb{P})G \|_{L^2_{x,v}(\nu)}^2-\frac{\delta_{2}}{\varepsilon^2} \sum_{\tilde{\beta} < \beta} \| \partial_x^\alpha \partial_v^{\tilde{\beta}}(\mathbb{I}-\mathbb{P})G   \|_{L^2_{x,v} (\nu)}^2\\
    &\lesssim\mathcal{E}_{N}^{\frac{1}{2}}(G,\phi)\mathcal{D}_{N,\varepsilon}(G)+\mathcal{E}_{N}^{\frac{1}{2}}(G,\phi)\mathcal{D}_{N,\varepsilon}^{\frac{1}{2}}(G),
  \end{aligned}
\end{equation}
for all $|\alpha| + |\beta| \leq N$ with $\beta \neq 0$ and for all $0 < \varepsilon \leq 1$.

Let $\eta > 0$ be sufficiently small number to be given latter. Plugging \eqref{Spatial-Bnd} in Lemma \ref{Lmm-Spatial} and $\eta$ times of \eqref{MM-Inq-Simple} in Remark \ref{Rmk-MM} to the $\eta^2$ times of the above inequality \eqref{Mixed-Inq-1}. Then there exists a small positive number $\eta_0 > 0$, independent of $\varepsilon$, such that for all $0 < \eta \leq \eta_0$
\begin{equation}\label{Mixed-Inq-2}
  \begin{aligned}
    & \frac{1}{2} \frac{d}{d t}\big(\| G\|^2_{H^N_x L^2_v} + \| \nabla_x \phi \|^2_{H^N_x} + \eta^2 \| \partial^\alpha_x \partial^\beta_v (\mathbb{I}-\mathbb{P}) G \|^2_{L^2_{x,v}}+ c_0 \eta E_N^{int} (G)\big) +\frac{\delta}{\varepsilon^2} \| (\mathbb{I}-\mathbb{P})G\|^2_{H^N_x L^2_v (\nu)} \\
    &+ \frac{\eta}{2} \| \nabla_x \mathbb{P} G\|^2_{H^{N-1}_x L^2_v}
    + \frac{\eta^2 \delta_{1}}{\varepsilon^2} \| \partial^\alpha_x \partial^\beta_x (\mathbb{I}-\mathbb{P})G\|^2_{L^2_{x,v} (\nu)}\lesssim\frac{1}{\varepsilon^2}\sum_{\tilde{\beta}< \beta} \| \partial^\alpha_x \partial^{\tilde{\beta}}_v (\mathbb{I}-\mathbb{P})G\|^2_{L^2_{x,v} (\nu)} \\
    &+\mathcal{E}_N^\frac{1}{2} (G, \phi) \mathcal{D}_{N, \varepsilon} (G)+\mathcal{E}_N^\frac{1}{2} (G, \phi) \mathcal{D}_{N,\varepsilon}^\frac{1}{2}(G),
  \end{aligned}
\end{equation}
for all $|\alpha| + |\beta| \leq N$ with $\alpha \neq 0$ and for all $0 < \varepsilon \leq 1$.

We notice that the quantity $ \frac{1}{\varepsilon^2} \sum_{\tilde{\beta} < \beta} \| \partial^\alpha_x \partial^{\tilde{\beta}}_v (\mathbb{I}-\mathbb{P})G\|^2_{L^2_{x,v} (\nu)} $ in the right-hand side of \eqref{Mixed-Inq-2} is still not controlled. However, the orders of $v$-derivatives in this quantity is strictly less than $|\beta|$, so that we can apply an induction over $|\beta| = k$, which ranges between $0$ and $N$, to obtain the inequality \eqref{Mixed-Inq-Closed}. For simplicity, we omit the details of the induction, and the proof of Lemma \ref{Lmm-Mixed-Est} is completed.
\end{proof}

\subsection{Global classical solutions: proof of Theorem \ref{Thm-global}}

From the differential inequality \eqref{Closed-Energy-Inq} in Remark \ref{Rmk-Mixed} and the energy bound \eqref{2.3} in Lemma \ref{Lmm-Local}, it is easy to deduce that for any $[t_1, t_2] \subseteq [0,T]$ and $0 < \varepsilon \leq 1$
\begin{equation}\no
  \begin{aligned}
    &\big| \mathbb{E}_{N, \eta_1} (G_{\varepsilon}, \phi_\varepsilon) (t_2) - \mathbb{E}_{N, \eta_1} (G_{\varepsilon}, \phi_\varepsilon) (t_1) \big| \lesssim \int_{t_1}^{t_2} \mathbb{E}_{N,\eta_1}^\frac{1}{2} (G_{\varepsilon}, \phi_\varepsilon) \mathbb{D}_{N, \varepsilon, \eta_1} (G_{\varepsilon}) d t\\
    &+\int_{t_1}^{t_2} \mathbb{E}_{N,\eta_1}^\frac{1}{2} (G_\varepsilon, \phi_\varepsilon) \mathbb{D}^\frac{1}{2}_{N, \varepsilon, \eta_1} (G_{\varepsilon}) d t\lesssim \sup_{0 \leq t \leq T} \mathcal{E}_N^\frac{1}{2} (G_{\varepsilon}, \phi_\varepsilon) \int_{t_1}^{t_2} \mathcal{D}_{N,\varepsilon} (G_{\varepsilon}) d t\\
     &+\sup_{0 \leq t \leq T} \mathcal{E}_N^\frac{1}{2} (G_{\varepsilon}, \phi_\varepsilon) \int_{t_1}^{t_2} \mathcal{D}^\frac{1}{2}_{N,\varepsilon} (G_{\varepsilon}) \d t\lesssim \int_{t_1}^{t_2} \frac{1}{\varepsilon^2} \| (\mathbb{I}-\mathbb{P}) G_{\varepsilon} \|^2_{H^N_{x,v} (\nu)} d t\\
     &+\int_{t_1}^{t_2} \frac{1}{\varepsilon} \| (\mathbb{I}-\mathbb{P}) G_{\varepsilon} \|_{H^N_{x,v} (\nu)} d t \lesssim|t_2 - t_1 | + |t_2 - t_1 |^{\frac{1}{2}} \rightarrow 0 \quad \textrm{as} \quad t_2 \rightarrow t_1 \,.
  \end{aligned}
\end{equation}

Therefore, the energy functional $ \mathbb{E}_{N, \eta_1} (G_\varepsilon, \phi_\varepsilon) $ of the local solution $(g_\varepsilon^{+}, g_\varepsilon^{-} , \phi_{\varepsilon})$ to \eqref{VPB-g}-\eqref{f's intial} constructed in Lemma \ref{Lmm-Local} is continuous in $t \in [0,T]$.

Then, we define
\begin{equation}\no
  \begin{aligned}
    T^* = \sup \Big\{ \xi \geq 0 ; \ C \sup_{t \in [0, \xi]} \mathbb{E}_{N, \eta_1}^\frac{1}{2} (G_\varepsilon, \phi_\varepsilon) (t) \leq \frac{1}{2} \Big\} \geq 0 \,.
  \end{aligned}
\end{equation}

By $ \mathbb{E}_{N, \eta_1} (G_{\varepsilon, 0} , \phi_{\varepsilon, 0}) \thicksim \mathcal{E}_N (G_{\varepsilon, 0}, \phi_{\varepsilon, 0}) $ and the initial condition in Theorem \ref{Thm-global}, we have
\begin{equation}\no
  \begin{aligned}
    \mathbb{E}_{N, \eta_1} (G_{\varepsilon, 0} , \phi_{\varepsilon, 0}) \leq C_0 \mathcal{E}_N (G_{\varepsilon, 0}, \phi_{\varepsilon, 0}) \leq C_0 \tau_0
  \end{aligned}
\end{equation}
for a positive constant $C_0 > 0$, where $\tau_0 \in (0,1]$ is small to be given latter.  Taking $0 < \tau_0 \leq \min \{ 1 , \delta , \frac{1}{16 C^2 C_0} \}$, where $\delta > 0$ is mentioned in Lemma \ref{Lmm-Local}, we deduce
\begin{equation}\label{Initial-Bnd}
  \begin{aligned}
    C \mathbb{E}_{N, \eta_1}^\frac{1}{2} (G_\varepsilon, \phi_\varepsilon) (0) = C \mathbb{E}_{N, \eta_1}^\frac{1}{2} (G_{\varepsilon, 0}, \phi_{\varepsilon, 0}) \leq C \sqrt{C_0 \tau_0} \leq \frac{1}{4} < \frac{1}{2} \,.
  \end{aligned}
\end{equation}

Then $T^* > 0$ by the continuity of $\mathbb{E}_{N, \eta_1} (G_\varepsilon, \phi_\varepsilon) (t)$. Consequently, we derive from the definition of $T^*$ and the inequality \eqref{Closed-Energy-Inq} in Remark \ref{Rmk-Mixed} that for all $t \in [0, T^*]$ and $0 < \varepsilon \leq 1$
\begin{equation}\no
  \begin{aligned}
    \frac{d}{d t} \mathbb{E}_{N, \eta_1} (G_\varepsilon , \phi_\varepsilon(t,x)) + \mathbb{D}_{N, \varepsilon \eta_1} (G_\varepsilon) \leq 0 \,.
  \end{aligned}
\end{equation}

Integrating the above inequality on $[0,t]$ for any $t \in [0, T^*]$, we have
\begin{equation}\label{Uniform-Global-Bnd-Inst}
  \begin{aligned}
    \mathbb{E}_{N,\eta_1} (G_\varepsilon, \phi_\varepsilon) (t) + \int_0^t \mathbb{D}_{N, \varepsilon, \eta_1} (G_\varepsilon) (\xi) d \xi \leq \mathbb{E}_{N, \eta_1} (G_{\varepsilon, 0} , \phi_{\varepsilon, 0}) \leq C_0 \tau_0,
  \end{aligned}
\end{equation}
uniformly for all $0 < \varepsilon \leq 1$, which immediately implies by the initial bound \eqref{Initial-Bnd}, that
\begin{equation}\no
  \begin{aligned}
    C \sup_{t \in [0, \tau]} \mathbb{E}_{N, \eta_1}^\frac{1}{2} (G_\varepsilon, \phi_\varepsilon) (t) \leq \frac{1}{4} < \frac{1}{2} \,.
  \end{aligned}
\end{equation}

Therefore, the continuity of $\mathbb{E}_{N, \eta_1} (G_\varepsilon, \phi_\varepsilon) (t)$ and the definition of $T^*$ imply that $T^* = + \infty$. In other words,  we can the local extend a local-in-time solution $(g_\varepsilon^{+} (t,x,v), g_\varepsilon^{-} (t,x,v), \phi_{\varepsilon}(t,x))$ constructed in Lemma \ref{Lmm-Local} to a global-in-time solution. Moreover, the uniform energy bound \eqref{Uniform-Global-Bnd} can be derived from \eqref{Equivalences-Energies} and \eqref{Uniform-Global-Bnd-Inst}. Then, the proof of Theorem \ref{Thm-global} is completed.  \qquad\qquad\qquad\qquad\qquad\qquad\qquad\qquad\qquad\qquad\qquad\qquad\ \ \  $\square$

\section{Limit to Two-fluid Incompressible NSFP Equations with Ohm's Law}\label{Sec: Limits}

In this section, my goal is deriving the two-fluid incompressible NSFP system with Ohm's law \eqref{NSFP} from the perturbed VPB system \eqref{VPB-g}-\eqref{f's intial} as $\varepsilon \rightarrow 0$, based on the uniform global energy bound \eqref{Uniform-Global-Bnd} in Theorem \ref{Thm-global}.

\subsection{Limits from the global energy estimate}

By Theorem \ref{Thm-global}, we have the Cauchy problem \eqref{VPB-g}-\eqref{f's intial} admits a global solution $(g_{\varepsilon}^{+}(t,x,v),g_\varepsilon^{-} (t,x,v))\in L^\infty (\R^+; H^N_{x,v}) $ and a uniform global energy estimate \eqref{Uniform-Global-Bnd}, there exists a positive constant $C$, independent of $\varepsilon$, such that
\begin{equation}\label{Unf-Bnd-1}
  \begin{aligned}
    \sup_{t \geq 0} \big( \|\frac{g_{\varepsilon}^{+}(t)+g_{\varepsilon}^{-}(t)}{2}\|^2_{H^N_{x,v}} +\| \frac{g_{\varepsilon}^{+}(t)-g_{\varepsilon}^{-}(t)}{2}\|^2_{H^N_{x,v}}+ \| \nabla_x \phi_\varepsilon (t) \|^2_{H^N_x}  \big) \leq C \,
  \end{aligned}
\end{equation}
and
\begin{equation}\label{Unf-Bnd-2}
  \begin{aligned}
    \int_{0}^{T}\| (I-P_{1})(\frac{g_{\varepsilon}^{+}(t)+g_{\varepsilon}^{-}(t)}{2})\|^2_{H^N_{x,v} (\nu)}+\| (I-P_{2})(\frac{g_{\varepsilon}^{+}(t)-g_{\varepsilon}^{-}(t)}{2})\|^2_{H^N_{x,v} (\nu)} d t \leq C \varepsilon^2,
  \end{aligned}
\end{equation}
for any given $T>0$.

From the uniform energy bound \eqref{Unf-Bnd-1}, there exist $g(t,x,v),\tilde{g}(t,x,v)\in L^\infty (\R^+ ; H^N_{x,v})$ and $\phi (t,x) \in L^\infty (\R^+ ; H^{N+1}_x)$ such that
\begin{equation}\label{Convgc-g-phi}
  \begin{aligned}
    & \frac{g_{\varepsilon}^{+}+g_{\varepsilon}^{-}}{2} \rightarrow g\quad \quad \ \ \textrm{weakly-}\star \ \textrm{for } t \geq 0, \ \quad \textrm{weakly in } H^N_{x,v} \,, \\
    & \frac{g_{\varepsilon}^{+}-g_{\varepsilon}^{-}}{2} \rightarrow \tilde{g}\quad \quad \ \ \textrm{weakly-}\star \ \textrm{for } t \geq 0, \ \quad \textrm{weakly in } H^N_{x,v} \,, \\
    & \nabla_x \phi_\varepsilon \rightarrow \nabla_{x}\phi \quad \textrm{weakly-}\star \ \textrm{for } t \geq 0, \ \quad \textrm{weakly in } H^N_x
  \end{aligned}
\end{equation}
as $\varepsilon \rightarrow 0$. We still employ the original notations of the sequences to denote by the subsequences throughout this paper for convenience although the limits may hold for some subsequences. The energy dissipation bound \eqref{Unf-Bnd-2} and the inequality $\| (I-P_{1})f\|^2_{H^N_{x,v}} \lesssim \| (I-P_{1})f \|^2_{H^N_{x,v} (\nu)},\| (I-P_{2})g\|^2_{H^N_{x,v}} \lesssim \| (I-P_{2})g \|^2_{H^N_{x,v} (\nu)}  $ implied by \eqref{nu} yield
\begin{equation}\label{Convgc-g-perp}
  \begin{aligned}
    &(I-P_{1})(\frac{g_\varepsilon^{+}+g_{\varepsilon}^{-}}{2}) \rightarrow 0 \quad \textrm{strongly in } \ L^2(\R^+; H^N_{x,v}),\\
    &(I-P_{2})( \frac{g_\varepsilon^{+}-g_{\varepsilon}^{-}}{2}) \rightarrow 0 \quad \textrm{strongly in } \ L^2(\R^+; H^N_{x,v})
  \end{aligned}
\end{equation}
as $\varepsilon \rightarrow 0$. Then, the first convergence in \eqref{Convgc-g-phi} and \eqref{Convgc-g-perp} yield
\begin{equation}\no
  \begin{aligned}
    (I-P_{1}) g= 0, (I-P_{2})\tilde{g}=0,
  \end{aligned}
\end{equation}
which imply that there exist  functions $\rho(t,x)$, $u(t,x)$, $\theta(t,x)$, $n(t,x) \in L^\infty (\R^+; H^N_x)$ such that
\begin{equation}\label{Limit-g}
  \begin{aligned}
    g(t,x,v) = \rho(t,x) + u(t,x)\cdot v + \theta(t,x) ( \frac{|v|^2}{2} - \frac{3}{2} ) , \tilde{g}(t,x,v)=n(t,x).
  \end{aligned}
\end{equation}

Next, we define the following fluid variables
\begin{equation}\label{define jw}
  \begin{aligned}
    &\rho_{\varepsilon} = \l \frac{g_{\varepsilon}^{+}+g_{\varepsilon}^{-}}{2} , 1 \r_{L^2_v} \,, \ u_{\varepsilon} = \l \frac{g_\varepsilon^{+}+g_{\varepsilon}^{-}}{2}, v \r_{L^2_v} , \theta_\varepsilon = \l \frac{g_\varepsilon^{+}+g_{\varepsilon}^{-}}{2}, \frac{|v|^2}{3} - 1 \r_{L^2_v},\\ & n_{\varepsilon}=\l\frac{g_\varepsilon^{+}+g_{\varepsilon}^{-}}{2},1\r_{L^{2}_{v}},
    j_{\varepsilon}=\frac{1}{\varepsilon}\l g_{\varepsilon}^{+}-g_{\varepsilon}^{-},v\r_{L^{2}_{v}}, w_{\varepsilon}=\frac{1}{\varepsilon}\l g_{\varepsilon}^{+}-g_{\varepsilon}^{-}, \frac{|v|^{2}}{3}-1\r_{L^{2}_{v}} \,.
  \end{aligned}
\end{equation}
where $n_{\varepsilon}$, $j_{\varepsilon}$ and $w_{\varepsilon}$ are called the electric charge, the electric current and the internal electric energy, respectively.

Taking inner products with the first $f$-equation of the perturbed VPB equation \eqref{VPB-g} in $L^2_v$ by $1$, $v$ and $ \frac{|v|^2}{3}-1 $, we have the local conservation laws:
\begin{equation}\label{Consevtn-Law-rho-u-theta-phi}
  \begin{cases}
    &\partial_{t}\rho_{\varepsilon}+\frac{1}{\varepsilon}\nabla_{x}\cdot u_{\varepsilon}=0,\\
    &\partial_{t}u_{\varepsilon}+\frac{1}{\varepsilon}\nabla_{x}(\rho_{\varepsilon}+\theta_{\varepsilon})+\nabla_{x}\cdot\l\hat{A},\frac{1}{\varepsilon}\mathcal{L}_{1}(I-P_{1})\big(\frac{g_{\varepsilon}^{+}+g_{\varepsilon}^{-}}{2}\big)\r_{L^{2}_{v}}=\frac{1}{2}n_{\varepsilon}\nabla_{x}\phi_{\varepsilon},\\
    &\partial_{t}\theta_{\varepsilon}+\frac{2}{3\varepsilon}\nabla_{x}\cdot u_{\varepsilon}+\frac{2}{3}\nabla_{x}\cdot\l\hat{B},\frac{1}{\varepsilon}\mathcal{L}_{1}(I-P_{1})\big(\frac{g_{\varepsilon}^{+}+g_{\varepsilon}^{-}}{2}\big)\r_{L^{2}}=\frac{\varepsilon}{2}j_{\varepsilon}\cdot\nabla_{x}\phi_{\varepsilon},\\
    &\Delta_{x}\phi_{\varepsilon}=n_{\varepsilon}.
  \end{cases}
\end{equation}

In order to reduce the equations about $n_{\varepsilon},j_{\varepsilon}$ and $w_{\varepsilon}$, we take inner products with the second $g$-equation of the perturbed VPB equation \eqref{VPB-g} in $L^2_v$ by $1$, $\tilde{\Phi}(v)$ and $\tilde{\Psi}(v)$ which defined in \eqref{hat{PP}}, we have
\begin{equation}\label{Consevtn-Law-n-j-w}
  \begin{cases}
    &\partial_{t}n_{\varepsilon}+\nabla_{x}\cdot j_{\varepsilon}=0,\\
    &j_{\varepsilon}=n_{\varepsilon} \mathcal{P}u_{\varepsilon}+\sigma(\nabla_{x}\phi_{\varepsilon}-\frac{1}{2}\nabla_{x}n_{\varepsilon})+j_{R,\varepsilon}\\
    &w_{\varepsilon}=n_{\varepsilon}\theta_{\varepsilon}+w_{R,\varepsilon},
  \end{cases}
\end{equation}
where $\mathcal{P}$ is the Leray projection on $\R^{3}$, $j_{R,\varepsilon}$ and $w_{R,\varepsilon}$ are defined as follow
\begin{equation}
\begin{aligned}
j_{R,\varepsilon}&=-n_{\varepsilon}\mathcal{P}^{\perp}u_{\varepsilon}-\varepsilon\l\partial_{t}(g_{\varepsilon}^{+}-g_{\varepsilon}^{-}),\tilde{\Phi}\r_{L^{2}_{v}}-\l v\cdot\nabla_{x}(I-P_{2})(g_{\varepsilon}^{+}-g_{\varepsilon}^{-}), \tilde{\Phi}\r_{L^{2}_{v}}\\
&+\l Q(n_{\varepsilon},(I-P_{1})(g_{\varepsilon}^{+}+g_{\varepsilon}^{-})),\tilde{\Phi}\r_{L^{2}_{v}}+\varepsilon \l Q(h_{\varepsilon}^{+}-h_{\varepsilon}^{-},g_{\varepsilon}^{+}+g_{\varepsilon}^{-}),\tilde{\Phi}\r_{L^{2}_{v}}\\
&-\varepsilon\l\nabla_{x}\phi_{\varepsilon}\cdot\nabla_{v}(g_{\varepsilon}^{+}+g_{\varepsilon}^{-}),\tilde{\Phi}\r_{L^{2}_{v}}+\varepsilon\l(g_{\varepsilon}^{+}+g_{\varepsilon}^{-})v\cdot\nabla_{x}\phi_{\varepsilon},\tilde{\Phi}\r_{L^{2}_{v}},\\
 w_{R,\varepsilon}&=-\frac{2\varepsilon}{3}\l\partial_{t}(g_{\varepsilon}^{+}-g_{\varepsilon}^{-}),\tilde{\Psi}\r_{L^{2}_{v}}-\frac{2}{3}\l v\cdot\nabla_{x}(I-P_{2})(g_{\varepsilon}^{+}-g_{\varepsilon}^{-}), \tilde{\Psi}\r_{L^{2}_{v}}\\
 &-\frac{2\varepsilon}{3}\l\nabla_{x}\phi_{\varepsilon}\cdot\nabla_{v}(g_{\varepsilon}^{+}+g_{\varepsilon}^{-}),\tilde{\Psi}\r_{L^{2}_{v}}+\frac{2\varepsilon}{3}\l(g_{\varepsilon}^{+}+g_{\varepsilon}^{-})v\cdot\nabla_{x}\phi_{\varepsilon},\tilde{\Psi}\r_{L^{2}_{v}}\\
 &+\frac{2}{3}\l Q(n_{\varepsilon},(I-P_{1})(g_{\varepsilon}^{+}+g_{\varepsilon}^{-})),\tilde{\Psi}\r_{L^{2}_{v}}+\frac{2\varepsilon}{3}\l Q(h_{\varepsilon}^{+}-h_{\varepsilon}^{-},g_{\varepsilon}^{+}+g_{\varepsilon}^{-}),\tilde{\Psi}\r_{L^{2}_{v}},\\
\end{aligned}
\end{equation}

Furthermore, the uniform bound \eqref{Unf-Bnd-1} yields
\begin{equation}\label{Unf-Bnd-3}
  \begin{aligned}
    \sup_{t \geq 0} \big( \nabla_{x}\phi_{\varepsilon}\|_{H^{N}_{x}}+\| \rho_\varepsilon\|_{H^N_x} + \| u_\varepsilon\|_{H^N_x} + \| \theta_\varepsilon \|_{H^N_x}+\|n_{\varepsilon}\|_{H^{N}_{x}}\big) \leq C \,.
  \end{aligned}
\end{equation}

Recalling \eqref{nu}, \eqref{Unf-Bnd-2} and \eqref{define jw}, we have
\begin{equation}\label{j and w}
\begin{aligned}
\int_{0}^{T}\|j_{\varepsilon}\|_{H^{N}_{x}}^{2}dt\lesssim 1, \int_{0}^{T}\|w_{\varepsilon}\|_{H^{N}_{x}}^{2}dt\lesssim 1,
\end{aligned}
\end{equation}
for any $T>0$, $N\geq 4$, where we use $j_{\varepsilon}=\frac{1}{\varepsilon}\l(I-P_{2}) (g_{\varepsilon}^{+}-g_{\varepsilon}^{-}),v\r_{L^{2}_{x}}$ and $w_{\varepsilon}=\frac{2}{3\varepsilon}\l(I-P_{2}) (g_{\varepsilon}^{+}-g_{\varepsilon}^{-}), \frac{|v|^{2}}{2}-\frac{3}{2}\r_{L^{2}_{x}}$.

Therefore,  we derive the following convergence from the convergence \eqref{Convgc-g-phi} and the form of limit function $g(t,x,v),\tilde{g}(t,x,v)$ given in \eqref{Limit-g} that
\begin{equation}\label{Con-rho-u-theta-n}
  \begin{aligned}
    & \rho_\varepsilon= \l \frac{g_{\varepsilon}^{+}+g_{\varepsilon}^{-}}{2} , 1 \r_{L^2_v} \rightarrow \l g , 1 \r_{L^2_v} = \rho \,, u_\varepsilon = \l \frac{g_{\varepsilon}^{+}+g_{\varepsilon}^{-}}{2}  , v \r_{L^2_v} \rightarrow \l g , v \r_{L^2_v} = u \,, \\
    & \theta_\varepsilon= \l \frac{g_{\varepsilon}^{+}+g_{\varepsilon}^{-}}{2}  , \frac{|v|^2}{3} - 1 \r_{L^2_v} \rightarrow \l g , \frac{|v|^2}{3} - 1 \r_{L^2_v} = \theta\,, n_{\varepsilon}=\l g_{\varepsilon}^{+}-g_{\varepsilon}^{-}, 1\r_{L^{2}_{v}}\rightarrow \l\tilde{g},1\r_{L^{2}_{v}}=n,
  \end{aligned}
\end{equation}
weakly-$\star$ for $t \geq 0$,  weakly in $H^{N}_x$ and strongly in $H^{N-1}_{x}$ as $\varepsilon \rightarrow 0$.  For the terms $j_{\varepsilon}$ and $w_{\varepsilon}$, there exist two functions $j,w\in L^{2}(\R ^{+},H^{N}_{x})$, such that
\begin{equation}\label{Con-j-w}
\begin{aligned}
j_{\varepsilon}=\frac{1}{\varepsilon}\l g_{\varepsilon}^{+}-g_{\varepsilon}^{-},v\r_{L^{2}_{v}}\rightarrow j, w_{\varepsilon}=\frac{1}{\varepsilon}\l g_{\varepsilon}^{+}-g_{\varepsilon}^{-},\frac{|v|^{2}}{2}-\frac{3}{2}\r_{L^{2}_{v}}\rightarrow w,
\end{aligned}
\end{equation}
weakly-$\star$ for $t \geq 0$,  weakly in $H^{N}_x$ and strongly in $H^{N-1}_{x}$ as $\varepsilon \rightarrow 0$.

\subsection{Convergences to limit equations}

In this subsection, my goal is to deduce the two-fluid incompressible NSFP equations \eqref{NSFP} with Ohm's law from the local conservation laws \eqref{Consevtn-Law-rho-u-theta-phi}, \eqref{Consevtn-Law-n-j-w} and the convergence obtained in the previous subsection.

\subsubsection{Incompressibility and Boussinesq relation}

The first equation of \eqref{Consevtn-Law-rho-u-theta-phi} and the uniform bound \eqref{Unf-Bnd-3} yield
\begin{equation}\no
  \begin{aligned}
    \nabla_x \cdot u_\varepsilon = - \varepsilon \partial_t \rho_\varepsilon\rightarrow 0
  \end{aligned}
\end{equation}
in the sense of distribution as $\varepsilon \rightarrow 0$ and combine with the convergence \eqref{Con-rho-u-theta-n} that
\begin{equation}\label{Incompressibility}
  \begin{aligned}
    \nabla_x \cdot u= 0 \,.
  \end{aligned}
\end{equation}

The second $u_\varepsilon$-equation of \eqref{Consevtn-Law-rho-u-theta-phi} yields
\begin{equation}\no
  \begin{aligned}
    \nabla_{x}(\rho_{\varepsilon}+\theta_{\varepsilon})=\frac{\varepsilon}{2}n_{\varepsilon}\nabla_{x}\phi_{\varepsilon}-\varepsilon\partial_{t}u_{\varepsilon}-\nabla_{x}\cdot\l \hat{A},\mathcal{L}_{1}(I-P_{1})(\frac{g_{\varepsilon}^{+}+g_{\varepsilon}^{-}}{2})\r_{L^{2}_{v}}.
  \end{aligned}
\end{equation}

The bound \eqref{Unf-Bnd-1} and \eqref{Unf-Bnd-3} show that
\begin{equation*}
  \begin{aligned}
    \|\frac{\varepsilon}{2}n_{\varepsilon}\nabla_{x}\phi_{\varepsilon} \|_{H^N_{x}} \lesssim\varepsilon \|n_{\varepsilon}\|_{H^{N}_{x}}\|\nabla_{x}\phi_{\varepsilon}\|_{H^{N}_{x}}\lesssim\varepsilon,
  \end{aligned}
\end{equation*}
which means that
\begin{equation}\no
  \begin{aligned}
    \frac{\varepsilon}{2}n_{\varepsilon}\nabla_{x}\phi_{\varepsilon}\rightarrow 0,
  \end{aligned}
\end{equation}
strongly in $L^\infty (\R^+ ; H^N_x)$ as $\varepsilon \rightarrow 0$.  The bound \eqref{Unf-Bnd-3} imply that $-\varepsilon\partial_{t} u_{\varepsilon}\rightarrow 0$ in the sense of distribution as $\varepsilon\rightarrow 0$.  Moreover,  the uniform energy dissipation bound \eqref{Unf-Bnd-2} and \eqref{nu} imply that
\begin{equation}\no
  \begin{aligned}
    \int_{0}^{T} \| \nabla_x \cdot \l \widehat{A} , \mathcal{L}_{1} (I-P_{1})(\frac{g_\varepsilon^{+}+g_{\varepsilon}^{-}}{2}) \r_{L^2_v} \|^2_{H^{N-1}_{x}} \d t \lesssim \int_{0}^{T}\| (I-P_{1})(\frac{g_\varepsilon^{+}+g_{\varepsilon}^{-}}{2}) \|^2_{H^{N}_{x,v} (\nu)} \d t \lesssim \varepsilon^{2} ,
  \end{aligned}
\end{equation}
for any given $T>0$. Then, we have
\begin{equation}\no
  \begin{aligned}
    \nabla_x \cdot \l \widehat{A} , \mathcal{L}_{1} (I-P_{1})(\frac{g_\varepsilon^{+}+g_{\varepsilon}^{-}}{2}) \r_{L^2_v} \rightarrow 0,
  \end{aligned}
\end{equation}
strongly in $L^2 (\R^+; H^{N-1}_x)$ as $\varepsilon \rightarrow 0$. In summary, we have show that
\begin{equation*}
  \begin{aligned}
    \nabla_x (\rho_\varepsilon+ \theta_\varepsilon) \rightarrow 0
  \end{aligned}
\end{equation*}
in the sense of distribution as $\varepsilon \rightarrow 0$, then we combine with the convergences \eqref{Unf-Bnd-1} and \eqref{Unf-Bnd-3}, yield the Boussinesq relation that
\begin{equation}\no
  \begin{aligned}
    \nabla_{x} (\rho+\theta) = 0,
  \end{aligned}
\end{equation}
which imply that
\begin{equation}\label{Boussinesq-Reltn}
\rho+\theta=0
\end{equation}
since $x\in\R^{3}$ and $\rho,\theta\in L^{\infty}(\R^{+},L^{2}(\R^{3}))$.

\subsubsection{Convergences of $\frac{3}{2}\theta_{\varepsilon}- \rho_{\varepsilon}$ and $\mathcal{P} u_\varepsilon$}

Before given the prove of convergence, we introduce the important lemma Aubin-Lions-Simon Lemma, a fundamental result of compactness in the study of nonlinear evolution problems(\cite{Boyer-Fabrie-2013-BOOK}).

\begin{lemma}[Aubin-Lions-Simon Theorem]\label{Lmm-Aubin-Lions-Simon}
	Suppose $B_0 \subset B_1 \subset B_2$ are three Banach spaces, the embedding of $B_1$ in $B_2$ is continuous and that the embedding of $B_0$ in $B_1$ is compact. Let $p$, $r$ be such that $1 \leq p, r \leq + \infty$. For $T > 0$, we define
	\begin{equation*}
	  \begin{aligned}
	    E_{p,r} = \big\{ u \in L^p (0,T; B_0), \ \partial_t u \in L^r (0,T; B_2) \big\} \,.
	  \end{aligned}
	\end{equation*}
Then
	\begin{enumerate}
		\item If $p < + \infty$, the embedding of $E_{p,r}$ in $L^p (0,T; B_1)$ is compact.
		
		\item If $p = + \infty$ and $r > 1$, the embedding of $E_{p,r}$ in $C(0,T; B_1)$ is compact.
	\end{enumerate}
\end{lemma}

Next, we will give the proof of the convergence of $\frac{3}{2} \theta_{\varepsilon}- \rho_{\varepsilon}$. The third $\theta_{\varepsilon}$-equation of \eqref{Consevtn-Law-rho-u-theta-phi} multiplied by $ \frac{3}{2} $ minus one  times of the first $\rho_{\varepsilon}$-equation yields
\begin{equation}\label{Equ-theta-rho}
  \begin{aligned}
    \partial_{t}(\frac{3}{2} \theta_{\varepsilon} - \rho_{\varepsilon})&=\frac{3\varepsilon}{4}j_{\varepsilon}\nabla_{x}\phi_{\varepsilon}-\nabla_{x}\cdot\l\hat{B},\frac{1}{\varepsilon}\mathcal{L}_{1}\big(\frac{g_{\varepsilon}^{+}+g_{\varepsilon}^{-}}{2}\big)\r_{L^{2}_{v}}.
  \end{aligned}
\end{equation}

For the term $F_{1}=\frac{3\varepsilon}{4}j_{\varepsilon}\cdot\nabla_{x}\phi_{\varepsilon}$, the bound \eqref{j and w} and \eqref{Unf-Bnd-1} show that
\begin{equation}\no
\int_{0}^{T}\|F_{1}\|_{H^{N}_{x}}^{2}d t\lesssim\varepsilon\sup_{t\geq 0}\|\nabla_{x}\phi_{\varepsilon}\|_{H^{N}_{x}}^{2}\int_{0}^{T}\|\nabla_{x}\phi_{\varepsilon}\|^{2}_{H^{N}_{x}}d t\lesssim\varepsilon,
\end{equation}
for any given $T>0$. For the term $F_{2}=-\nabla_x \cdot \l \widehat{B} ,\frac{1}{\varepsilon}\mathcal{L}_{1} (I-P_{1})(\frac{ g_{\varepsilon}^{+}+g_{\varepsilon}^{-}}{2})\r_{L^2_v} $,  the bound \eqref{Unf-Bnd-2} show that
\begin{equation}\no
\int_{0}^{T}\|F_{2}\|_{H^{N-1}_{x}}^{2}dt\lesssim\frac{1}{\varepsilon^{2}}\int_{0}^{T}\|(I-P_{1})(\frac{g^{+}_{\varepsilon}+g_{\varepsilon}^{-}}{2})\|^{2}_{H^{N}_{x}L^{2}_{v}}dt\lesssim\frac{1}{\varepsilon^{2}}\int_{0}^{T}\|(I-P_{1})(\frac{g^{+}_{\varepsilon}+g^{+}_{\varepsilon}}{2})\|^{2}_{H^{N}_{x,v}(\nu)}d t\lesssim1,
\end{equation}
 for any given $T>0$. In summary, we gain
\begin{equation}\label{Bnd-rho-theta-t}
  \begin{aligned}
    \| \partial_{t} (\frac{3}{2}\theta_{\varepsilon} - \rho_{\varepsilon} ) \|_{L^{2} (\R^{+}; H^{N-1}_{x})} \lesssim 1,
  \end{aligned}
\end{equation}
for all $0 < \varepsilon \leq 1$. It is easily derived from the uniform bound \eqref{Unf-Bnd-3} that
\begin{equation}\label{Bnd-rho-theta}
  \begin{aligned}
    \| \frac{3}{2} \theta_{\varepsilon}- \rho_{\varepsilon} \|_{L^\infty (\R^+; H^N_x)} \lesssim \|\theta_{\varepsilon}\|_{L^\infty (\R^+; H^N_x)} +\|\rho_{\varepsilon}\|_{L^\infty (\R^+; H^N_x)}\lesssim1,
  \end{aligned}
\end{equation}
for all $0 < \varepsilon \leq 1$. Then, from Aubin-Lions-Simon Theorem in Lemma \ref{Lmm-Aubin-Lions-Simon}, the bounds \eqref{Bnd-rho-theta-t} and \eqref{Bnd-rho-theta} imply exist $\widetilde{\theta} (t,x) \in L^\infty (\R^+; H^N_x) \cap C (\R^+ ; H^{N-1}_x)$ such that
\begin{equation}\no
  \begin{aligned}
    \frac{3}{2} \theta_{\varepsilon}-  \rho_{\varepsilon} \rightarrow \widetilde{\theta},
  \end{aligned}
\end{equation}
strongly in $C(\R^+; H^{N-1}_x)$ as $\varepsilon \rightarrow 0$. Combining with the convergences \eqref{Con-rho-u-theta-n}, we know that $ \widetilde{\theta} = \frac{3}{2} \theta - \rho \in L^\infty (\R^+; H^N_x) \cap C (\R^+ ; H^{N-1}_x)  $. Consequently, we have
\begin{equation}\label{Convgnc-rho-theta}
  \begin{aligned}
    \frac{3}{2} \theta_{\varepsilon} - \rho_{\varepsilon} \rightarrow \frac{3}{2} \theta - \rho
  \end{aligned}
\end{equation}
strongly in $C(\R^+; H^{N-1}_x)$ as $\varepsilon \rightarrow 0$.

Next we will prove the convergence of $\mathcal{P} u_{\varepsilon}$, where $\mathcal{P}$ is the Leray projection on $\R^3$. Employing $\mathcal{P}$ on the second $u_\varepsilon$-equation of \eqref{Consevtn-Law-rho-u-theta-phi} yields
\begin{equation}\label{Equ-Pu}
  \begin{aligned}
    \partial_{t} \mathcal{P} u_{\varepsilon} + \mathcal{P} \nabla_x \cdot \l \widehat{A} , \frac{1}{\varepsilon} \mathcal{L}_{1} (I-P_{1})(\frac{g_{\varepsilon}^{+}+g_{\varepsilon}^{-}}{2})\r_{L^2_v}=\frac{1}{2}\mathcal{P}(n_{\varepsilon}\nabla_{x}\phi_{\varepsilon}).
  \end{aligned}
\end{equation}

The bound $\| A \|_{L^2_v} \lesssim 1$ and $\|\mathcal{P}f\|_{L^{2}_{x}}\leq \|f\|_{L^{2}_{x}}$, \eqref{nu}, \eqref{Unf-Bnd-1}, \eqref{Unf-Bnd-2} and \eqref{Unf-Bnd-3} implies that
\begin{equation}\label{Bnd-Pu-t}
  \begin{aligned}
\| \partial_{t} \mathcal{P} u_{\varepsilon}\|_{L^2 (0, T ; H^{N-1}_x)} &\lesssim  \frac{1}{\varepsilon} \| (I-P_{1}) (\frac{g_{\varepsilon}^{+}+g_{\varepsilon}^{-}}{2}) \|_{H^N_{x,v} (\nu)} + \| n_\varepsilon \|_{L^\infty (\R^+; H^N_x)} \| \nabla_x \phi_\varepsilon \|_{L^\infty (\R^+ ; H^N_x)} T\lesssim 1,
  \end{aligned}
\end{equation}
for any $T > 0$ and $0 < \varepsilon \leq 1$. Moreover, the bound \eqref{Unf-Bnd-3} show that for any given $T > 0$ and $0 < \varepsilon \leq 1$
\begin{equation}\label{Bnd-Pu}
  \begin{aligned}
    \| \mathcal{P} u_{\varepsilon}\|_{L^\infty (0,T; H^{N}_{x})} \lesssim 1 \,.
  \end{aligned}
\end{equation}

Then, Aubin-Lions-Simon Theorem in Lemma \ref{Lmm-Aubin-Lions-Simon}, the bounds \eqref{Bnd-Pu-t}, \eqref{Bnd-Pu} imply that there exists a $\widetilde{u}\in L^\infty (\R^+ ; H^N_x) \cap C (\R^+ ; H^{N-1}_x)$ such that
\begin{equation*}
  \begin{aligned}
    \mathcal{P} u_{\varepsilon}\rightarrow \widetilde{u},
  \end{aligned}
\end{equation*}
strongly in $ C (0, T; H^{N-1}_x) $ for any given $T > 0$ as $\varepsilon \rightarrow 0$. The convergences \eqref{Con-rho-u-theta-n} and incompressibility \eqref{Incompressibility}yield that $ \widetilde{u}= \mathcal{P}u= u\in L^\infty (\R^+; H^N_x) \cap C (\R^+; H^{N-1}_x) $. In summary, we have
\begin{equation}\label{Convgnc-Pu}
  \begin{aligned}
    \mathcal{P} u_{\varepsilon} \rightarrow u,
  \end{aligned}
\end{equation}
strongly in $ C (\R^+; H^{N-1}_x) $ as $\varepsilon\rightarrow 0$. Consequently, we have
\begin{equation}\label{Convgnc-Pu-perp}
  \begin{aligned}
    \mathcal{P}^\perp u_{\varepsilon}\rightarrow 0,
  \end{aligned}
\end{equation}
weakly-$\star$ in $t \geq 0$, weakly in $H^N_x$ and strongly in $H^{N-1}_x$ as $\varepsilon \rightarrow 0$, where $\mathcal{P}^\perp$ is the orthogonal projection of $\mathcal{P}$ in $L^2_x$. Recalling the equation of $n_{\varepsilon}$ in \eqref{Consevtn-Law-n-j-w},\eqref{Con-rho-u-theta-n} and the bound in \eqref{j and w}, it's easy to deduce that $n\in L^{\infty}(\R^{+}, H^{N}_{x})\cap C (0, T; H^{N-1}_x)$ for any given $T>0$ and
\begin{equation}\label{Con-n}
n_{\varepsilon} \rightarrow n,
\end{equation}
strongly in $ C (\R^+; H^{N-1}_x)$ as $\varepsilon\rightarrow 0$. As for the term $\phi_{\varepsilon}$, we can obtain that
\begin{equation}\no
\Delta_{x}\partial_{t}\phi_{\varepsilon}=-\nabla_{x}\cdot j_{\varepsilon},
\end{equation}
in the sense of distribution about $t$, where \eqref{Consevtn-Law-rho-u-theta-phi} and \eqref{Consevtn-Law-n-j-w} are used. Employing the same argument above, we can deduce by the bound of \eqref{Unf-Bnd-1},\eqref{Unf-Bnd-2} and the  convergence of \eqref{Convgc-g-phi} that $\nabla_{x}\phi\in L^{\infty}(\R^{+},H^{N}_{x})\cap C(0,T,H^{N-1}_{x})$ for any given $T>0$. We also have
\begin{equation}\label{Con-phi}
\Delta_{x}\phi_{\varepsilon}\rightarrow \Delta_{x}\phi,
\end{equation}
strongly in $ C (\R^+; H^{N-2}_x)$ as $\varepsilon\rightarrow 0$.

\subsubsection{Equations of $\rho$, $u$, $n$, $j$ and $w$}

In this subsection, we will derive the two-fluid incompressible NSFP system with Ohm's law from \eqref{Equ-Pu}, \eqref{Equ-theta-rho}, \eqref{Consevtn-Law-n-j-w} and the last Poisson equation of \eqref{Consevtn-Law-rho-u-theta-phi}.  The following term will be calculated first
\begin{equation*}
  \begin{aligned}
    \l \widehat{\Theta} , \frac{1}{\varepsilon}\mathcal{ L}_{1} (I-P_{1})(\frac{g_{\varepsilon}^{+}+g_{\varepsilon}^{-}}{2})\r_{L^2_v} \,,
  \end{aligned}
\end{equation*}
where $\Theta = A$ or $B$. According to reference \cite{BGL1,Arsenio-SaintRaymond-2016}, we have
\begin{equation}\label{A-Dissipation}
  \begin{aligned}
    \l \widehat{A} , \frac{1}{\varepsilon}\mathcal{ L}_{1} (I-P_{1})(\frac{g_{\varepsilon}^{+}+g_{\varepsilon}^{-}}{2})\r_{L^2_v} = u_\varepsilon\otimes u_\varepsilon - \frac{|u_\varepsilon|^2}{3} I - \mu \Sigma (u_\varepsilon) + R_{\varepsilon} \,
  \end{aligned}
\end{equation}
and
\begin{equation}\label{B-Dissipation}
  \begin{aligned}
    \l \widehat{B} , \frac{1}{\varepsilon}\mathcal{ L}_{1} (I-P_{1})(\frac{g_{\varepsilon}^{+}+g_{\varepsilon}^{-}}{2})\r_{L^2_v} = \frac{5}{2} u_\varepsilon \theta_\varepsilon - \frac{5}{2} \kappa \nabla_x \theta_\varepsilon+ R_{\varepsilon} \,,
  \end{aligned}
\end{equation}
where $\Sigma (u_\varepsilon) = \nabla_x u_\varepsilon+ (\nabla_x u_\varepsilon)^\top - \frac{2}{3} \nabla_x \cdot u_\varepsilon I$, $\mu$, $\kappa$ are given in \eqref{mu-kappa}, and $R_{\varepsilon}$ are of the form
\begin{equation}\label{R-eps-Xi}
  \begin{aligned}
R_{\varepsilon}&=2Q\big(P_{1}(\frac{g_{\varepsilon}^{+}+g_{\varepsilon}^{-}}{2})\big),(I-P_{1})(\frac{g_{\varepsilon}^{+}+g_{\varepsilon}^{-}}{2})\big)+2Q\big((I-P_{1})(\frac{g_{\varepsilon}^{+}+g_{\varepsilon}^{-}}{2}),P_{1}(\frac{g_{\varepsilon}^{+}+g_{\varepsilon}^{-}}{2})\big)\\
    &+2Q\big((I-P_{1})(\frac{g_{\varepsilon}^{+}+g_{\varepsilon}^{-}}{2})\big),(I-P_{1})(\frac{g_{\varepsilon}^{+}+g_{\varepsilon}^{-}}{2})+\varepsilon \frac{g_{\varepsilon}^{+}-g_{\varepsilon}^{-}}{2}v\cdot\nabla_{x}\phi_{\varepsilon}-\varepsilon\nabla_{v}\frac{g_{\varepsilon}^{+}-g_{\varepsilon}^{-}}{2}\cdot\nabla_{x}\phi_{\varepsilon}\\
    &-v\cdot\nabla_{x}(I-P_{1})(\frac{g_{\varepsilon}^{+}+g_{\varepsilon}^{-}}{2})-\varepsilon\partial_{t}\frac{g_{\varepsilon}^{+}+g_{\varepsilon}^{-}}{2}.
  \end{aligned}
\end{equation}

The relation \eqref{A-Dissipation}, decomposition $u_\varepsilon = \mathcal{P} u_\varepsilon + \mathcal{P}^\perp u_\varepsilon$ and\eqref{Equ-Pu} yield that
\begin{equation}\label{Evolution-Pu}
  \begin{aligned}
    \partial_{t} \mathcal{P} u_{\varepsilon} + \mathcal{P} \nabla_{x} \cdot ( \mathcal{P} u_\varepsilon &\otimes \mathcal{P} u_{\varepsilon}) - \mu \Delta_{x} \mathcal{P} u_{\varepsilon}= \frac{1}{2} \mathcal{P}(n_{\varepsilon}\cdot\nabla_{x}\phi_{\varepsilon})+ R_{\varepsilon, u} \,,
  \end{aligned}
\end{equation}
where
\begin{equation}\label{R-eps-u}
  \begin{aligned}
    &R_{\varepsilon,u}=-\mathcal{P}\nabla_{x}\cdot R_{\varepsilon}-\mathcal{P}\nabla_{x}\cdot(\mathcal{P}^{\perp}\otimes\mathcal{P}+\mathcal{P}\otimes\mathcal{P}^{\perp}+\mathcal{P}^{\perp}\otimes\mathcal{P}^{\perp}).
  \end{aligned}
\end{equation}

The relation \eqref{B-Dissipation}, \eqref{Boussinesq-Reltn}, $u_{\varepsilon} = \mathcal{P} u_{\varepsilon} + \mathcal{P}^\perp u_{\varepsilon}$ and \eqref{Equ-theta-rho} yield that
\begin{equation}\label{Evolution-theta-rho}
  \begin{aligned}
    \partial_{t}\theta_{\varepsilon}+\mathcal{P}u_{\varepsilon}\cdot\nabla_{x}\theta_{\varepsilon}-\kappa\Delta_{x}\theta_{\varepsilon}=R_{\varepsilon,\theta}\,,
  \end{aligned}
\end{equation}
where
\begin{equation}\label{R-eps-theta}
  \begin{aligned}
R_{\varepsilon, \theta}=\frac{3\varepsilon}{10}j_{\varepsilon}\cdot\nabla_{x}\phi_{\varepsilon}-\frac{2}{5}\nabla_{x}\cdot R_{\varepsilon}-\theta_{\varepsilon}\nabla_{x}\cdot u_{\varepsilon}-\mathcal{P}u_{\varepsilon}\cdot\nabla_{x}\theta_{\varepsilon}.
  \end{aligned}
\end{equation}

Then we can take the limit from \eqref{Evolution-Pu} to gain the $u$-equation of \eqref{NSFP}. For any given $T > 0$, let a vector-valued text function $\psi (t,x) \in C^1 (0,T; C_0^\infty (\R^3)) $ with $\nabla_x \cdot \psi = 0$, $\psi (0,x) = \psi_0 (x) \in C_0^\infty (\R^3)$ and $\psi (t,x) = 0$ for $t \geq T'$, where $T' < T$. Then we have
\begin{equation}\no
  \begin{aligned}
     \int_{0}^{T} \int_{\R^3} \partial_t \mathcal{P} u_{\varepsilon} \cdot \psi (t,x) d x d t
    =&  - \int_{\R^3} \mathcal{P} u_\varepsilon (0,x) \psi (0,x) d x - \int_0^T \int_{\R^3} \mathcal{P} u_{\varepsilon} \partial_{t} \psi (t,x) d x d t \\
    = & - \int_{\R^3} \mathcal{P}\l \frac{g_{\varepsilon, 0}^{+}+g_{\varepsilon,0}^{-}}{2} , v \r_{L^2_v} \psi_0 (x) d x - \int_0^T \int_{\R^3} \mathcal{P} u_{\varepsilon} \partial_{t} \psi (t,x) d x d t \,.
  \end{aligned}
\end{equation}

The initial conditions in Theorem \ref{Thm-Limit} and the convergence \eqref{Convgnc-Pu} yield that
\begin{equation}\no
  \begin{aligned}
    \int_{\R^3} \mathcal{P}\l\frac{g_{\varepsilon, 0}^{+}+g_{\varepsilon,0}^{-}}{2} , v \r_{L^2_v} \psi_0 (x) d x \rightarrow \int_{\R^3} \mathcal{P} \l \frac{g_{0}^{+}+g_{0}^{-}}{2}, v \r_{L^2_v} \psi_0 (x) d x = \int_{\R^3} \mathcal{P} u_0 (x)  \psi_0 (x) d x,
  \end{aligned}
\end{equation}
and
\begin{equation}\no
  \begin{aligned}
    \int_0^T \int_{\R^3} \mathcal{P} u_\varepsilon\cdot \partial_t \psi (t,x) d x d t \rightarrow \int_0^T \int_{\R^3} u \cdot \partial_t \psi (t,x) d x d t,
  \end{aligned}
\end{equation}
as $\varepsilon \rightarrow 0$. Then, we have
\begin{equation}\label{Converge-Pu-1}
  \begin{aligned}
    \int_0^T \int_{\R^3} \partial_t \mathcal{P} u_\varepsilon \psi (t,x) d x d t \rightarrow - \int_{\R^3} \mathcal{P} u_0 (x) \psi_0 (x) d x - \int_0^T \int_{\R^3} u \partial_t \psi (t,x) d x d t,
  \end{aligned}
\end{equation}
as $\varepsilon \rightarrow 0$. The strong convergence \eqref{Convgnc-Pu} implies that
\begin{equation}\label{Converge-Pu-2}
  \begin{aligned}
    & \mathcal{P} \nabla_x \cdot ( \mathcal{P} u_\varepsilon \otimes \mathcal{P} u_\varepsilon ) \rightarrow \mathcal{P} \nabla_x \cdot (u \otimes u) \quad \textrm{strongly in } \quad C(\R^+; H^{N-2}_x) \,, \\
    & \mu \Delta_x \mathcal{P} u_\varepsilon \rightarrow \mu \Delta_x u\quad \textrm{strongly in } \quad C(\R^+; H^{N-3}_x),
  \end{aligned}
\end{equation}
as $\varepsilon \rightarrow 0$. From the bound \eqref{Unf-Bnd-3}, strong convergence \eqref{Con-n} and \eqref{Con-phi}, we deduce that
\begin{equation}\label{Converge-Pu-3}
  \begin{aligned}
    \frac{1}{2}\mathcal{P}(n_{\varepsilon}\nabla_{x}\phi_{\varepsilon})\rightarrow\frac{1}{2}\mathcal{P}(n\nabla_{x}\phi)
  \end{aligned}
\end{equation}
weakly-$\star$ for $t \geq 0$, strongly in $H^{N-1}_x$ as $\varepsilon \rightarrow 0$.

Next, we will prove
\begin{equation}\label{Converge-Pu-4}
   \begin{aligned}
     R_{\varepsilon, u} \rightarrow 0
   \end{aligned}
\end{equation}
in the sense of distribution as $\varepsilon \rightarrow 0$, where $R_{\varepsilon, u}$ is defined in \eqref{R-eps-u}. In fact, by applying the convergences \eqref{Convgnc-Pu} and \eqref{Convgnc-Pu-perp}, we obtain
\begin{equation}
  \begin{aligned}
    \mathcal{P} \nabla_x \cdot ( \mathcal{P} u_\varepsilon\otimes \mathcal{P}^\perp u_\varepsilon + \mathcal{P}^\perp u_\varepsilon \otimes \mathcal{P} u_\varepsilon + \mathcal{P}^\perp u_\varepsilon \otimes \mathcal{P}^\perp u_\varepsilon) \rightarrow 0
  \end{aligned}
\end{equation}
weakly-$\star$ in $t \geq 0$ and strongly in $H^{N-1}_x$  as $\varepsilon\rightarrow 0$. Finally, the definition of $R_{\varepsilon}$ in \eqref{R-eps-Xi} and the convergences \eqref{Unf-Bnd-1}, \eqref{Unf-Bnd-2}, \eqref{Unf-Bnd-3}, \eqref{Convgnc-Pu}, \eqref{Convgnc-Pu-perp} tell us that
\begin{equation}\label{Div-R-eps-Xi-to-0}
  \begin{aligned}
    \nabla_x \cdot R_{\varepsilon} \rightarrow 0
  \end{aligned}
\end{equation}
in the sense of distribution as $\varepsilon \rightarrow 0$. In summary, we prove the convergence \eqref{Converge-Pu-4}. The convergences \eqref{Converge-Pu-1}, \eqref{Converge-Pu-2}, \eqref{Converge-Pu-3}, \eqref{Converge-Pu-4} and the incompressibility \eqref{Incompressibility} that $u \in L^\infty (\R^+; H^N_x) \cap C (\R^+; H^{N-1}_x)$ yield that
\begin{equation}
  \begin{cases}
    \partial_t u + \mathcal{P} \nabla_x \cdot (u\otimes u) - \mu \Delta_x u =\frac{1}{2} \mathcal{P}(n\nabla_{x}\phi) \,, \\
    \nabla_x \cdot u = 0 \,,
  \end{cases}
\end{equation}
with initial data
\begin{equation}
  \begin{aligned}
    u (0, x) = \mathcal{P} u_0 (x) \,.
  \end{aligned}
\end{equation}

Next, we take the limit from \eqref{Evolution-theta-rho} to the third $\theta$-equation in \eqref{NSFP} as $\varepsilon \rightarrow 0$. For any given $T > 0$, taking $\zeta (t,x)$ be a test function satisfying $\zeta (t,x) \in C^1 (0,T; C_0^\infty (\R^3))$ with $\zeta (0,x) = \zeta_0 (x) \in C_0^\infty (\R^3)$ and $\zeta (t,x) = 0$ for $t \geq T'$, where $T' < T$. The initial conditions in Theorem \ref{Thm-Limit} and the strong convergence \eqref{Convgnc-rho-theta} show that
\begin{equation}\label{Converge-theta-1}
  \begin{aligned}
    & \int_0^T \int_{\R^3} \partial_t \theta_{\varepsilon}(t,x) \zeta(t,x) d x d t \\
    &=  - \int_{\R^3} \l \frac{g^{+}_{\varepsilon , 0}(x,v)+g^{-}_{\varepsilon,0}(x,v)}{2} , \frac{|v|^2}{3} - 1 \r_{L^2_v} \zeta_0 (x) d x - \int_0^T \int_{\R^3} \theta_{\varepsilon}\partial_t \zeta (t,x) d x d t \\
    &\rightarrow - \int_{\R^3} \l \frac{g^{+}_{\varepsilon , 0}(x,v)+g^{-}_{\varepsilon,0}(x,v)}{2} , \frac{|v|^2}{3} - 1 \r_{L^2_v} \zeta_0 (x) d x - \int_0^T \int_{\R^3}\theta(t,x)\partial_t \zeta (t,x) d x d t \\
    &= - \int_{\R^3} \theta_{0} (x) \zeta_0 (x) d x - \int_0^T \int_{\R^3} \theta (t,x) \partial_t \zeta (t,x) d x d t
  \end{aligned}
\end{equation}
as $\varepsilon \rightarrow 0$. The strong convergences \eqref{Convgnc-Pu}, \eqref{Convgnc-rho-theta} and \eqref{Boussinesq-Reltn} yield that
\begin{equation}\label{Converge-theta-2}
  \begin{aligned}
    & \mathcal{P}u_{\varepsilon}\cdot\nabla_{x}\theta_{\varepsilon}\rightarrow  u\cdot\nabla_{x}\theta \quad \textrm{strongly in } C(\R^{+}; H^{N-2}_{x}) \,, \\
    & \kappa \Delta_{x}\theta_{\varepsilon} \rightarrow \kappa \Delta_{x}\theta \quad \textrm{strongly in } C(\R^+; H^{N-3}_x)
  \end{aligned}
\end{equation}
as $\varepsilon \rightarrow 0$.

Next, we will prove
\begin{equation}\label{Converge-theta-3}
  \begin{aligned}
    R_{\varepsilon, \theta} \rightarrow 0
  \end{aligned}
\end{equation}
in the sense of distribution as $\varepsilon \rightarrow 0$, where $R_{\varepsilon, \theta}$ is given in \eqref{R-eps-theta}. In fact, the convergences \eqref{Unf-Bnd-1}, \eqref{Unf-Bnd-3}, \eqref{Con-j-w} and \eqref{Convgnc-Pu-perp} show that almost any $t \geq 0$
\begin{equation}\label{Limit-theta-1}
  \begin{aligned}
    & \frac{3\varepsilon}{10}j_{\varepsilon}\cdot\nabla_{x}\phi_{\varepsilon} \rightarrow 0 \quad \textrm{strongly in } H^{N-1}_x \,,  \theta_{\varepsilon}(\nabla_{x}\cdot u_{\varepsilon}) \rightarrow 0 ,\quad \textrm{strongly in } H^{N-2}_x
  \end{aligned}
\end{equation}
as $\varepsilon \rightarrow 0$. The convergence \eqref{Div-R-eps-Xi-to-0} tells us
\begin{equation}\label{Limit-theta-2}
  \begin{aligned}
    -\frac{2}{5} \nabla_x \cdot R_{\varepsilon} \rightarrow 0
  \end{aligned}
\end{equation}
in the sense of distribution as $\varepsilon \rightarrow 0$. In summary, the limits \eqref{Limit-theta-1} and \eqref{Limit-theta-2} show the convergence \eqref{Converge-theta-3}. And the limits \eqref{Converge-theta-1}, \eqref{Converge-theta-2} and \eqref{Converge-theta-3} implies that
\begin{equation}
  \begin{aligned}
    \partial_{t}\theta+u\cdot\nabla_{x}\theta-\kappa\Delta_{x}\theta=0,
  \end{aligned}
\end{equation}
with the initial data
\begin{equation}
  \begin{aligned}
    \theta(0,x) =\theta_{0}(x) \,.
  \end{aligned}
\end{equation}

Finally, the convergence \eqref{Convgc-g-phi} and \eqref{Con-rho-u-theta-n} tell us that
\begin{equation}\label{Limit-Delta-phi}
  \begin{aligned}
   & \Delta_x \phi_\varepsilon \rightarrow \Delta_x \phi,  n_{\varepsilon}  \rightarrow n,
  \end{aligned}
\end{equation}
weakly-$\star$ for $t \geq 0$ and weakly in $H^{N-1}_x$ as $\varepsilon \rightarrow 0$. Then, the convergence \eqref{Limit-Delta-phi} and equation \eqref{Consevtn-Law-rho-u-theta-phi} show that
\begin{equation}\no
\Delta_{x}\phi=n.
\end{equation}

Next, we deduce the equation of $n$, $j$ and $w$. Recalling the equation of $n_{\varepsilon}$ in \eqref{Consevtn-Law-n-j-w},  For any given $T>0$, taking $\eta(t,x)$ be a test function satisfying $\eta(t,x)\in C^{1}(0,T; C_{0}^{\infty}(\R^{3}))$ with $\eta(0,x)=\eta_{0}(x)\in C_{0}^{\infty}(\R^{3})$ and $\eta(t,x)=0$ for $t\geq T^{'}$, where $T'<T$, The initial conditions in Theorem \ref{Thm-Limit} and the strong convergence \eqref{Con-n} show that
\begin{equation}\label{Converge-n-1}
  \begin{aligned}
    & \int_0^T \int_{\R^3} \partial_t n_{\varepsilon}(t,x) \eta(t,x) d x d t \\
    &=  - \int_{\R^3} \l g_{\varepsilon,0}^{+}(x,v)-g_{\varepsilon,0}^{-}(x,v), 1 \r_{L^2_v} \zeta_0 (x) d x - \int_0^T \int_{\R^3} n_{\varepsilon}\partial_t \eta (t,x) d x d t \\
    &\rightarrow - \int_{\R^3} \l g_{\varepsilon,0}^{+}(x,v)-g_{\varepsilon,0}^{-}(x,v), 1 \r_{L^2_v} \eta_0 (x) d x - \int_0^T \int_{\R^3}n(t,x)\partial_t \eta (t,x) d x d t \\
    &= - \int_{\R^3} n_{0} (x) \eta_0 (x) d x - \int_0^T \int_{\R^3} n (t,x) \partial_t \eta (t,x) d x d t
  \end{aligned}
\end{equation}
as $\varepsilon \rightarrow 0$. Then, the convergence \eqref{Con-j-w} yield that
\begin{equation}\label{Converge-n-2}
\nabla_{x}\cdot j_{\varepsilon} \rightarrow \nabla_{x}\cdot j
\end{equation}
weakly-$\star$ for $t\geq0$, weakly in $H^{N-1}_{x}$, strongly in $H^{N-2}_{x}$. Combining the convergence of \eqref{Converge-n-1} and \eqref{Converge-n-1}, we yield that
\begin{equation}\label{Limit-n}
\partial_{t}n+\nabla_{x}\cdot j=0,
\end{equation}
with initial data
\begin{equation}\no
n(0,x)=n_{0}(x)=\Delta_{x}\phi_{0}(x).
\end{equation}

Applying the convergence \eqref{Con-n}, \eqref{Convgnc-Pu} and \eqref{Con-phi}, we obtain
\begin{equation}\label{Converge-j-1}
  \begin{aligned}
    &n_{\varepsilon}\mathcal{P}u_{\varepsilon} \rightarrow n u \quad \textrm{strongly in } \quad C(\R^+; H^{N-2}_x) \,, \\
    & \sigma(\nabla_{x}\phi_{\varepsilon}-\frac{1}{2}\nabla_{x}n_{\varepsilon}) \rightarrow \sigma(\nabla_{x}\phi-\frac{1}{2}\nabla_{x}n)\quad \textrm{strongly in } \quad C(\R^+; H^{N-2}_x).
  \end{aligned}
\end{equation}

It's easy to deduce that
\begin{equation}\label{Converge-j-1}
j_{R,\varepsilon}\rightarrow 0
\end{equation}
in the sense of distribution as $\varepsilon\rightarrow 0$, where we make use of the bound \eqref{Unf-Bnd-1}, \eqref{Unf-Bnd-2} and Lemma \ref{a-b}. As result, we have
\begin{equation}\label{Limit-j}
j= n u+\sigma(\nabla_{x}\phi-\frac{1}{2}\nabla_{x}n),
\end{equation}
with initial data
\begin{equation}\no
j(0,x)=n_{0}(x)\mathcal{P}u_{0}(x)+\sigma(\nabla_{x}\phi_{0}(x)-\frac{1}{2}\nabla_{x}n_{0}(x))=\Delta_{x}\phi_{0}(x)\mathcal{P}u_{0}(x)+\sigma(\nabla_{x}\phi_{0}(x)-\frac{1}{2}\nabla_{x}\Delta_{x}\phi_{0}(x)).
\end{equation}

Employing the same argument above, we have
\begin{equation}\label{Limit-w}
w=n\theta,
\end{equation}
with initial data
\begin{equation}\no
w(0,x)=n_{0}(x)\theta_{0}(x)=\Delta_{x}\phi_{0}(x)\theta_{0}(x).
\end{equation}

Combining the all above convergence results, we have obtain that $(u,\theta,n,j,w,\nabla_{x}\phi) \in L^\infty (\R^+ ; H^N_x)$ with $ (u, \theta,n,j,w,\nabla_{x}\phi) \in C(\R^+; H^{N-1}_x) $ obey the following two-fluid incompressible NSFP equations with Ohm's law
\begin{equation*}
  \left\{
    \begin{array}{l}
       \partial_{t}u+u\cdot\nabla_{x}u+\nabla_{x}p=\mu\Delta_{x}u+\frac{1}{2}n\nabla_{x}\phi\,,\\
       \partial_{t}\theta+u\cdot\nabla_{x}\theta=\kappa\Delta_{x}\theta\,,\\
       \partial_{t}n+u\cdot\nabla_{x}n+\sigma n=\frac{\sigma}{2}\Delta_{x}n\,,\\
       j=n u+\sigma(\nabla_{x}\phi-\frac{1}{2}\nabla_{x}n)\,,\\
       \text{div}_{x}u=0, \rho+\theta=0, \Delta_{x}\phi=n,w=n\theta,
    \end{array}
  \right.
\end{equation*}
with initial data
\begin{equation}\no
  \begin{aligned}
    &u(0,x) = \mathcal{P}u_{0}(x) \,, \quad \theta(0,x)=\theta_{0}(x) \,,\quad \phi(0,x)=\phi_{0}(x)\,,\quad\rho(0,x)=-\theta_{0}(x)\,,\quad n(0,x)=\Delta_{x}\phi_{0}(x)\\
    &j(0,x)=\Delta_{x}\phi_{0}(x)\mathcal{P}u_{0}(x)+\sigma(\nabla_{x}\phi_{0}(x)-\frac{1}{2}\nabla_{x}\Delta_{x}\phi_{0}(x))\,,\quad w(0,x)=\Delta_{x}\phi_{0}(x)\theta_{0}(x).
  \end{aligned}
\end{equation}
Consequently, the proof of Theorem \ref{Thm-Limit} is completed.

\section*{Acknowledgment}

This work was supported by the grants from the National Natural Science Foundation of
China under contract No. 11971360 and No. 11731008.


\bigskip

\bibliography{reference}

\end{document}